\documentclass[11pt]{article}
\usepackage{amsmath, amssymb, amsfonts, amsthm, authblk, color, graphicx, enumitem, mathrsfs, indentfirst}
\usepackage{hyperref, cleveref}
\hypersetup{colorlinks=true, linkcolor=blue, filecolor=magenta, urlcolor=cyan}
\usepackage[textwidth=6in,textheight=9in]{geometry}
\allowdisplaybreaks

\usepackage{algorithm}
\usepackage{algorithmic}
\usepackage{subcaption}
\usepackage{chngcntr}
\counterwithin{table}{section}
\newtheorem{theorem}{Theorem}[section]

\newtheorem{example}{Example}[section]
\newtheorem{assumption}{Assumption}
\newtheorem{remark}{Remark}[section]

\numberwithin{equation}{section}

\newcommand{\keywords}[1]{\small\textbf{\textit{Keywords---}}#1}

\title{A randomized progressive iterative regularization method for data fitting problems}

\author{Dakang Cen}
\author{Wenlong Zhang\footnote{Corresponding author: zhangwl@sustech.edu.cn, supported by the National Natural Science Foundation of China under grant
numbers No.12371423 and No.12241104.}}
\author{Junbin Zhong}
\affil{Department of Mathematics, Southern University of Science and Technology, Shenzhen, 518055, China}

\begin{document}
\maketitle

\abstract{In this work, we investigate data fitting problems with random noises. A randomized progressive iterative regularization method is proposed. It works well for large-scale matrix computations and converges in expectation to the least-squares solution. Furthermore, we present an optimal estimation for the regularization parameter, which inspires the construction of self-consistent algorithms without prior information.  The numerical results confirm the theoretical analysis and show the performance in curve and surface fittings.
}


\section{Introduction}
Data fitting problems are among the classical topics in computational mathematics, with a wide range of applications in science and engineering, including computer-aided geometric design (CAGD), differential equations, approximation theory, and computer graphics\cite{LIN201840}.

Among various fitting approaches, the Progressive Iterative Approximation (PIA) method stands out as an intuitive and well-established technique. It was initially introduced by Qi et al. in 1975\cite{Qi1975} and de Boor in 1979\cite{deboor1979agee}, and then further studied by Lin et al. in 2005\cite{LIN2005575}. PIA method avoids solving large-scale linear systems, thereby reducing computational cost, and is capable of producing a sequence of approximating curves or surfaces. However, it requires the number of control points to be equal to the number of data points, which becomes impractical when the dataset is large. Several variants of PIA have been proposed to address different aspects of the method, such as the local PIA\cite{LIN2010322}, and weighted PIA\cite{LU2010129,ZHANG2018331}, all of which still require the control points to match the number of data points. A major breakthrough was made by Lin and Zhang in 2011\cite{LIN2011967}, who introduced the Extended PIA (EPIA), allowing the number of control points to be fewer than the number of data points. Additionally, EPIA is well-suited for parallel execution due to its computational independence. Lin and Deng\cite{DENG201432} later proposed a least-squares PIA (LSPIA) based on B-spline basis functions, which expanded the applicability of PIA in shape modeling. Furthermore, Rios and Jüttler\cite{RIOS2022113921} proved algebraically that LSPIA is equivalent to the gradient descent method. Many other PIA variants have also been proposed—see Ref\cite{EBRAHIMI20191,HUANG2020101931,LIU2020112389,CARNICER20102010,6035703,WU2025102439} for details, such as SLSPIA and MLSPIA. Both of them are global methods and need to update all control points in each iteration. Although local PIA method updates only part of the control points, it cannot handle least-squares fitting problems. In the case of surface fitting, all of these methods require the computation of the Kronecker product of two collocation matrices, which is computationally demanding. To address these limitations, Liu and Wu\cite{WU2024128669} proposed a novel Randomized Progressive Iterative Approximation (RPIA) method in 2024. This method updates a subset of control points based on randomized criteria to construct a sequence of fitting curves or surfaces and was theoretically shown to yield a least-squares fit to the given data points.
At the same time, in the existing methods for surface fitting, PDE-based deformation surfaces, bilinearly blended Coons patches, B-spline and tensor product smoothing methods have been extensively used for surface fitting and geometric modeling in three dimensions\cite{ZHU2024113436}. While these approaches are highly effective for representing complex geometric structures and achieving high-accuracy reconstruction, they typically rely on large-scale matrix computations and may be sensitive to noise or require careful tuning of prior parameters.

This paper is based on the RPIA method proposed by Liu and Wu\cite{WU2024128669} in 2024, with a focus on enhancing its robustness under noisy data conditions. Although Liu and Wu explained the noise tolerance of RPIA method through numerical experiments, they did not provide a theoretical justification. Considering that regularization is a widely recognized and effective approach for handling noise, it has become a fundamental tool in data fitting to address ill-posedness and improve solution stability\cite{MONTEGRANARIO2007583}. When data are affected by measurement errors, missing values, or perturbations, directly minimizing the fitting error often leads to overfitting or unstable solutions. Regularization addresses these issues by incorporating prior knowledge or additional constraints—such as smoothness, sparsity, or norm bounds—into the optimization process, thereby improving the conditioning of the problem and enhancing the robustness and interpretability of the solution. Consequently, regularization theory constitutes a crucial part of modern data fitting methodologies and plays a key role in reconstructing curves and surfaces from noisy data. Existing studies have demonstrated the effectiveness of various regularization techniques, including Tikhonov regularization\cite{Calvetti2025DistributedTR}, variational regularization\cite{Buccini2021AVN,Zhang2022ImpulseNI} and kernel-based methods\cite{Yin2020CurveFO}, in a wide range of fitting applications. So in summary, we propose the Randomized Progressive Iterative Regularization (RPIA) Method. The main contributions of this work are as follows:

\begin{enumerate}
\item Under a certain level of noise, we theoretically prove that the proposed randomized iterative regularization method converges to the least-squares solution under the appropriate conditions.
\item Given the level of noise and the size of the data, we provide a theoretical estimate of the optimal regularization parameter. 
\item Self-consistent algorithms without prior information are presented, which are significant in practical applications.
\end{enumerate}

We organize the remainder of this paper as follows. In Section 2, we briefly review the original RPIA method for curve and surface fitting. In Section 3, we introduce the  regularized RPIA Method for noisy curve and surface fitting. Section 4 presents a theoretical analysis of the convergence of the proposed method. In Section 5, we provide a stochastic estimate of the optimal regularization parameter. Furthermore, self-consistent algorithms without prior information are proposed. Subsequently, in Section 7, we validate the theoretical findings through several numerical examples. Finally, Section 8 concludes the paper with a summary of our main results.

\section{The RPIA method}
In this section, we introduce the original RPIA method for curve and surface fitting\cite{WU2024128669}.

\subsection{The case of curves}
We call $\left\{U_{i}\right\}_{i=0}^{t}$ a partition of $\left[n_{1}\right]$ if $U_{i} \cap U_{j}=\emptyset$ for $i \neq j$ and $\cup_{i=0}^{t} U_{i}=\left[n_{1}\right]$. Given the data point $\left\{\boldsymbol{q}_{j}\right\}_{j=0}^{m}$, the initial control point $\left\{\boldsymbol{p}_{i}^{(0)}\right\}_{i=0}^{n_{1}}$ and a blending basis sequence $\left\{\mu_i(x): i \in\left[n_1\right]\right\}$ defined on $[0,1]$, we first construct an initial curve.
$$
\mathcal{C}^{(0)}(x)=\sum_{i=0}^{n_{1}} \mu_{i}(x) \boldsymbol{p}_{i}^{(0)},~x \in\left[x_{0}, x_{m}\right].
$$
Correspondingly, the $k$-th curve is:
\begin{equation*}
\mathcal{C}^{(k)}(x)=\sum_{i=0}^{n_{1}} \mu_{i}(x) \boldsymbol{p}_{i}^{(k)},~x \in\left[x_{0}, x_{m}\right].
\end{equation*}
for $k=0,1,2, \cdots$ and the $j$ th difference
\begin{equation*}
\boldsymbol{r}_{j}^{(k)}=\boldsymbol{q}_{j}-\mathcal{C}^{(k)}\left(x_{j}\right),
\end{equation*}
for $j \in[m]$. Let the index $t_{k} \in[t]$ be selected with probability.
\begin{equation*}
\mathbb{P}\left(\text { Index }=t_{k}\right)=\frac{\sum_{i_{k} \in U_{t_{k}}} \sum_{j=0}^{m} \mu_{i_{k}}^{2}\left(x_{j}\right)}{\sum_{i=0}^{n_{1}} \sum_{j=0}^{m} \mu_{i}^{2}\left(x_{j}\right)}. 
\end{equation*}
If $i_{k} \in U_{t_{k}} \subseteq\left[n_{1}\right]$, the $i_{k}$ th adjusting vector is computed by
\begin{equation*}
\boldsymbol{\delta}_{i_{k}}^{(k)}=\frac{\sum_{j=0}^{m} \mu_{i_{k}}\left(x_{j}\right) r_{j}^{(k)}}{\sum_{i_{k} \in U_{t_{k}}} \sum_{j=0}^{m} \mu_{i_{k}}^{2}\left(x_{j}\right)},
\end{equation*}
otherwise it is zero. Then we update the control points in a block form according to
\begin{equation*}
p_{U_{t_{k}}}^{(k+1)}=p_{U_{t_{k}}}^{(k)}+\boldsymbol\delta_{U_{t_{k}}}^{(k)},
\end{equation*}
while the other control points keep fixed. After that, the next curve is generated by

$$
\mathcal{C}^{(k+1)}(x)=\sum_{i=0}^{n_{1}} \mu_{i}(x) \boldsymbol{p}_{i}^{(k+1)}=\mathcal{C}^{(k)}(x)+\sum_{i_{k} \in U_{t_{k}}} \mu_{i_{k}}(x) \boldsymbol{\delta}_{i_{k}}^{(k)} .
$$
\subsection{The case of surfaces}
Now we present the original RPIA method for surface fitting. Let $\left\{V_{j}\right\}_{j=0}^{s}$ denote another partition of $\left[n_{1}\right]$. Given the data points $\left\{\boldsymbol{Q}_{h l}\right\}_{h, l=0}^{m, p}$, the initial control points $\left\{\boldsymbol{P}_{i j}^{(0)}\right\}_{i, j=0}^{n_{1}, n_{2}}$, we first construct an initial surface

$$
\mathcal{S}^{(0)}(x, y)=\sum_{i=0}^{n_{1}} \sum_{j=0}^{n_{2}} \mu_{i}(x) v_{j}(y) \boldsymbol{P}_{i j}^{(0)}, ~x \in\left[x_{0}, x_{m}\right], ~y \in\left[y_{0}, y_{p}\right],
$$
where $\nu_j(y)$ is another blending basis sequence defined on $[0,1]$. Correspondingly, we can obtain the $k$-th surface:
$$
S^{(k)}(x, y)=\sum_{i=0}^{n_{1}} \sum_{j=0}^{n_{2}} \mu_{i}(x) v_{j}(y) \boldsymbol{P}_{i j}^{(k)},
$$
and the $(h, l)$ th difference
\begin{equation*}
\boldsymbol{R}_{h l}^{(k)}=\boldsymbol{Q}_{h l}-\boldsymbol{S}^{(k)}\left(x_{h}, y_{l}\right),
\end{equation*}
for $h \in[m]$ and $l \in[p]$, we randomly choose $t_{k} \in[t]$ and $s_{k} \in[s]$ with probabilities
\begin{equation*}
\mathbb{P}\left(\text { Index }=t_{k}\right)=\frac{\sum_{i_{k} \in U_{t_{k}}} \sum_{h=0}^{m} \mu_{i_{k}}^{2}\left(x_{h}\right)}{\sum_{i=0}^{n_{1}} \sum_{h=0}^{m} \mu_{i}^{2}\left(x_{h}\right)} \text { and } \mathbb{P}\left(\text { Index }=s_{k}\right)=\frac{\sum_{j_{k} \in V_{s_{k}}} \sum_{l=0}^{p} v_{j_{k}}^{2}\left(y_{l}\right)}{\sum_{j=0}^{n_{2}} \sum_{l=0}^{p} v_{j}^{2}\left(y_{l}\right)}, 
\end{equation*}
respectively, and compute the $\left(i_{k}, j_{k}\right)$ th adjusting vector based on
\begin{equation*}
\boldsymbol{\Delta}_{i_{k}, j_{k}}^{(k)}=\frac{\sum_{h=0}^{m} \sum_{l=0}^{p} \mu_{i_{k}}\left(x_{h}\right) v_{j_{k}}\left(y_{l}\right) \boldsymbol{R}_{h l}^{(k)}}{\left(\sum_{i_{k} \in U_{t_{k}}} \sum_{h=0}^{m} \mu_{i_{k}}^{2}\left(x_{h}\right)\right)\left(\sum_{j_{k} \in V_{s_{k}}} \sum_{l=0}^{p} v_{j_{k}}^{2}\left(y_{l}\right)\right)},
\end{equation*}
if $\left(i_{k}, j_{k}\right) \in\left(U_{t_{k}}, V_{s_{k}}\right)$ with $U_{t_{k}} \subseteq\left[n_{1}\right]$ and $V_{s_{k}} \subseteq\left[n_{2}\right]$, otherwise it is zero. Then, the control points in a block form are updated by
\begin{equation*}
\boldsymbol{P}_{U_{t_{k}}, V_{s_{k}}}^{(k+1)}=\boldsymbol{P}_{U_{t_{k}}, V_{s_{k}}}^{(k)}+\boldsymbol{\Delta}_{U_{t_{k}}, V_{s_{k}}}^{(k)}, 
\end{equation*}
where the other control points retain fixed. With the above preparation, the next surface is generated by
$$
{S}^{(k+1)}(x, y)=\sum_{i=0}^{n_{1}} \sum_{j=0}^{n_{2}} \mu_{i}(x) v_{j}(y) \boldsymbol{P}_{i j}^{(k+1)}=S^{(k)}(x, y)+\sum_{i_{k} \in U_{t_{k}}} \sum_{j_{k} \in V_{s_{k}}} \mu_{i_{k}}(x) v_{j_{k}}(y) \boldsymbol\Delta_{i_{k}, j_{k}}^{(k)}.
$$
\section{The regularized RPIA method}
In this section, we propose a regularized RPIA method for curve and surface fitting based on the RPIA method described above.

\subsection{The regularized case of curves}
Let $\boldsymbol{A} \in \mathbb{R}^{(m+1) \times\left(n_1+1\right)}$ be the B-spline design matrix, $\boldsymbol{q} \in \mathbb{R}^{m+1}$ the data points, and $\Gamma$ is a second-order difference operator. The Tikhonov regularized problem is formulated as:

\begin{equation*}
\min _{\boldsymbol{p}}\|\boldsymbol{A} \boldsymbol{p}-\boldsymbol{q}\|_F^2+\lambda\|\Gamma \boldsymbol{p}\|_F^2, 
\end{equation*}
which is equivalent to the augmented least-squares problem:
$$
\min _p\|\hat{\boldsymbol{A}} \boldsymbol{p}-\hat{\boldsymbol{q}}\|_F^2,
$$
where the augmented matrix and target vector are defined as:
$$
\hat{\boldsymbol{A}}=\left[\begin{array}{c}
\boldsymbol{A} \\
\sqrt{\lambda} \Gamma
\end{array}\right], \quad \hat{\boldsymbol{q}}=\left[\begin{array}{l}
\boldsymbol{q} \\
\mathbf{0}
\end{array}\right].
$$

Next, let $\left\{U_i\right\}_{i=0}^t$ be a partition of $\left[n_1\right]$. Given the data point$\left\{\boldsymbol{q}_{j}\right\}_{j=0}^{m}$, the initial control point$\left\{\boldsymbol{p}_{i}^{(0)}\right\}_{i=0}^{n_{1}}$ and a blending basis sequence $\left\{\mu_i(x): i \in\left[n_1\right]\right\}$ defined on $[0,1]$, we first construct an initial curve.

$$
\mathcal{C}^{(0)}(x)=\sum_{i=0}^{n_{1}} \mu_{i}(x) \boldsymbol{p}_{i}^{(0)}, ~x \in\left[x_{0}, x_{m}\right],
$$
and compute the difference according to
$$
\boldsymbol{r}^{(0)}=\hat{\boldsymbol{q}}-\hat{\boldsymbol{A}} \boldsymbol{p}^{(0)} .
$$
Let the index $t_{0} \in[t]$ be selected with probability
$$
\mathbb{P}\left(\text { Index }=t_0\right)=\frac{\left\|\hat{\boldsymbol{A}}_{:, U_{t_0}}\right\|_F^2}{\sum_{i=0}^t\left\|\hat{\boldsymbol{A}}_{:, U_{i}}\right\|_F^2},
$$
where $\hat{\boldsymbol{A}}_{U_{t_k}}$ denote the columns of $\hat{\boldsymbol{A}}$ indexed by $U_{t_k}$. The adjusting vectors in keeping with
$$
\tilde{\mu}_0=\frac{1}{\left\|\hat{\boldsymbol{A}}_{:, U_{t_0}}\right\|_F^2}, \quad \boldsymbol{\tilde\delta}_{i_0}^{(0)}=\tilde{\mu}_0 \cdot \hat{\boldsymbol{A}}_{:, U_{t_0}}^T \boldsymbol{r}^{(0)}.
$$
Then we update the partial control points in a block form as

$$
\boldsymbol{p}_{U_{t_{0}}}^{(1)}=\boldsymbol{p}_{U_{t_{0}}}^{(0)}+\boldsymbol{\tilde\delta}_{U_{t_{0}}}^{(0)},
$$
in which the rest of control points remain unchanged and $\boldsymbol{x}_{U}$ denotes the subvector of $\boldsymbol{x}$ indexed by $U$ for any vector $\boldsymbol{x}$.

Recursively, assuming that we have obtained the $k$ th curve

\begin{equation}\label{3.1}
\mathcal{C}^{(k)}(x)=\sum_{i=0}^{n_{1}} \mu_{i}(x) \boldsymbol{p}_{i}^{(k)}, ~x \in\left[x_{0}, x_{m}\right],
\end{equation}
for $k=0,1,2, \cdots$ and the difference
\begin{equation}\label{3.2}
\boldsymbol{r}^{(k)}=\hat{\boldsymbol{q}}-\hat{\boldsymbol{A}} \boldsymbol{p}^{(k)}.
\end{equation}
Let the index $t_{k} \in[t]$ be selected with probability
\begin{equation}\label{3.3}
\mathbb{P}\left(\text { Index }=t_k\right)=\frac{\left\|\hat{\boldsymbol{A}}_{:, U_{t_k}}\right\|_F^2}{\sum_{i=0}^t\left\|\hat{\boldsymbol{A}}_{:, U_{i}}\right\|_F^2}.
\end{equation}
The adjusting vectors are computed by
\begin{equation}\label{3.4}
\tilde{\mu}_k=\frac{1}{\left\|\hat{\boldsymbol{A}}_{:, U_{t_k}}\right\|_F^2}, \quad \boldsymbol{\tilde\delta_{i_k}}^{(k)}=\tilde{\mu}_k \cdot \hat{\boldsymbol{A}}_{:, U_{t_k}}^T \boldsymbol{r}^{(k)}. 
\end{equation}
Then we update the control points in a block form according to
\begin{equation}\label{3.5}
\boldsymbol{p}_{U_{t_{k}}}^{(k+1)}=\boldsymbol{p}_{U_{t_{k}}}^{(k)}+\boldsymbol{\tilde\delta}_{U_{t_{k}}}^{(k)},
\end{equation}
while the other control points keep fixed. After that, the next curve is generated by
$$
\mathcal{C}^{(k+1)}(x)=\sum_{i=0}^{n_{1}} \mu_{i}(x) \boldsymbol{p}_{i}^{(k+1)}=\mathcal{C}^{(k)}(x)+\sum_{i_{k} \in U_{t_{k}}} \mu_{i_{k}}(x) \boldsymbol{\tilde\delta}_{i_{k}}^{(k)} .
$$
The regularized RPIA for curve fitting is arranged in Algorithm \ref{algorithm1}.

\begin{algorithm}
\caption{Regularized RPIA for curve fitting}
\label{algorithm1}
\begin{algorithmic}[1]
\STATE \textbf{Input:} Data points \(\{\mathbf{q}_j\}_{j=0}^m\), initial control points \(\{\mathbf{p}_i^{(0)}\}_{i=0}^n\), real increasing sequence \(\{x_j\}_{j=0}^m\), partition \(\{U_i\}_{i=0}^\ell\), maximum iteration number \(\ell\).
\STATE \textbf{Output:} \(\mathcal{C}^{(\ell+1)}(x)\).

\FOR{\(k = 0, 1, 2, \ldots, \ell\)}
      \STATE Generate the blending curve \(\mathcal{C}^{(k)}(x)\) as (\ref{3.1});
      \STATE Calculate the differences as (\ref{3.2});
      \STATE Randomly pick the index \(t_k\) using (\ref{3.3});
      \STATE Compute the adjusting vectors via (\ref{3.4});
      \STATE Update control points via (\ref{3.5});
\ENDFOR
\end{algorithmic}
\end{algorithm}

\subsection{The regularized case of surfaces}
We now turn to the case of surfaces. Let $\boldsymbol{A} \in \mathbb{R}^{(m+1) \times\left(n_1+1\right)}$ and $\boldsymbol{B} \in \mathbb{R}^{(p+1) \times\left(n_2+1\right)}$ be the B-spline design matrices in the u-direction and v-direction respectively. Let $\boldsymbol{Q} \in$ $\mathbb{R}^{(m+1) \times(p+1)}$ the data points . $\boldsymbol{L}_u \in \mathbb{R}^{\left(n_1+1\right) \times\left(n_1+1\right)}$ and $\boldsymbol{L}_v \in \mathbb{R}^{\left(n_2+1\right) \times\left(n_2+1\right)}$ are the second-order difference regularization matrices in the $u, ~ v$ directions respectively. The regularization problem is stated as:
$$
\min _{\boldsymbol{P}}\left\|\boldsymbol{A} \boldsymbol{P} \boldsymbol{B}^T-\boldsymbol{Q}\right\|_F^2+\lambda\left\|\boldsymbol{A} \boldsymbol{P}\boldsymbol{L}_v^T\right\|_F^2+\lambda\left\|\boldsymbol{L_u} \boldsymbol{P}\boldsymbol{B}^T\right\|_F^2+\lambda^2\left\|\boldsymbol{L_u} \boldsymbol{P}\boldsymbol{L_v}^T\right\|_F^2
$$
It is equivalent to the augmented tensor least squares problem:
$$
\min _{\boldsymbol{P}}\left\|\hat{\boldsymbol{A}} \boldsymbol{P} \hat{\boldsymbol{B}}^T-\hat{\boldsymbol{Q}}\right\|_F^2,
$$
where $\hat{\boldsymbol{A}}=\left[\begin{array}{c}\boldsymbol{A} \\ \sqrt{\lambda} \boldsymbol{L}_u\end{array}\right], \hat{\boldsymbol{B}}=\left[\begin{array}{c}\boldsymbol{B} \\ \sqrt{\lambda} \boldsymbol{L}_v\end{array}\right], \hat{\boldsymbol{Q}}=\left[\begin{array}{cc}\boldsymbol{Q} & \mathbf{0} \\ \mathbf{0} & \mathbf{0}\end{array}\right]$.

Let $\left\{U_i\right\}_{i=0}^t$ and $\left\{V_j\right\}_{j=0}^s$ denote the partition of $\left[n_1\right]$ and $\left[n_2\right]$ respectively. Then the initial fitting surface is:
$$
\mathcal{S}^{(0)}(x, y)=\sum_{i=0}^{n_{1}} \sum_{j=0}^{n_{2}} \mu_{i}(x) v_{j}(y) \boldsymbol{P}_{i j}^{(0)}, ~x \in\left[x_{0}, x_{m}\right], ~y \in\left[y_{0}, y_{p}\right],
$$
where $\nu_j(y)$ is another blending basis sequence defined on $[0,1]$ and $\left\{\boldsymbol{P}_{i j}^{(0)}\right\}_{i, j=0}^{n_{1}, n_{2}}$ is the initial control points. And compute the difference according to
$$
\boldsymbol{R}^{(0)}=\hat{\boldsymbol{Q}}-\hat{\boldsymbol{A}} \boldsymbol{P}^{(0)} \hat{\boldsymbol{B}}^T.
$$
Randomly selecting $t_{0} \in[t]$ and $s_{0} \in[s]$ with probabilities
$$
\mathbb{P}\left(\text {Index=} {t_0}\right)=\frac{\sum_{i \in U_{t_0}}\left\|\hat{\boldsymbol{A}}_{:, i}\right\|_F^2}{\sum_{l=0}^t \sum_{i \in U_{l}}\left\|\hat{\boldsymbol{A}}_{:, i}\right\|_F^2}=\frac{\left\|\hat{\boldsymbol{A}}_{:, U_{t_0}}\right\|_F^2}{\|\hat{\boldsymbol{A}}\|_F^2},
$$
$$
\mathbb{P}\left(\text {Index=} {s_0}\right)=\frac{\sum_{j \in V_{s_0}}\left\|\hat{\boldsymbol{B}}_{:, j}\right\|_F^2}{\sum_{l=0}^s \sum_{j \in V_{l}}\left\|\hat{\boldsymbol{B}}_{:, j}\right\|_F^2}=\frac{\left\|\hat{\boldsymbol{B}}_{:, V_{s_0}}\right\|_F^2}{\|\hat{\boldsymbol{B}}\|_F^2}.
$$
We calculate the adjusting vectors according to
$$
\tilde{\boldsymbol\Delta}_{U_{t_0}, V_{s_0}}^{(0)}=\frac{1}{\left\|\hat{\boldsymbol{A}}_{:, U_{t_0}}\right\|_F^2\left\|\hat{\boldsymbol{B}}_{:, V_{s_0}}\right\|_F^2} \hat{\boldsymbol{A}}_{:, U_{t_0}}^T \boldsymbol{R}^{(0)} \hat{\boldsymbol{B}}_{:, V_{s_0}}.
$$
And update the control points in the light of
$$
\boldsymbol{P}_{U_{t_{0}}, V_{s_{0}}}^{(1)}=\boldsymbol{P}_{U_{t_{0}}, V_{s_{0}}}^{(0)}+\tilde{\boldsymbol{\Delta}}_{U_{t_{0}}, V_{s_{0}}}^{(0)},
$$
where the other control points remain unchanged.\\

Recursively, after obtaining the $k$ th surface
\begin{equation}\label{3.6}
\mathcal{S}^{(k)}(x, y)=\sum_{i=0}^{n_{1}} \sum_{j=0}^{n_{2}} \mu_{i}(x) v_{j}(y) \boldsymbol{P}_{i j}^{(k)},
\end{equation}
and the difference
\begin{equation}\label{3.7}
\boldsymbol{R}^{(k)}=\hat{\boldsymbol{Q}}-\hat{\boldsymbol{A}} \boldsymbol{P}^{(k)} \hat{\boldsymbol{B}}^T, 
\end{equation}
we randomly choose $t_{k} \in[t]$ and $s_{k} \in[s]$ with probabilities
\begin{equation}
\left\{
\label{3.8}
\begin{aligned}
    &\mathbb{P}\left(\text {Index=} {t_k}\right)=\frac{\sum_{i \in U_{t_k}}\left\|\hat{\boldsymbol{A}}_{:, i}\right\|_F^2}{\sum_{l=0}^t \sum_{i \in U_{l}}\left\|\hat{\boldsymbol{A}}_{:, i}\right\|_F^2}=\frac{\left\|\hat{\boldsymbol{A}}_{:, U_{t_k}}\right\|_F^2}{\|\hat{\boldsymbol{A}}\|_F^2}, \\
    &\mathbb{P}\left(\text {Index=} {s_k}\right)=\frac{\sum_{j \in V_{s_k}}\left\|\hat{\boldsymbol{B}}_{:, j}\right\|_F^2}{\sum_{l=0}^s \sum_{j \in V_{l}}\left\|\hat{\boldsymbol{B}}_{:, j}\right\|_F^2}=\frac{\left\|\hat{\boldsymbol{B}}_{:, V_{s_k}}\right\|_F^2}{\|\hat{\boldsymbol{B}}\|_F^2} .
\end{aligned}
\right.
\end{equation}
We calculate the adjusting vectors according to
\begin{equation}\label{3.9}
\tilde{\boldsymbol\Delta}_{U_{t_k}, V_{s_k}}^{(k)}=\frac{1}{\left\|\hat{\boldsymbol{A}}_{:, U_{t_k}}\right\|_F^2\left\|\hat{\boldsymbol{B}}_{:, V_{s_k}}\right\|_F^2} \hat{\boldsymbol{A}}_{:, U_{t_k}}^T \boldsymbol{R}^{(k)} \hat{\boldsymbol{B}}_{:, V_{s_k}}.
\end{equation}
Then, the control points in a block form are updated by
\begin{equation}\label{3.10}
\boldsymbol{P}_{U_{t_{k}}, V_{s_{k}}}^{(k+1)}=\boldsymbol{P}_{U_{t_{k}}, V_{s_{k}}}^{(k)}+\tilde{\boldsymbol{\Delta}}_{U_{t_{k}}, V_{s_{k}}}^{(k)},
\end{equation}
where the other control points retain fixed. With the above preparation, the next surface is generated by
$$
\boldsymbol{S}^{(k+1)}(x, y)=\sum_{i=0}^{n_{1}} \sum_{j=0}^{n_{2}} \mu_{i}(x) v_{j}(y) \boldsymbol{P}_{i j}^{(k+1)}=S^{(k)}(x, y)+\sum_{i_{k} \in U_{t_{k}}} \sum_{j_{k} \in V_{s_{k}}} \mu_{i_{k}}(x) v_{j_{k}}(y) \tilde{\boldsymbol\Delta}_{i_{k}, j_{k}}^{(k)} .
$$
The regularized RPIA for surface fitting is arranged in Algorithm~\ref{algorithm2}.

\begin{algorithm}
\caption{Regularized RPIA for surface fitting}
\label{algorithm2}
\begin{algorithmic}[1]
\STATE \textbf{Input:} Data points $\left\{\boldsymbol{Q}_{h l}\right\}_{h, l=0}^{m, p}$, the initial control points $\left\{\boldsymbol{P}_{i j}^{(0)}\right\}_{i, j=0}^{n_{1}, n_{2}}$, two real increasing sequences $\left\{x_{h}\right\}_{h=0}^{m}$ and $\left\{y_{l}\right\}_{l=0}^{p}$, two partitions $\left\{U_{i}\right\}_{i=0}^{t}$ and $\left\{V_{j}\right\}_{j=0}^{s}$, and the maximum iteration number $\ell$.
\STATE \textbf{Output:} \(\mathcal{S}^{(\ell+1)}(x)\).

\FOR{\(k = 0, 1, 2, \ldots, \ell\)}
      \STATE generate the blending surface $S^{(k)}(x, y)$ as (\ref{3.6});
      \STATE calculate the differences as (\ref{3.7});
      \STATE randomly pick the indices $t_{k}$ and $s_{k}$ as (\ref{3.8});
      \STATE compute the adjusting vectors as (\ref{3.9});
      \STATE update the control points as (\ref{3.10});
\ENDFOR
\end{algorithmic}
\end{algorithm}

\section{Convergence analyses for regularized RPIA method}
In this section, we utilize matrix theory to analyze the convergence of regularized RPIA for least-squares fitting.
\subsection{The case of curves}
By the linear algebra theory, $\boldsymbol{p}^{*}$ is a least-squares solution of $\hat{\boldsymbol{A}} \boldsymbol{p}=\hat{\boldsymbol{q}}$ if and only if $\hat{\boldsymbol{A}}^{T} \hat{\boldsymbol{A}} \boldsymbol{p}^{*}=\hat{\boldsymbol{A}}^{T} \hat{\boldsymbol{q}}$. Now we give the convergence analysis of Algorithm \ref{algorithm1} in the following theorem.
\begin{theorem}
Let $\left\{\mu_{i}(x): i \in\left[n_{1}\right]\right\}$ be a blending basis sequence and $\hat{\boldsymbol{A}}$ be the corresponding collocation matrix on the real increasing sequence $\left\{x_{h}\right\}_{h=0}^{m}$. Suppose that $\hat{\boldsymbol{A}}$ has a full-column rank, when the number of data points is larger than that of control points, the fitting curve sequence, generated by regularized RPIA (see Algorithm \ref{algorithm1}), converges to the least-squares fitting solution in expectation.
\end{theorem}
\begin{proof}
Let us introduce an auxiliary intermediate $\boldsymbol{z}^{(k)}=\hat{\boldsymbol{A}}\left(\boldsymbol{p}^{(k)}-\boldsymbol{p}^{*}\right)$ for $k=0,1,2, \cdots$, where $\boldsymbol{p}^{*}:=\left(\hat{\boldsymbol{A}}^{T} \hat{\boldsymbol{A}}\right)^{-1} \hat{\boldsymbol{A}}^{T} \hat{\boldsymbol{q}}$ is a least-squares solution. Multiplying both sides by the transpose of the collocation matrix yields
$$
\hat{\boldsymbol{A}}^{T} \boldsymbol{z}^{(k)}=\hat{\boldsymbol{A}}^{T} \hat{\boldsymbol{A}} \boldsymbol{p}^{(k)}-\left(\hat{\boldsymbol{A}}^{T} \hat{\boldsymbol{A}}\right)\left(\hat{\boldsymbol{A}}^{T} \hat{\boldsymbol{A}}\right)^{-1} \hat{\boldsymbol{A}}^{T} \hat{\boldsymbol{q}}=-\hat{\boldsymbol{A}}^{T} \boldsymbol{r}^{(k)},
$$
with $\boldsymbol{r}^{(k)}=\hat{\boldsymbol{q}}-\hat{\boldsymbol{A}} \boldsymbol{p}^{(k)}$. Combined with formula (\ref{3.5}), it implies that
$$
\hat{\boldsymbol{A}} \boldsymbol{\delta}^{(k)}=\hat{\boldsymbol{A}}\left[\boldsymbol{\delta}_{0}^{(k)} \boldsymbol{\delta}_{1}^{(k)} \cdots \boldsymbol{\delta}_{n}^{(k)}\right]^{T}=\frac{\hat{\boldsymbol{A}}_{:, U_{t_{k}}} \hat{\boldsymbol{A}}_{:, U_{t_{k}}}^{T} \boldsymbol{r}^{(k)}}{\left\|\hat{\boldsymbol{A}}_{:, U_{t_{k}}}\right\|_{F}^{2}}=-\frac{\hat{\boldsymbol{A}}_{:, U_{t_{k}}} \hat{\boldsymbol{A}}_{:, U_{t_{k}}}^{T} \boldsymbol{z}^{(k)}}{\left\|\hat{\boldsymbol{A}}_{:, U_{t_{k}}}\right\|_{F}^{2}},
$$
and
$$
\boldsymbol{z}^{(k+1)}=\hat{\boldsymbol{A}}\left(\boldsymbol{p}^{(k)}+\boldsymbol{\delta}^{(k)}-\boldsymbol{p}^{*}\right)=\left(\boldsymbol{I}_{m+1}-\frac{\hat{\boldsymbol{A}}_{:, U_{t_{k}}} \hat{\boldsymbol{A}}_{:, U_{t_{k}}}^{T}}{\left\|\hat{\boldsymbol{A}}_{:, U_{t_{k}}}\right\|_{F}^{2}}\right) \boldsymbol{z}^{(k)}.
$$
Let $\mathbb{E}_{k}^{U}[\cdot]=\mathbb{E}\left[\cdot \mid t_{0}, t_{1}, \cdots, t_{k}\right]$ denote the conditional expectation conditioned on the first $k$ iterations of Algorithm \ref{algorithm1} , where $t_{k}$ means that the $t_{k}$ th index set $U_{t_{k}}$ is chosen. By taking this conditional expectation, it yields that
$$
\begin{aligned}
\mathbb{E}_k^U\left[\boldsymbol{z}^{(k+1)}\right] & =\boldsymbol{z}^{(k)}-\mathbb{E}_k^U\left[\frac{\hat{\boldsymbol{A}}_{:, U_{t_k}} \hat{\boldsymbol{A}}_{:, U_{t_k}}^T}{\left\|\hat{\boldsymbol{A}}_{:, U_{t_k}}\right\|_F^2}\right] \boldsymbol{z}^{(k)} \\
& =\boldsymbol{z}^{(k)}-\sum_{i=0}^t \frac{\left\|\hat{\boldsymbol{A}}_{:, U_i}\right\|_F^2}{\|\hat{\boldsymbol{A}}\|_F^2} \frac{\hat{\boldsymbol{A}}_{:, U_i} \hat{\boldsymbol{A}}_{:, U_i}^T}{\left\|\hat{\boldsymbol{A}}_{:, U_i}\right\|_F^2} \boldsymbol{z}^{(k)} \\
& =\left(\boldsymbol{I}-\frac{\hat{\boldsymbol{A}} \hat{\boldsymbol{A}}^T}{\|\hat{\boldsymbol{A}}\|_F^2}\right) \boldsymbol{z}^{(k)} .
\end{aligned}
$$
Based on the law of total expectation and unrolling the recurrence, it gives that
$$
\mathbb{E}\left[\boldsymbol{z}^{(k+1)}\right]=\left(\boldsymbol{I}-\frac{\hat{\boldsymbol{A}} \hat{\boldsymbol{A}}^{T}}{\|\hat{\boldsymbol{A}}\|_{F}^{2}}\right) \mathbb{E}\left[\boldsymbol{z}^{(k)}\right]=\cdots=\left(\boldsymbol{I}-\frac{\hat{\boldsymbol{A}} \hat{\boldsymbol{A}}^{T}}{\|\hat{\boldsymbol{A}}\|_{F}^{2}}\right)^{k+1} \boldsymbol{z}^{(0)}
$$
Multiplying left by $\hat{\boldsymbol{A}}^{T}$, we have

$$
\hat{\boldsymbol{A}}^{T} \mathbb{E}\left[\boldsymbol{z}^{(k+1)}\right]=\left(\boldsymbol{I}-\frac{\hat{\boldsymbol{A}}^{T} \hat{\boldsymbol{A}}}{\|\hat{\boldsymbol{A}}\|_{F}^{2}}\right)^{k+1} \hat{\boldsymbol{A}}^{T} \boldsymbol{z}^{(0)}.
$$
Since $\hat{\boldsymbol{A}}$ is full of column rank, $\hat{\boldsymbol{A}}^{T} \hat{\boldsymbol{A}}$ is positive definite. It leads to
$$
0 \leq \rho\left(\boldsymbol{I}-\frac{\hat{\boldsymbol{A}}^{T} \hat{\boldsymbol{A}}}{\|\hat{\boldsymbol{A}}\|_{F}^{2}}\right)<1,
$$
where $\rho(\boldsymbol{M})$ is the spectral radius of $\boldsymbol{M}$. Therefore,
$$
\lim _{k \rightarrow \infty}\left(\boldsymbol{I}-\frac{\hat{\boldsymbol{A}}^{T} \hat{\boldsymbol{A}}}{\|\hat{\boldsymbol{A}}\|_{F}^{2}}\right)^{k}=\boldsymbol{O},
$$
where $\boldsymbol{O}$ is the rank zero matrix. It follows that
$$
\mathbb{E}\left[\boldsymbol{p}^{(\infty)}\right]=\left(\hat{\boldsymbol{A}}^{T} \hat{\boldsymbol{A}}\right)^{-1}\left(\hat{\boldsymbol{A}}^{T} \mathbb{E}\left[\boldsymbol{z}^{(\infty)}\right]+\hat{\boldsymbol{A}}^{T} \hat{\boldsymbol{A}} \boldsymbol{p}^{*}\right)=\left(\hat{\boldsymbol{A}}^{T} \hat{\boldsymbol{A}}\right)^{-1} \hat{\boldsymbol{A}}^{T} \hat{\boldsymbol{A}} \boldsymbol{p}^{*}=\boldsymbol{p}^{*}
$$
According to algebraic considerations, the regularized RPIA limit curve is the least-squares fitting result to the given data points.
\end{proof}
\subsection{The case of surfaces}
Similar to the case of curves, the surface sequence, generated by Algorithm \ref{algorithm2}, is convergent in expectation to the least-squares fitting result for the given data points. The result is stated as follows.
\begin{theorem}
Let $\left\{\mu_{i}(x): i \in\left[n_{1}\right]\right\}$ and $\left\{v_{j}(y): j \in\left[n_{2}\right]\right\}$ be two blending basis sequences, and $\hat{\boldsymbol{A}}$ and $\hat{\boldsymbol{B}}$ be the corresponding collocation matrices on the real increasing sequences $\left\{x_{h}\right\}_{h=0}^{m}$ and $\left\{y_{l}\right\}_{l=0}^{p}$, respectively. Suppose that $\hat{\boldsymbol{A}}$ and $\hat{\boldsymbol{B}}$ have a full-column rank, when the number of data points is larger than that of control points, the fitting surface sequence, generated by regularized RPIA (see Algorithm \ref{algorithm2}), converges to the least-squares fitting solution in expectation.
\end{theorem}

\begin{proof}
As a preparatory step, let us introduce an auxiliary third-order tensor

$$
\boldsymbol{Z}^{(k)}=\left[\boldsymbol{Z}_{i j f}^{(k)}\right]_{i=0, j=0, f=1}^{m, p, 3},
$$
whose $f$ th frontal slice is given by
$$
\boldsymbol{Z}_{(f)}^{(k)}=\hat{\boldsymbol{A}}\left(\boldsymbol{P}_{(f)}^{(k)}-\boldsymbol{P}_{(f)}^{*}\right) \hat{\boldsymbol{B}^T}
$$
with $\boldsymbol{P}_{(f)}^{*}=\left(\hat{\boldsymbol{A}}^{T} \hat{\boldsymbol{A}}\right)^{-1} \hat{\boldsymbol{A}}^{T} \boldsymbol{Q}_{(f)} \hat{\boldsymbol{B}}\left(\hat{\boldsymbol{B}}^{T} \hat{\boldsymbol{B}}\right)^{-1}$ for $f=1,2,3$ and $k=0,1,2, \cdots$. Left multiplication by $\boldsymbol{A}^{T}$ and right multiplication by $\hat{\boldsymbol{B}}$ to $\boldsymbol{Z}_{(f)}^{(k)}$ lead to
$$
\hat{\boldsymbol{A}}^{T} \boldsymbol{Z}_{(f)}^{(k)} \hat{\boldsymbol{B}}=\hat{\boldsymbol{A}}^{T} \hat{\boldsymbol{A}} \boldsymbol{P}_{(f)}^{(k)} \hat{\boldsymbol{B}}^{T} \hat{\boldsymbol{B}}-\hat{\boldsymbol{A}}^{T} \hat{\boldsymbol{A}} \boldsymbol{P}_{(f)}^{*} \hat{\boldsymbol{B}}^{T} \hat{\boldsymbol{B}}=-\hat{\boldsymbol{A}}^{T} \boldsymbol{R}_{(f)}^{(k)} \hat{\boldsymbol{B}}
$$
with $\boldsymbol{R}_{(f)}^{(k)}=\boldsymbol{Q}-\hat{\boldsymbol{A}} \boldsymbol{P}_{(f)}^{(k)} \hat{\boldsymbol{B}}^T$. It follows that
$$
\begin{aligned}
\boldsymbol{Z}_{(f)}^{(k+1)} & =\hat{\boldsymbol{A}}\left(\boldsymbol{P}_{(f)}^{(k+1)}-\boldsymbol{P}_{(f)}^{*}\right) \hat{\boldsymbol{B}}^T \\
& =\hat{\boldsymbol{A}}\left(\boldsymbol{P}_{(f)}^{(k)}-\boldsymbol{P}_{(f)}^{*}+\boldsymbol{\Delta}_{(f)}^{(k)}\right) \hat{\boldsymbol{B}}^T \\
& =\boldsymbol{Z}_{(f)}^{(k)}-\frac{\hat{\boldsymbol{A}}_{:, U_{t_{k}}} \hat{\boldsymbol{A}}_{:, U_{t_{k}}}^{T} \boldsymbol{Z}_{(f)}^{(k)} \hat{\boldsymbol{B}}_{:,V_{s_{k}}} \hat{\boldsymbol{B}}_{:,V_{s_{k}}}^{T}}{\left\|\hat{\boldsymbol{A}}_{:, U_{t_{k}}}\right\|_{F}^{2}\left\|\hat{\boldsymbol{B}}_{V_{:,s_{k}}}\right\|_{F}^{2}},
\end{aligned}
$$
where the last equality is from
$$
\hat{\boldsymbol{A}} \boldsymbol{\Delta}_{(f)}^{(k)} \hat{\boldsymbol{B}}=\frac{\hat{\boldsymbol{A}} \widehat{\boldsymbol{I}}_{:, U_{t_{k}}} \hat{\boldsymbol{A}}_{:, U_{t_{k}}}^{T} \boldsymbol{R}_{(f)}^{(k)} \hat{\boldsymbol{B}}_{:,V_{s_{k}}}\widehat{\boldsymbol{I}}_{V_{s_{k}},:} \hat{\boldsymbol{B}}^{T}}{\left\|\hat{\boldsymbol{A}}_{:, U_{t_{k}}}\right\|_{F}^{2}\left\|\hat{\boldsymbol{B}}_{:,V_{s_{k}}}\right\|_{F}^{2}}
$$
$$
\begin{aligned}
& =\frac{\hat{\boldsymbol{A}}_{:, U_{t_{k}}} \hat{\boldsymbol{I}}_{:, U_{t_{k}}}^{T} \hat{\boldsymbol{A}}^{T} \boldsymbol{R}_{(f)}^{(k)} \hat{\boldsymbol{B}}\hat{\boldsymbol{I}}_{V_{s_{k}},:}^{T} \boldsymbol{B}^{T} _{:,V_{s_{k}}}}{\left\|\hat{\boldsymbol{A}}_{:, U_{t_{k}}}\right\|_{F}^{2}\left\|\hat{\boldsymbol{B}} _{:,V_{s_{k}}}\right\|_{F}^{2}} \\
& =-\frac{\hat{\boldsymbol{A}}_{:, U_{t_{k}}} \hat{\boldsymbol{A}}_{:, U_{t_{k}}}^{T} \boldsymbol{Z}_{(f)}^{(k)} \hat{\boldsymbol{B}}_{:,V_{s_{k}}} \hat{\boldsymbol{B}}_{:,V_{s_{k}}}^{T}}{\left\|\hat{\boldsymbol{A}}_{:, U_{t_{k}}}\right\|_{F}^{2}\left\|\hat{\boldsymbol{B}}_{:,V_{s_{k}}}\right\|_{F}^{2}}
\end{aligned}
$$
with $\widehat{\boldsymbol{I}}$ being an identity matrix. Let $\mathbb{E}_{k}^{U, V}[\cdot]=\mathbb{E}\left[\cdot \mid s_{0}, t_{0}, s_{1}, t_{1}, \cdots, s_{k}, t_{k}\right]$ denote the conditional expectation conditioned on the first $k$ iterations of Algorithm \ref{algorithm2}, where the $t_{k}$ th index set $U_{t_{k}}$ and the $s_{k}$ th index set $V_{s_{k}}$ are chosen as $\mathbb{E}_{k}^{V}[\cdot]=\mathbb{E}\left[\cdot \mid s_{0}, t_{0}, \cdots, s_{k-1}, t_{k-1}, t_{k}\right]$ and $\mathbb{E}_{k}^{U}[\cdot]=\mathbb{E}\left[\cdot \mid s_{0}, t_{0}, \cdots, s_{k-1}, t_{k-1}, s_{k}\right]$, respectively. We have

$$
\begin{aligned}
\mathbb{E}_{k}^{U, V}\left[\boldsymbol{Z}_{(f)}^{(k+1)}\right] & =\boldsymbol{Z}_{(f)}^{(k)}-\mathbb{E}_{k}^{U, V}\left[\frac{\hat{\boldsymbol{A}}_{:, U_{t_{k}}} \hat{\boldsymbol{A}}_{:, U_{t_{k}}}^{T} \boldsymbol{Z}_{(f)}^{(k)} \hat{\boldsymbol{B}}_{:,V_{s_{k}}} \hat{\boldsymbol{B}}^{T}_{:,V_{s_{k}}}}{\left\|\hat{\boldsymbol{A}}_{:, U_{t_{k}}}\right\|_{F}^{2}\left\|\hat{\boldsymbol{B}}_{:,V_{s_{k}}}\right\|_{F}^{2}}\right] \\
& =\boldsymbol{Z}_{(f)}^{(k)}-\sum_{i=0}^{t} \sum_{j=0}^{s} \frac{\left\|\hat{\boldsymbol{A}}_{:, U_{i}}\right\|_{F}^{2}}{\|\hat{\boldsymbol{A}}\|_{F}^{2}} \frac{\left\|\hat{\boldsymbol{B}}_{:,V_{j}}\right\|_{F}^{2}}{\|\hat{\boldsymbol{B}}\|_{F}^{2}} \frac{\hat{\boldsymbol{A}}_{:, U_{i}} \hat{\boldsymbol{A}}_{:, U_{i}}^{T} \boldsymbol{Z}_{(f)}^{(k)} \hat{\boldsymbol{B}}_{:,V_{j}} \hat{\boldsymbol{B}}_{:,V_{j}}^{T}}{\left\|\hat{\boldsymbol{A}}_{:, U_{i}}\right\|_{F}^{2}\left\|\hat{\boldsymbol{B}}_{:,V_{j}}\right\|_{F}^{2}} \\
& =\boldsymbol{Z}_{(f)}^{(k)}-\frac{\hat{\boldsymbol{A}} \hat{\boldsymbol{A}}^{T} \boldsymbol{Z}_{(f)}^{(k)} \hat{\boldsymbol{B}} \hat{\boldsymbol{B}}^{T}}{\|\hat{\boldsymbol{A}}\|_{F}^{2}\|\hat{\boldsymbol{B}}\|_{F}^{2}} .
\end{aligned}
$$
By the law of total expectation, it yields that
$$
\mathbb{E}\left[\boldsymbol{Z}_{(f)}^{(k+1)}\right]=\mathbb{E}\left[\boldsymbol{Z}_{(f)}^{(k)}\right]-\frac{\hat{\boldsymbol{A}} \hat{\boldsymbol{A}}^{T} \mathbb{E}\left[\boldsymbol{Z}_{(f)}^{(k)}\right] \hat{\boldsymbol{B}} \hat{\boldsymbol{B}}^{T}}{\|\hat{\boldsymbol{A}}\|_{F}^{2}\|\hat{\boldsymbol{B}}^{T}\|_{F}^{2}}.
$$
Let $\boldsymbol{Z}_{(f)}^{(k)}$ be arranged into a row partition, i.e.,
$$
\widetilde{\mathbf{z}}_{(f)}^{(k)}=\left[\begin{array}{llllllllll}
Z_{00 f}^{(k)} & \cdots & Z_{m 0 f}^{(k)} & Z_{01 f}^{(k)} & \cdots & Z_{m 1 f}^{(k)} & \cdots & Z_{0 p f}^{(k)} & \cdots & Z_{m p f}^{(k)}
\end{array}\right]^{T},
$$
for $f=1,2,3$ and $k=0,1,2 \cdots$. It achieves that
$$
\begin{aligned}
& \mathbb{E}\left[\widetilde{\boldsymbol{z}}_{(f)}^{(k+1)}\right]=\left(\boldsymbol{I}_{(m+1)(p+1)}-\left(\frac{\hat{\boldsymbol{B}} \hat{\boldsymbol{B}}^{T}}{\|\hat{\boldsymbol{B}}\|_{F}^{2}}\right) \otimes\left(\frac{\hat{\boldsymbol{A}} \hat{\boldsymbol{A}}^{T}}{\|\hat{\boldsymbol{A}}\|_{F}^{2}}\right)\right) \mathbb{E}\left[\widetilde{\boldsymbol{z}}_{(f)}^{(k)}\right] \\
& \vdots \\
&=\left(\boldsymbol{I}_{(m+1)(p+1)}-\left(\frac{\hat{\boldsymbol{B}} \hat{\boldsymbol{B}}^{T}}{\|\hat{\boldsymbol{B}}\|_{F}^{2}}\right) \otimes\left(\frac{\hat{\boldsymbol{A}} \hat{\boldsymbol{A}}^{T}}{\|\hat{\boldsymbol{A}}\|_{F}^{2}}\right)\right)^{k+1} \widetilde{\mathbf{z}}_{(f)}^{(0)}
\end{aligned}
$$
Multiplying left by $\hat{\boldsymbol{B}}^{T} \otimes \hat{\boldsymbol{A}}^{T}$, we have
$$
\left(\hat{\boldsymbol{B}}^{T} \otimes \hat{\boldsymbol{A}}^{T}\right) \mathbb{E}\left[\widetilde{\boldsymbol{z}}_{(f)}^{(k+1)}\right]=\left(\boldsymbol{I}_{\left(n_{1}+1\right)\left(n_{2}+1\right)}-\left(\frac{\hat{\boldsymbol{B}}^{T} \hat{\boldsymbol{B}}}{\|\hat{\boldsymbol{B}}\|_{F}^{2}}\right) \otimes\left(\frac{\hat{\boldsymbol{A}}^{T} \hat{\boldsymbol{A}}}{\|\hat{\boldsymbol{A}}\|_{F}^{2}}\right)\right)^{k+1}\left(\hat{\boldsymbol{B}}^{T} \otimes \hat{\boldsymbol{A}}^{T}\right) \widetilde{\boldsymbol{z}}_{(f)}^{(0)} .
$$
Since $\hat{\boldsymbol{A}}$ and $\hat{\boldsymbol{B}}$ are full of column rank, $\hat{\boldsymbol{A}}^{T} \hat{\boldsymbol{A}}$ and $\hat{\boldsymbol{B}}^{T} \hat{\boldsymbol{B}}$ are positive definite. It leads to
$$
0 \leq \rho\left(\boldsymbol{I}-\left(\frac{\hat{\boldsymbol{B}}^{T} \hat{\boldsymbol{B}}}{\|\hat{\boldsymbol{B}}\|_{F}^{2}}\right) \otimes\left(\frac{\hat{\boldsymbol{A}}^{T} \hat{\boldsymbol{A}}}{\|\hat{\boldsymbol{A}}\|_{F}^{2}}\right)\right)<1 .
$$
Therefore,
$$
\lim _{k \rightarrow \infty}\left(\boldsymbol{I}-\left(\frac{\hat{\boldsymbol{B}}^{T} \hat{\boldsymbol{B}}}{\|\hat{\boldsymbol{B}}\|_{F}^{2}}\right) \otimes\left(\frac{\hat{\boldsymbol{A}}^{T} \hat{\boldsymbol{A}}}{\|\hat{\boldsymbol{A}}\|_{F}^{2}}\right)\right)^{k}=\boldsymbol{O} .
$$
It follows that
$$
\mathbb{E}\left[\boldsymbol{P}_{(f)}^{(\infty)}\right]=\left(\hat{\boldsymbol{A}}^{T} \hat{\boldsymbol{A}}\right)^{-1}\left(\hat{\boldsymbol{A}}^{T} \mathbb{E}\left[\boldsymbol{Z}_{(f)}^{(\infty)}\right] \hat{\boldsymbol{B}}+\hat{\boldsymbol{A}}^{T} \hat{\boldsymbol{A}} \boldsymbol{P}_{(f)}^{*} \hat{\boldsymbol{B}}^{T} \hat{\boldsymbol{B}}\right)\left(\hat{\boldsymbol{B}}^{T} \hat{\boldsymbol{B}}\right)^{-1}=\boldsymbol{P}_{(f)}^{*},
$$
which indicates that the sequence of surfaces, generated by regularized RPIA, converges to the least-squares fitting result in expectation.

\end{proof}
\section{Optimal estimation for regularization parameter}\label{section5}
We proposed algorithms to solve data fitting problems.
In the following, we will illustrate the insight of the regularization method in denoise aspects. Denote measurements $q^e=q+e$, $q$ is exact data, $e$ is independent and identically distributed random noise, and $\mathbb{E}[e_i]=0$, $\mathbb{E}[e_i^2]=\sigma^2$. Our aim is to find $p$, such that $Ap$ close to $q$, $A$ is a design matrix. The direct approach is to use regularization methods to denoise noisy data, and then perform data fitting. The process is as follows.\\
(i) Find $u^*$ satisfies $\min\limits_{u} \|u-q^e\|^2 + \lambda\|Lu\|^2$. It gives that $u^*=(I+\lambda L^TL)^{-1}q^e$.\\
(ii) For $u^*$, find $p^*$ satisfies $\min\limits_{p} \|Ap-u^*\|^2$. It gives that $p^*=(A^TA)^{-1}A^{T}u^*$.

We try to find $p^*$ directly, which satisfies $\min\limits_{p} \|Ap-q^e\|^2 + \lambda\|\Gamma p\|^2$, $\Gamma$ is a positive definite matrix. 
Measurements $q^e=A\bar{p}+\epsilon+e$, where $q=A\bar{p}+\epsilon$. When $\epsilon$ are very small variables relative to $e$, this is an intuitively feasible idea. It is worth noting that only one inversion is needed here. $p^*=(A^TA+\lambda \Gamma^T\Gamma)^{-1}A^Tq^e$. There exists an invertible matrix $B$, such that $\Gamma B=I$, $I$ is identity matrix. Let $\hat{p^*}$ be the solution of $B^T(A^TA+\lambda\Gamma^T\Gamma)B\hat{p^*}=B^TA^Tq^e$. Denote $Q=AB$, we have
$(Q^TQ+\lambda I)\hat{p^*}=Q^Tq^e$. Hence, $\hat{p^*}$ satisfies $\min\limits_{\hat{p}} \|Q\hat{p}-q^e\|^2 + \lambda\|\hat{p}\|^2$, and $B\hat{p^*}=p^*$. We will conduct a stochastic convergence analysis on this method.

\begin{assumption}\label{ass-1} We assume that the eigenvalues, $\rho_1\geq\rho_2\geq\cdots$, of the eigenvalue problem
\begin{align*}
(Q\varphi,Qu)=\rho(\varphi,u),~~\forall u\in R^n,    
\end{align*}
satisfy that $\rho_k\leq Ck^{-\alpha}$, $k=1,2,\cdots,n$ and the corresponding eigenvectors form an orthonormal basis respect to the inner product.
\end{assumption}

\begin{theorem}\label{Th-sto-1}
Let Assumption \ref{ass-1} be fulfilled, and $\hat{p^*}$ be the unique solution of 
\begin{align}\label{Q-mini-1}
\min\limits_{\hat{p}\in R^n} \|Q\hat{p}-q^e\| + \lambda\|\hat{p}\|^2.
\end{align}
Then there exists constants $\lambda_0>0$ and $C>0$ such that for any $\lambda\leq\lambda_0$,
\begin{align*}
\mathbb{E}\big[\|Q\hat{p^*}-Q\bar{p^*}\|^2\big]\leq C\lambda\|\bar{p^*}\|^2+C\frac{\sigma^2+\sum_{i=1}^n\epsilon_i^2}{\lambda^{1/\alpha}},
\end{align*}
where $q=Q\bar{p^*}+\epsilon$, $\Gamma\bar{p}=\bar{p^*}$.
\end{theorem}
\begin{proof} 
By deriving the necessary condition of the quadratic minimization (\ref{Q-mini-1}), we can readily see that the unique minimizer $\hat{p^*}$ satisfies the
variational equation
\begin{align}\label{Sto-err-1}
(Q\hat{p^*},Qu)+\lambda(\hat{p^*},u)=(q^e,Qu),~~\forall u\in R^n.
\end{align}
For any $u\in R^n$, we introduce $\| u \|_{\lambda}^2:=\|Qu\|^2+\lambda\|u\|^2$. By taking $u=\hat{p^*}-\bar{p^*}$ in (\ref{Sto-err-1}), one has
\begin{align}\label{Sto-err-3}
\| \hat{p^*}-\bar{p^*}\|_{\lambda}\leq \lambda^{1/2}\|\bar{p^*}\|+\sup\limits_{u\in R^n}\frac{|(\epsilon+e,Qu)|}{\| u \|_{\lambda}}.
\end{align}

Let $G=\sqrt{Q^TQ}$, $G$ is a positive definite matrix. Denote its eigenvalues system $(\hat{\rho_k},\varphi_k)$, where $\{\varphi_k\}_{k=1}^n$ is an orthonormal basis.
For $\forall u\in R^n$, $u=\sum_{k=1}^n(u,\varphi_k)\varphi_k$, $Qu=\sum_{k=1}^n(u,\varphi_k)Q\varphi_k$. It gives that
\begin{align}\label{Sto-err-4}
(u,u)=\sum_{k=1}^n(u,\varphi_k)^2(\varphi_k,\varphi_k).
\end{align}
\begin{align}\label{Sto-err-5}
(Qu,Qu)
&=\bigg(\sum_{k=1}^n(u,\varphi_k)Q\varphi_k,\sum_{k=1}^n(u,\varphi_k)Q\varphi_k\bigg)\\ \nonumber
&=\sum_{k=1}^n(u,\varphi_k)^2(G\varphi_k,G\varphi_k)\\ \nonumber
&=\sum_{k=1}^n(u,\varphi_k)^2\hat{\rho_k}^2(\varphi_k,\varphi_k),
\end{align}
where we use $(Q\varphi_i,Q\varphi_j)=(G\varphi_i,G\varphi_j)=0$, $i\neq j$. Denote 
$u_k=(u,\varphi_k)$, $\rho_k=\hat{\rho_k}^2$. Let $(\varphi_k,\varphi_k)=1/\rho_k$, $k=1,\cdots,n$. We get $(u,u)=\sum_{k=1}^n\rho_k^{-1}u_k^2$, $(Qu,Qu)=\sum_{k=1}^nu_k^2$.
From (\ref{Sto-err-4}) and (\ref{Sto-err-5}), $\| u\|_{\lambda}^2=\sum_{k=1}^n(\lambda\rho_k^{-1}+1)u_k^2$.

By the Cauchy-Schwarz inequality, it holds that
\begin{align*}
(\epsilon+e,Qu)^2
&=\bigg(\sum_{i=1}^n(e_i+\epsilon_i)\bigg(\sum_{k=1}^nu_k(Q\varphi_k)_i\bigg)\bigg)^2\\
&=\bigg(\sum_{k=1}^nu_k\bigg(\sum_{i=1}^n(e_i+\epsilon_i)(Q\varphi_k)_i\bigg)\bigg)^2\\
&\leq \sum_{k=1}^n(1+\lambda\rho_k^{-1})u_k^2\sum_{k=1}^n(1+\lambda\rho_k^{-1})^{-1}\bigg(\sum_{i=1}^n(e_i+\epsilon_i)(Q\varphi_k)_i\bigg)^2\\
&\leq \| u\|_{\lambda}^2\sum_{k=1}^n(1+\lambda\rho_k^{-1})^{-1}\bigg(\sum_{i=1}^n(e_i+\epsilon_i)(Q\varphi_k)_i\bigg)^2,
\end{align*}
that is
\begin{align*}
\frac{(\epsilon+e,Qu)^2}{\| u\|_{\lambda}^2}
\leq \sum_{k=1}^n(1+\lambda\rho_k^{-1})^{-1}\bigg(\sum_{i=1}^n(e_i+\epsilon_i)(Q\varphi_k)_i\bigg)^2.
\end{align*}
For $\big(\sum_{i=1}^n(e_i+\epsilon_i)(Q\varphi_k)_i\big)^2$, we have
\begin{align*}
\bigg(\sum_{i=1}^n(e_i+\epsilon_i)(Q\varphi_k)_i\bigg)^2
&=\bigg(\sum_{i=1}^ne_i(Q\varphi_k)_i+\sum_{i=1}^n\epsilon_i(Q\varphi_k)_i\bigg)^2\\
&\leq 2\bigg(\sum_{i=1}^ne_i(Q\varphi_k)_i\bigg)^2+2\bigg(\sum_{i=1}^n\epsilon_i(Q\varphi_k)_i\bigg)^2\\
&\leq 2\bigg(\sum_{i=1}^ne_i(Q\varphi_k)_i\bigg)^2+2\sum_{i=1}^n\epsilon_i^2\sum_{i=1}^n(Q\varphi_k)_i^2\\
&\leq 2\bigg(\sum_{i=1}^ne_i(Q\varphi_k)_i\bigg)^2+2\sum_{i=1}^n\epsilon_i^2,
\end{align*}
where $\sum_{i=1}^n(Q\varphi_k)_i^2=(Q\varphi_k,Q\varphi_k)=1$.
Then, we get
\begin{align*}
\mathbb{E}\bigg[\bigg(\sum_{i=1}^n(e_i+\epsilon_i)(Q\varphi_k)_i\bigg)^2\bigg]
&\leq 2\mathbb{E}\bigg[\bigg(\sum_{i=1}^ne_i(Q\varphi_k)_i\bigg)^2\bigg]+2\sum_{i=1}^n\epsilon_i^2\\&
\leq 2\sum_{i=1}^n\mathbb{E}[e_i^2](Q\varphi_k)_i^2+2\sum_{i=1}^n\epsilon_i^2\\
&\leq 2\sigma^2\sum_{i=1}^n(Q\varphi_k)_i^2+2\sum_{i=1}^n\epsilon_i^2\\
&\leq 2\sigma^2+2\sum_{i=1}^n\epsilon_i^2,
\end{align*}
the last estimate holds from the fact random variables $\{e_i\}_{i=1}^n$ are independent and identically distributed, i.e., $\mathbb{E}[e_ie_j]=0$, $i\neq j$. Hence, one has
\begin{align*}
\mathbb{E}\bigg[\frac{(\epsilon+e,Qu)^2}{\| u\|_{\lambda}^2}\bigg]
\leq 2\sum_{k=1}^n(1+\lambda\rho_k^{-1})^{-1}\bigg(\sigma^2+\sum_{i=1}^n\epsilon_i^2\bigg).
\end{align*}
By Assumption \ref{ass-1}, it is easy to verify that
\begin{align*}
\sum_{k=1}^n(1+\lambda\rho_k^{-1})^{-1}
&\leq \sum_{k=1}^n(1+\lambda k^{\alpha})^{-1}\\
&\leq \int_1^{\infty}(1+\lambda t^{\alpha})^{-1}dt\\
&\leq \lambda^{-1/\alpha}\int_{\lambda^{1/\alpha}}^{\infty}(1+ s^{\alpha})^{-1}ds\\
&\leq C\lambda^{-1/\alpha}.
\end{align*}
It arrvies 
\begin{align}\label{Sto-err-7}
\mathbb{E}\bigg[\frac{(\epsilon+e,Qu)^2}{\| u\|_{\lambda}^2}\bigg]
\leq C\lambda^{-1/\alpha}\bigg(\sigma^2+\sum_{i=1}^n\epsilon_i^2\bigg).
\end{align}
From (\ref{Sto-err-3}) and (\ref{Sto-err-7}), we have
\begin{align*}
\mathbb{E}\big[\| \hat{p^*}-\bar{p^*}\|_{\lambda}^2\big]\leq C\lambda\|\bar{p^*}\|^2+C\lambda^{-1/\alpha}\bigg(\sigma^2+\sum_{i=1}^n\epsilon_i^2\bigg).
\end{align*}
The proof is complete. 
\end{proof} 

\begin{remark}
By Theorem \ref{Th-sto-1}, we have
\begin{align*}
\mathbb{E}\big[\|Ap^*-A\bar{p}\|^2\big]\leq C\lambda\|\Gamma\bar{p}\|^2+C\frac{\sigma^2+\sum_{i=1}^n\epsilon_i^2}{\lambda^{1/\alpha}}.
\end{align*}
$A_{m\times n}$ is design matrix in the fitting problems. In general, the number of measurement data $m$ exceeds the number of control points $n$. $\Gamma_{n\times n}$ is a positive definite matrix in the term of regularization. For all $q\in R^d$, denote $(q,q)_d=\frac{1}{d}\|q\|^2:=\|q\|_d^2$. We have the following mean squared error:
\begin{align*}
\mathbb{E}\big[\|Ap^*-A\bar{p}\|_m^2\big]\leq C\frac{n\lambda}{m}\|\Gamma\bar{p}\|_n^2+C\frac{\sigma^2+\sum_{i=1}^n\epsilon_i^2}{m\lambda^{1/\alpha}}.
\end{align*}
When $\sum_{i=1}^n\epsilon_i^2<\sigma^2$, an optimal choice of the parameter is such that $\lambda^{1+1/\alpha}=O(\sigma^2n^{-1})\|\Gamma\bar{p}\|_n^{-2}$.
\end{remark}

\section{Self-consistent iterative algorithms without prior information}\label{section6}
In the previous section, we gave an estimate of the optimal regularization parameter. However, in practical applications, the noise level $\sigma^2$ and the control points $\bar{p}$ are unknown. Therefore, we further propose an algorithm to find the optimal regularization $\lambda$. A direct idea is to find $\lambda$ by using a fixed-point iteration method. 

\begin{table}[htbp]
\small
\centering
\caption{Optimal regularization parameter iterative algorithm for regularized RPIA curve fitting.}
\begin{tabular}{|p{0.96\linewidth}|}
\hline
\textbf{An algorithm for finding optimal regularization parameter and recovering the control points.} \\
\hline
\textbf{1. Input:} Observation data $q_\mathrm{noise}$, number of control points $n$, number of data points $m$, convergence threshold $\varepsilon_\lambda$, spectral decay parameter $\alpha$. \\
\textbf{2.} Set up initial guess $\lambda_1$ such that $\lambda_1^{1+1/\alpha} = n^{-1}$. \\
\textbf{3.} Solve for control points $p^{(1)}$ by minimizing \\
\hspace*{3em} $\min\limits_{p} \|A p - q_\mathrm{noise}\|_F^2 + \lambda_1 \|\Gamma p\|_F^2$. \\
\textbf{4.} Update parameter: \\
\hspace*{3em} $\lambda_2^{1+\frac{1}{\alpha}} = \dfrac{\|A p^{(1)} - q_\mathrm{noise}\|_m^2}{\|\Gamma p^{(1)}\|_n^2} n^{-1}$. \\
\textbf{5.} Set $k = 2$. \\
\textbf{6.} \textbf{while} $|\lambda_k - \lambda_{k-1}| > \varepsilon_\lambda \lambda_{k-1}$ \textbf{do} \\
\hspace*{2em} 6.1 Solve for $p^{(k)}$ by minimizing \\
\hspace*{3em} $\min\limits_{p} \|A p - q_\mathrm{noise}\|_F^2 + \lambda_k \|\Gamma p\|_F^2$. \\
\hspace*{2em} 6.2 Update parameter: \\
\hspace*{3em} $\lambda_{k+1}^{1+\frac{1}{\alpha}} = \dfrac{\|A p^{(k)} - q_\mathrm{noise}\|_m^2}{\|\Gamma p^{(k)}\|_n^2} n^{-1}$ \\
\hspace*{2em} 6.3 $k \leftarrow k + 1$. \\
\textbf{7.} \textbf{end while} \\
\textbf{8. Output:} regularization parameter $\lambda$,  control points $p$. \\
\hline
\end{tabular}
\end{table}

\begin{table}[htbp]
\small
\centering
\caption{Optimal regularization parameter iterative algorithm for regularized RPIA  surface  fitting.}
\begin{tabular}{|p{0.96\linewidth}|}
\hline
\textbf{An algorithm for finding optimal regularization parameter and recovering the control points.} \\
\hline
\textbf{1. Input:} Observation data $Q_\mathrm{noise}$, number of control points $n$, number of data points $m$, convergence threshold $\varepsilon_\lambda$, spectral decay parameter $\alpha$. \\
\textbf{2.} Set up initial guess $\lambda_1$ such that $\lambda_1^{1+1/\alpha} = n^{-1}$. \\
\textbf{3.} Solve for control points $P^{(1)}$ by minimizing \\
\hspace*{3em} $\min\limits_{P} \|A PB^T - Q_\mathrm{noise}\|_F^2 + \lambda_1\left\|{A P} {L}_v^T\right\|_F^2+\lambda_1\left\|{L}_{{u}} {P} {B}^T\right\|_F^2+\lambda_1^2\left\|{L}_{{u}} {P} {L}_{{v}}^T\right\|_F^2$. \\
\textbf{4.} Update parameter: \\
\hspace*{3em} $\lambda_2^{1+\frac{1}{\alpha}} = \dfrac{\|A P^{(1)}B^T - Q_\mathrm{noise}\|_m^2}{\left\|{A P^{(1)}} {L}_v^T\right\|_n^2+\left\|{L}_{{u}} {P}^{(1)} {B}^T\right\|_n^2} n^{-1}$. \\
\textbf{5.} Set $k = 2$. \\
\textbf{6.} \textbf{while} $|\lambda_k - \lambda_{k-1}| > \varepsilon_\lambda \lambda_{k-1}$ \textbf{do} \\
\hspace*{2em} 6.1 Solve for $P^{(k)}$ by minimizing \\
\hspace*{3em} $\min\limits_{P} \|A PB^T - Q_\mathrm{noise}\|_F^2 + \lambda_k\left\|{A P} {L}_v^T\right\|_F^2+\lambda_k\left\|{L}_{{u}} {P} {B}^T\right\|_F^2+\lambda_k^2\left\|{L}_{{u}} {P} {L}_{{v}}^T\right\|_F^2$. \\
\hspace*{2em} 6.2 Update parameter: \\
\hspace*{3em} $\lambda_{k+1}^{1+\frac{1}{\alpha}} = \dfrac{\|A P^{(k)}B^T - Q_\mathrm{noise}\|_m^2}{\left\|{A P^{(k)}} {L}_v^T\right\|_n^2+\left\|{L}_{{u}} {P}^{(k)} {B}^T\right\|_n^2} n^{-1}$. \\
\hspace*{2em} 6.3 $k \leftarrow k + 1$. \\
\textbf{7.} \textbf{end while} \\
\textbf{8. Output:} regularization parameter $\lambda$, control points $P$. \\
\hline
\end{tabular}
\end{table}
\section{Numerical experiment}
In this section, we give several numerical examples and implement the regularized RPIA algorithm to fit curves and surfaces, and further verify the correctness of the stochastic optimal estimate for the regularization method. We use the cubic B-spline basis because it is simple and widely used in computer-aided design\cite{DENG201432,EBRAHIMI20191,HUANG2020101931,ZHU2024113436}.

\subsection{The case of curves}\label{section7.1}

For consistency, we adhere to the implementation framework outlined by Liu and Wu\cite{WU2024128669} when performing curve fitting, specifically structured as follows. We determine the parameters $\left\{x_j\right\}_{j=0}^m$ corresponding to the points $\left\{q_j\right\}_{j=0}^m$ using the normalized accumulated chord length parametrization, computed as:

$$
x_0=0, x_m=1, \text { and } x_j=x_{j-1}+\frac{\left\|\boldsymbol{q}_j-\boldsymbol{q}_{j-1}\right\|}{\sum_{s=1}^m\left\|\boldsymbol{q}_s-\boldsymbol{q}_{s-1}\right\|},
$$
for $j=1,2, \cdots, m$. The knot vector of cubic B-spline basis is defined by
$$
\left\{0,0,0,0, \bar{x}_4, \bar{x}_5, \cdots, \bar{x}_{n_1}, 1,1,1,1\right\},
$$
with
$$
\bar{x}_{j+3}=(1-\alpha) x_{i-1}+\alpha x_i,
$$
where $i=\lfloor j d\rfloor, ~\alpha=j d-i, ~d=(m+1) /\left(n_1-2\right)$ for $j=1,2, \cdots, n_1-3$. Here, the notation $\lfloor\cdot\rfloor$ represents the greatest integer less than or equal to the given value. For $i=1,2, \cdots, n_1-1$, the initial control points are selected by $p_i^{(0)}=q_{\left\lfloor m i / n_1\right\rfloor}$ for $i \in\left[n_1\right]$. 

Next, we give the following settings: the numbers of data and control points are $m=1000$ and $n_1=100$, the block size is 5. The algorithm terminates when:
$$
\frac{\left\|A p^{(k)}-A p^{(k-1)}\right\|_F}{\|Ap^{(k-1)}\|_F}<10^{-8},
$$
or the maximum number of iterations (8000) is reached. Define the noise data ${\boldsymbol{q^\delta}}$, which is generated from the exact data $\boldsymbol{q}$ as
$$
\boldsymbol{q^\delta}=\boldsymbol{q}+10 \frac{\widetilde{\boldsymbol{q}}}{\|\widetilde{\boldsymbol{q}}\|_F},
$$
with the random variable $\widetilde{\boldsymbol{q}}$ following an i.i.d. standard Gaussian distribution.
$\Gamma$ denotes the second-order difference matrix with Dirichlet boundary conditions, which is explicitly given by:
\[
\Gamma = 
C\begin{pmatrix}
-2 & 1 & 0 & 0 & \cdots & 0 & 0 \\
1 & -2 & 1 & 0 & \cdots & 0 & 0 \\
0 & 1 & -2 & 1 & \cdots & 0 & 0 \\
\vdots & \ddots & \ddots & \ddots & \ddots & \vdots & \vdots \\
0 & 0 & \cdots & 1 & -2 & 1 & 0 \\
0 & 0 & \cdots & 0 & 1 & -2 & 1 \\
0 & 0 & \cdots & 0 & 0 & 1 & -2
\end{pmatrix}_{(n_1+1) \times (n_1+1),}
\]
where $C$ is a positive constant. Taking $C=1600$ in Examples \ref{example 7.1} and \ref{example 7.2}.
\begin{example}
\label{example 7.1}
 m+1 data points sampled uniformly from a rose-type curve, whose polar coordinate equation is
$$\boldsymbol{r}=\sin (\theta / 4) \quad(0 \leq \theta \leq 8 \pi).$$
\end{example}

\begin{example}
\label{example 7.2}
 m+1 data points sampled uniformly from a blob-shaped curve, whose polar coordinate equation is
$$\boldsymbol{r}=1 + 2\cos (2\theta +1/2)+2\cos(3\theta+1/2) \quad(0 \leq \theta \leq 2 \pi).$$
\end{example}

\begin{figure}[htbp]
    \centering
    \begin{subfigure}[t]{0.50\textwidth}
        \centering
        \includegraphics[width=\linewidth]{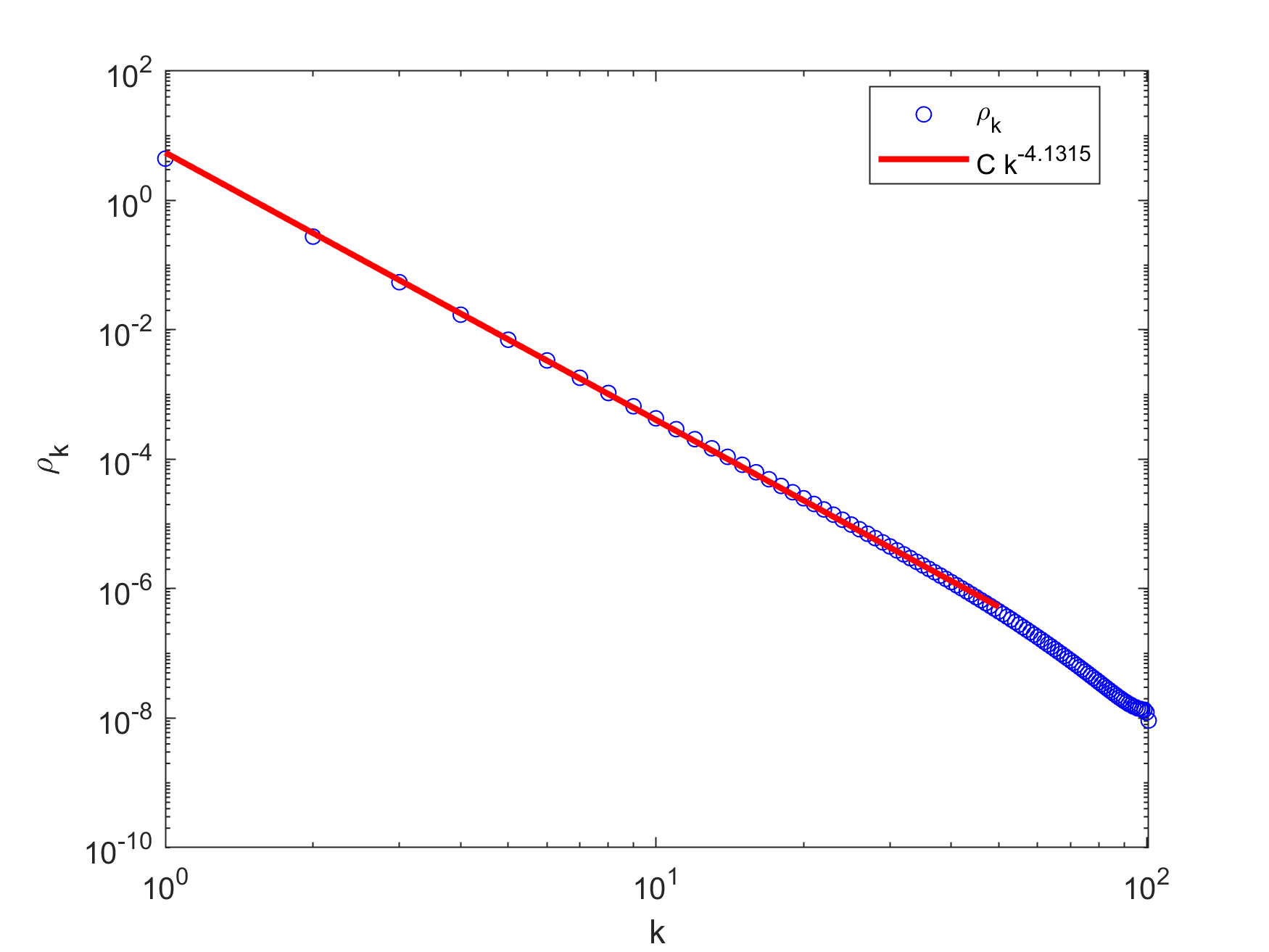}
    \end{subfigure}
    \caption{Spectrum of $Q^TQ$ for Example \ref{example 7.1}}
    \label{Fig:2}
\end{figure}

\begin{figure}[htbp]
    \centering
    \begin{subfigure}[t]{0.48\textwidth}
        \centering
        \includegraphics[width=\linewidth]{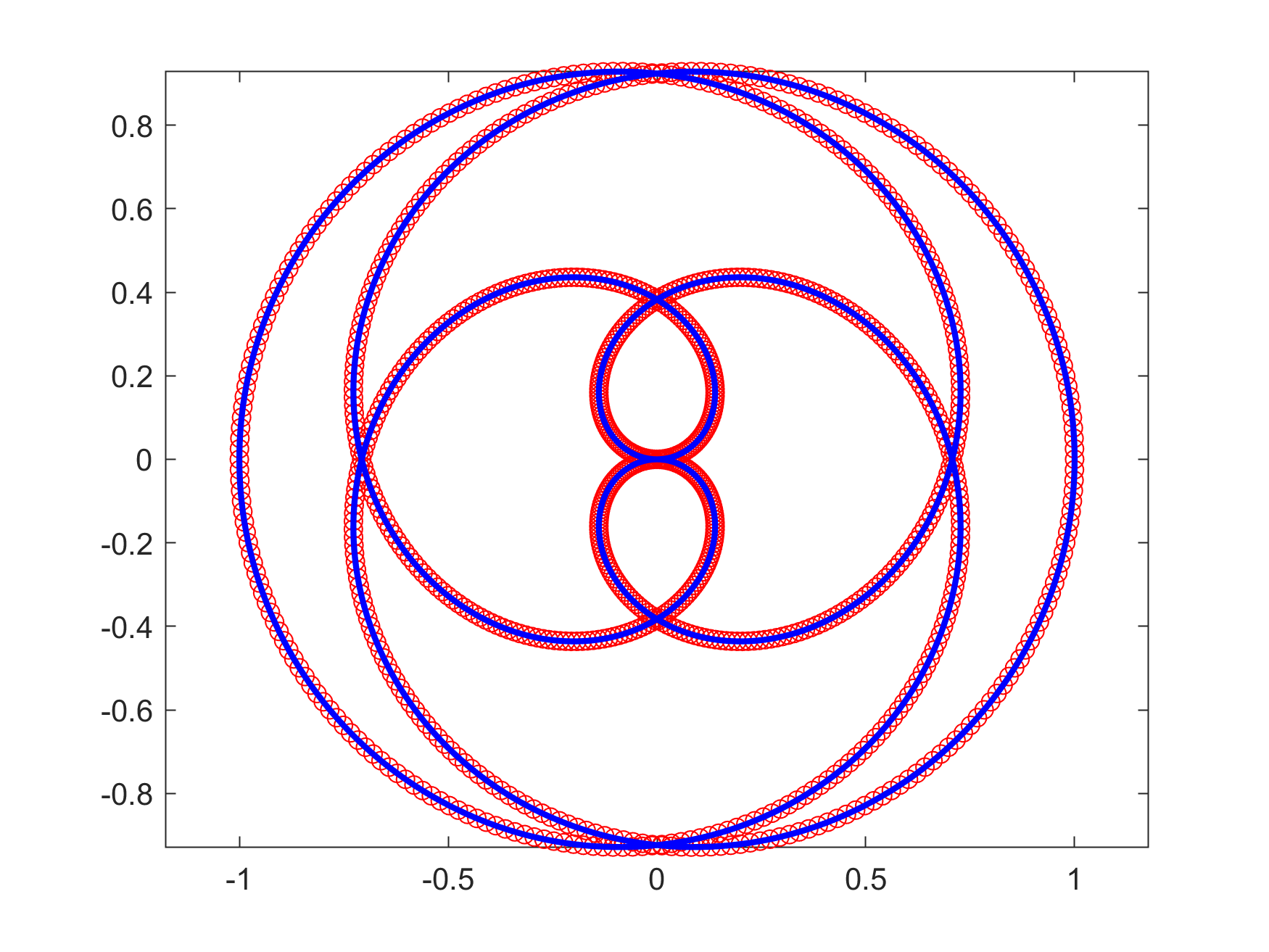}
        \caption{Example \ref{example 7.1} Real Curve}
    \end{subfigure}
    \hfill
    \begin{subfigure}[t]{0.48\textwidth}
        \centering
        \includegraphics[width=\linewidth]{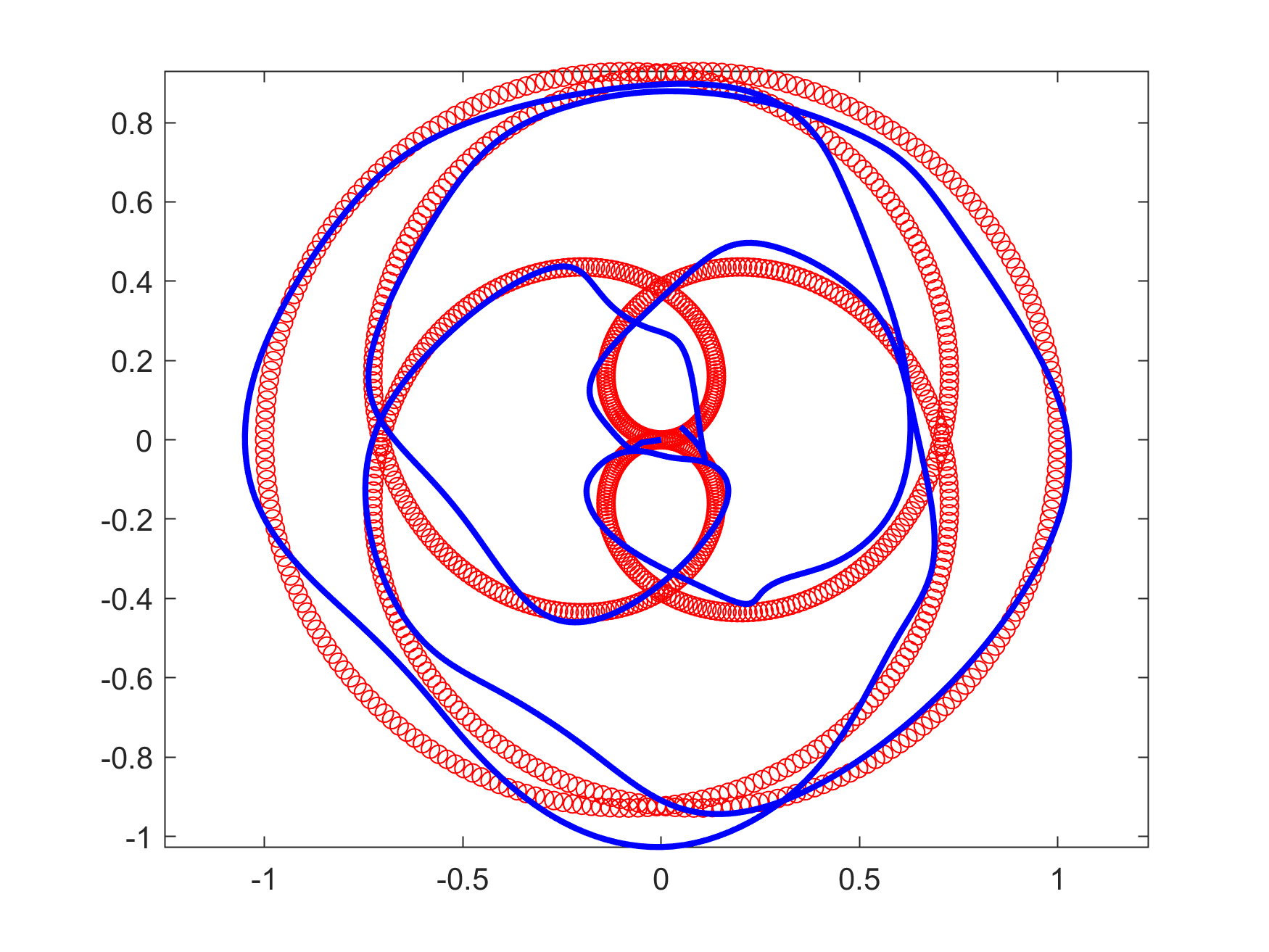}
        \caption{Regularized RPIA Fitted Curve}
    \end{subfigure}
    \vskip\baselineskip 
    \begin{subfigure}[t]{0.48\textwidth}
        \centering
        \includegraphics[width=\linewidth]{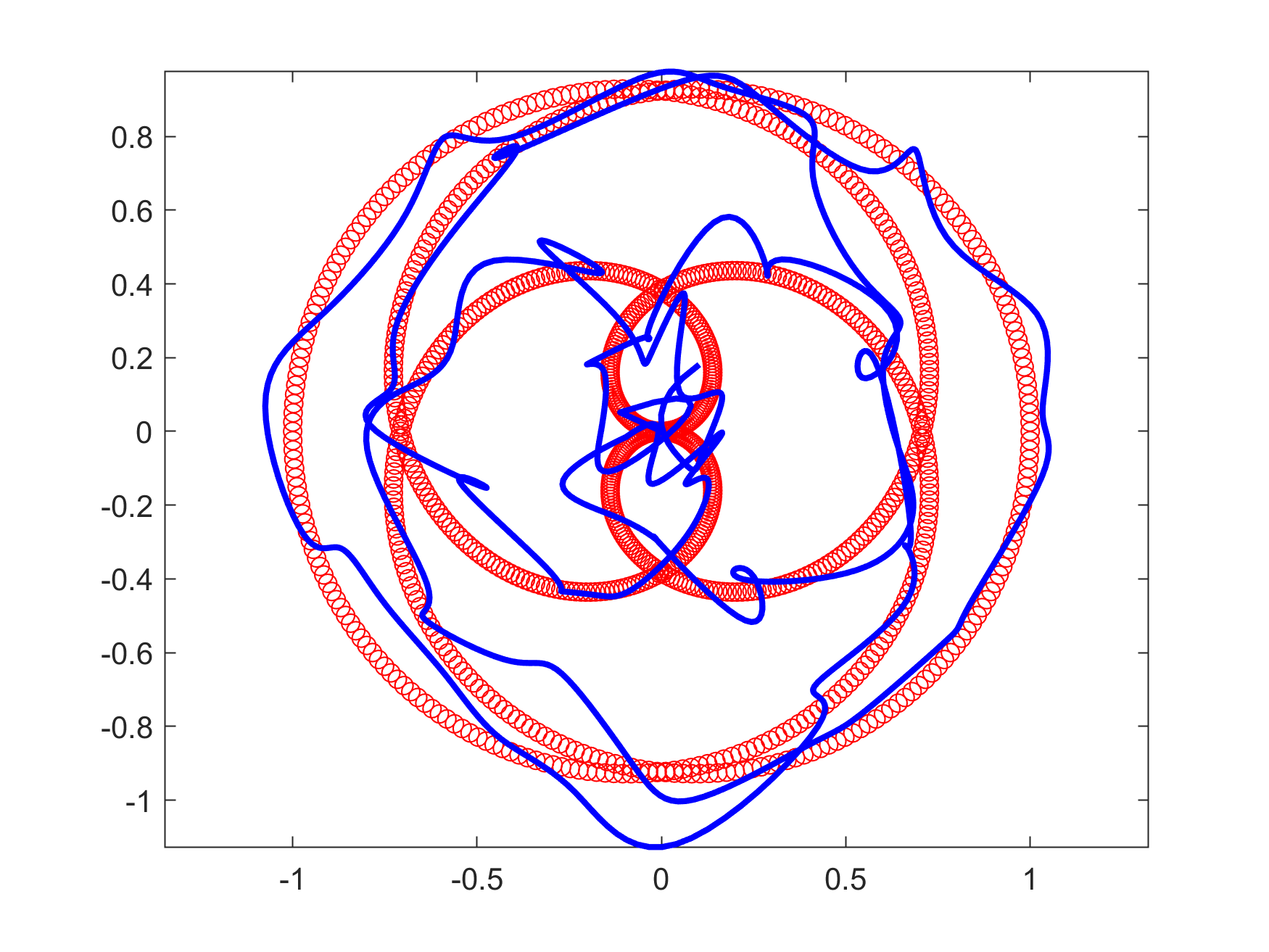}
        \caption{Regularized RPIA Fitted Curve($\lambda=0$)}
    \end{subfigure}
    \hfill
    \begin{subfigure}[t]{0.48\textwidth}
        \centering
        \includegraphics[width=\linewidth]{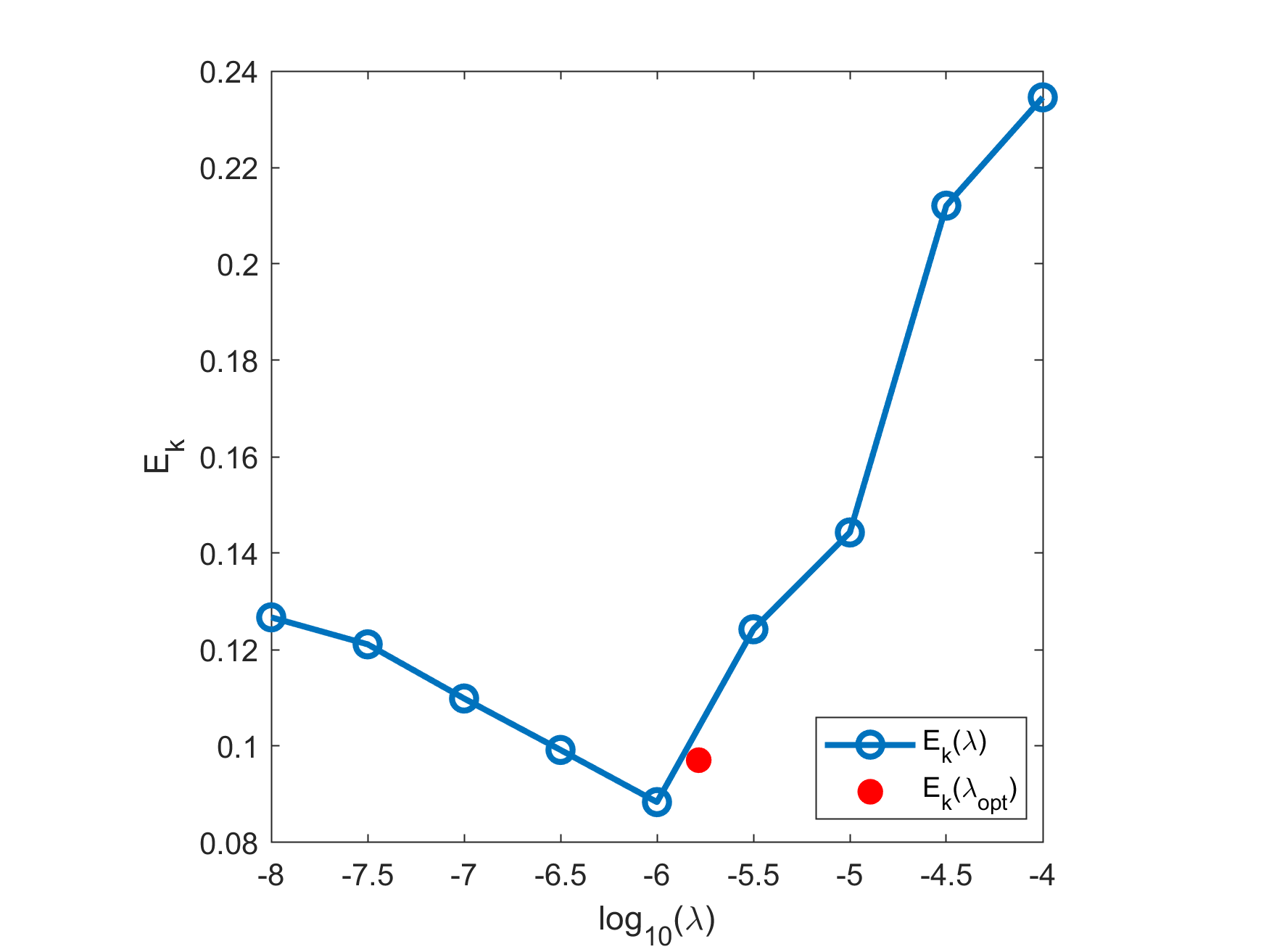}
        \caption{Fitting Error}
    \end{subfigure}
    \caption{The fitting curves given by Regularized RPIA for Example \ref{example 7.1}}
    \label{Fig:1}
\end{figure}

From the results in Figure \ref{Fig:1}, we can see that the regularization effect is very obvious. Specifically, the blue lines are the fitting curves and the red circles are the noise-free data points. Figure (a) represents the real curve, which perfectly fits the noise-free data points; Figure (b) represents the fitting curve with the optimal parameter ($\lambda$=1.646e-06) given by our optimal parameter estimation method in Section \ref{section5}, the error $E_k = 0.097077$; Figure (c) represents the curve without regularization ($\lambda=0$), and the error $E_k = 0.137592$; Figure (d) represents the errors as $\lambda$ changes, and the red dot denotes the estimated optimal point, which is very close to the point with the minimum error. This further verifies the correctness of the optimal parameter estimation. The definition of the fitting error is as follows.
$$
E_k(\lambda)=\frac{\left\|Ap^{*}-A\bar{p}\right\|_F^2}{\|A\bar{p}\|_F^2}.
$$

Figure \ref{Fig:2} shows the decay of the eigenvalue of matrix $Q$ in Section \ref{section5}(Select the first 50 eigenvalues), with a decay rate of 4.1315, i.e. $\alpha=4.1315$ in $\lambda^{1+1/\alpha}=O(\sigma^2n^{-1})\|\Gamma\bar{p}\|_n^{-2}$.
The above results are obtained by running the regularized RPIA 10 times in a row and taking the arithmetic mean of the results.

\begin{figure}[htbp]
    \centering
    \begin{subfigure}[t]{0.50\textwidth}
        \centering
        \includegraphics[width=\linewidth]{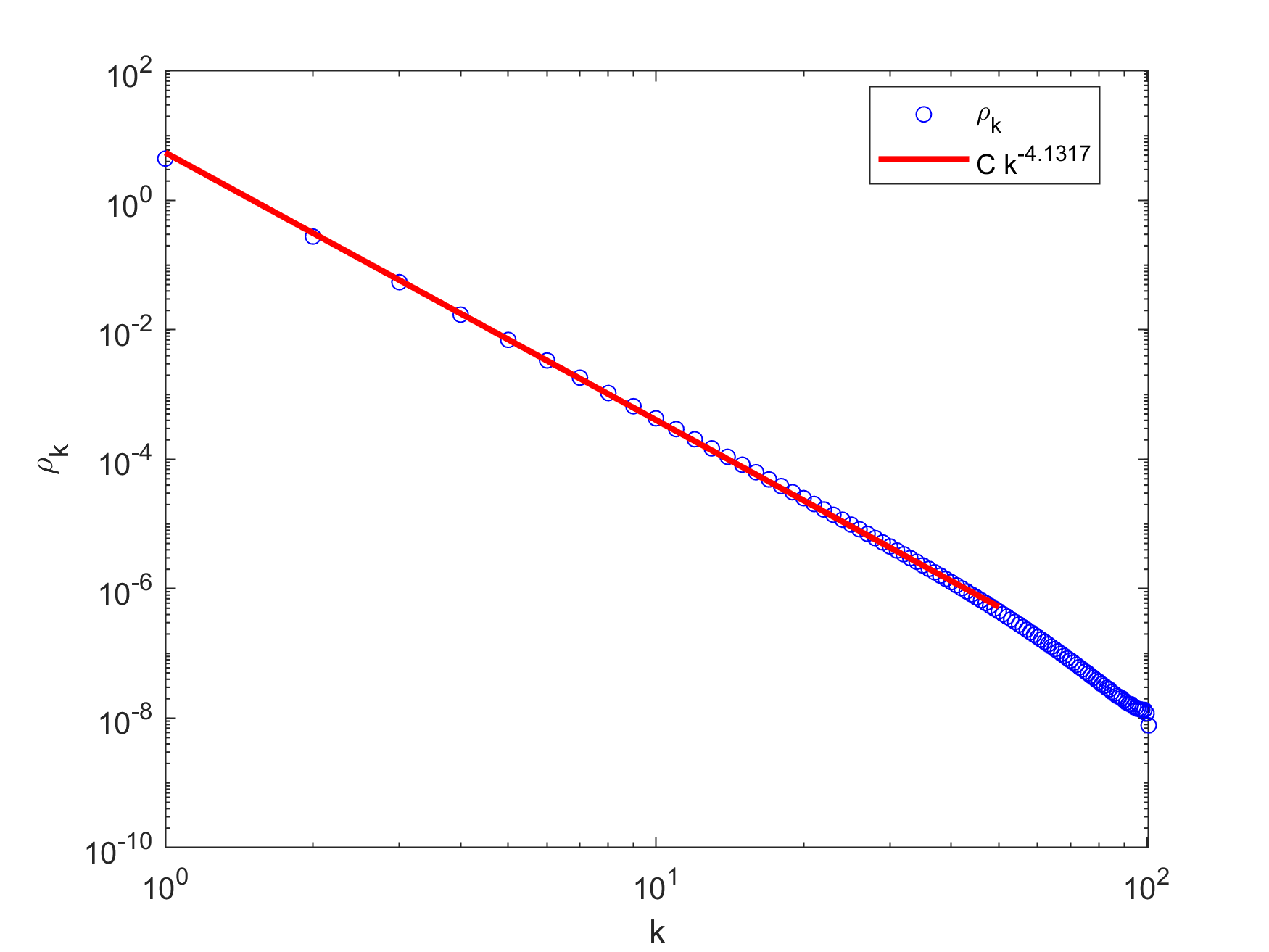}
    \end{subfigure}
    \caption{Spectrum of $Q^TQ$ for Example \ref{example 7.2}}
    \label{Fig:4}
\end{figure}

\begin{figure}[htbp]
    \centering
    \begin{subfigure}[t]{0.48\textwidth}
        \centering
        \includegraphics[width=\linewidth]{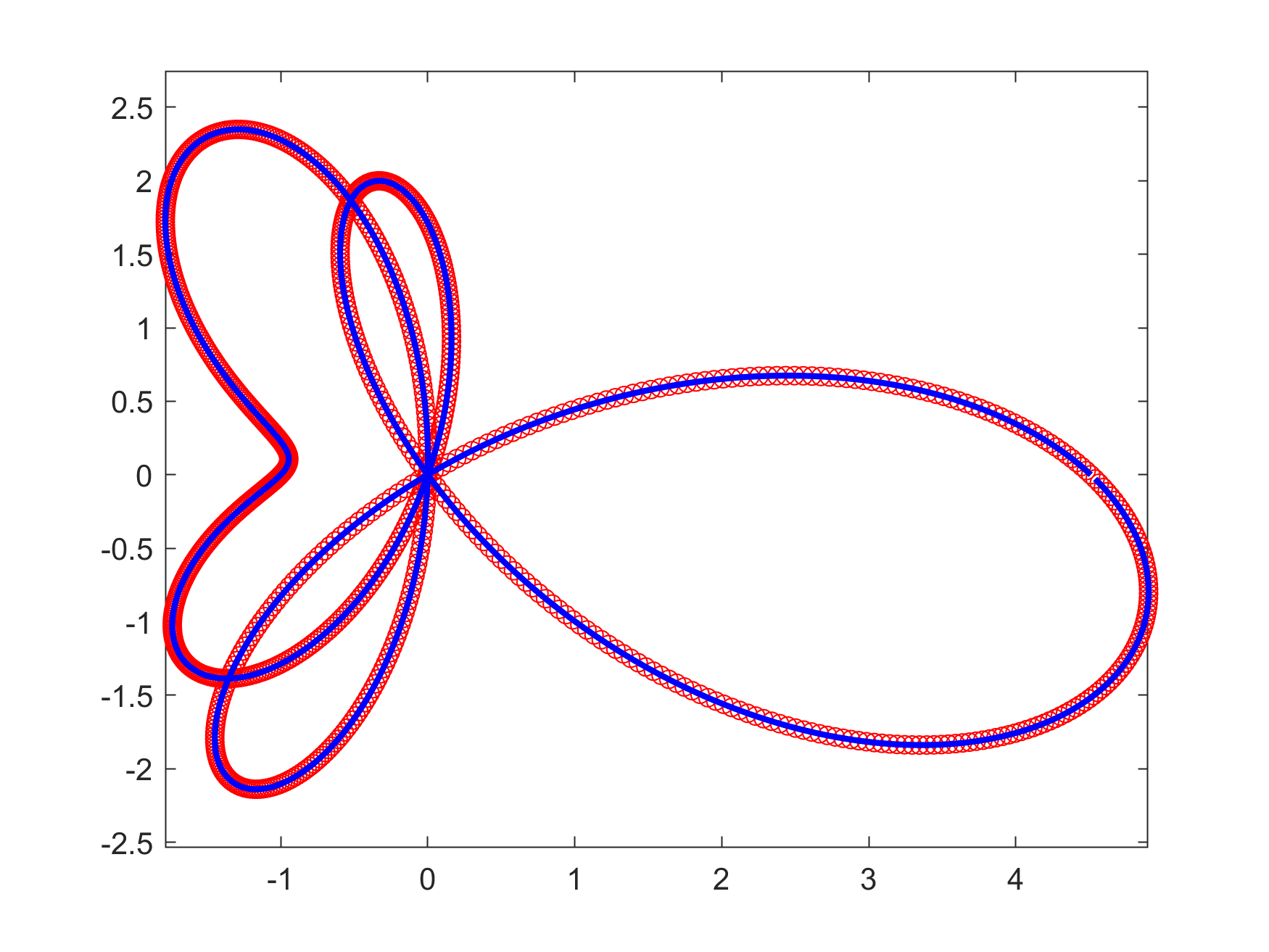}
        \caption{Example \ref{example 7.2} Real Curve}
    \end{subfigure}
    \hfill
    \begin{subfigure}[t]{0.48\textwidth}
        \centering
        \includegraphics[width=\linewidth]{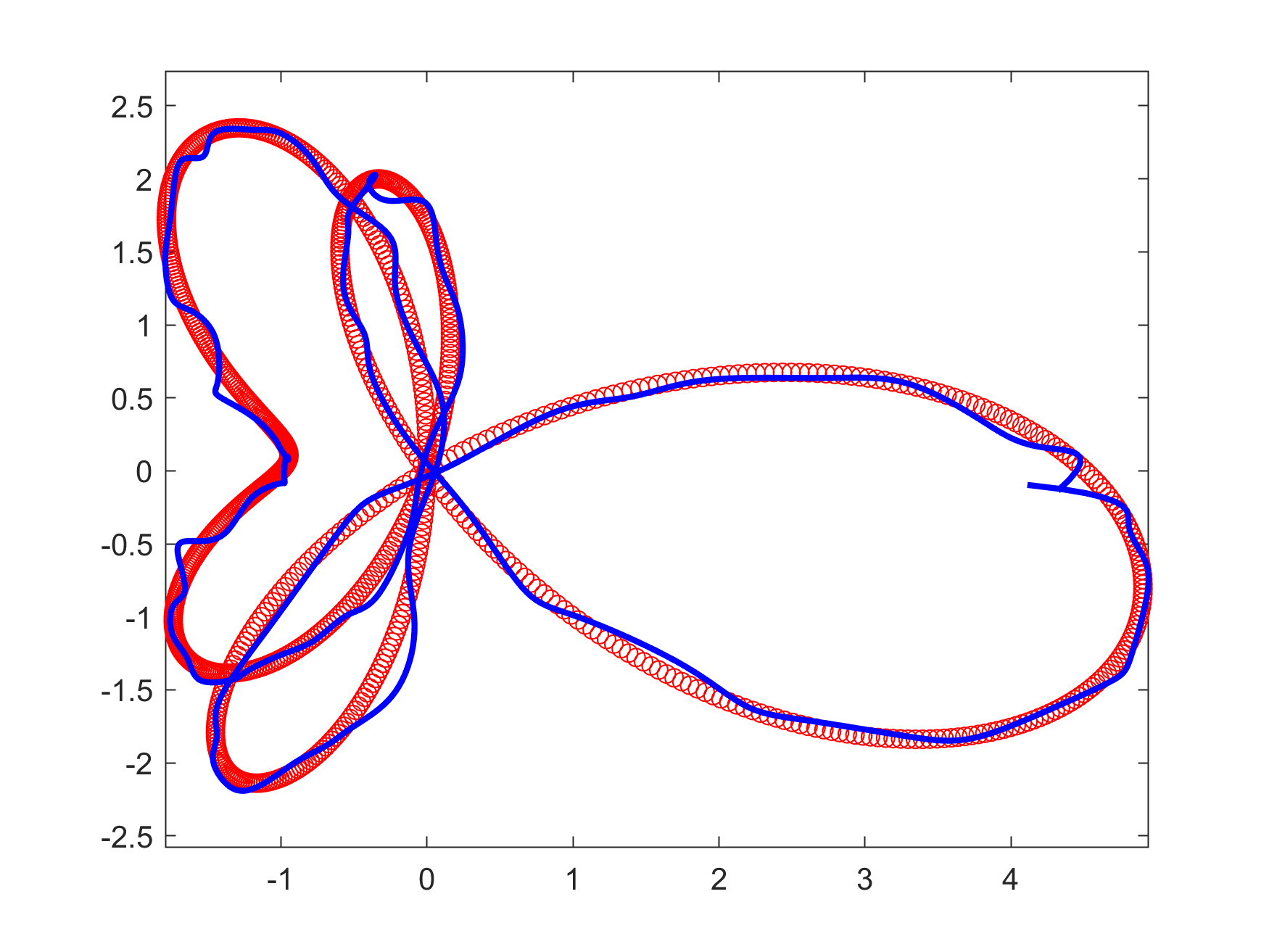}
        \caption{Regularized RPIA Fitted Curve}
    \end{subfigure}
    \vskip\baselineskip 
    \begin{subfigure}[t]{0.48\textwidth}
        \centering
        \includegraphics[width=\linewidth]{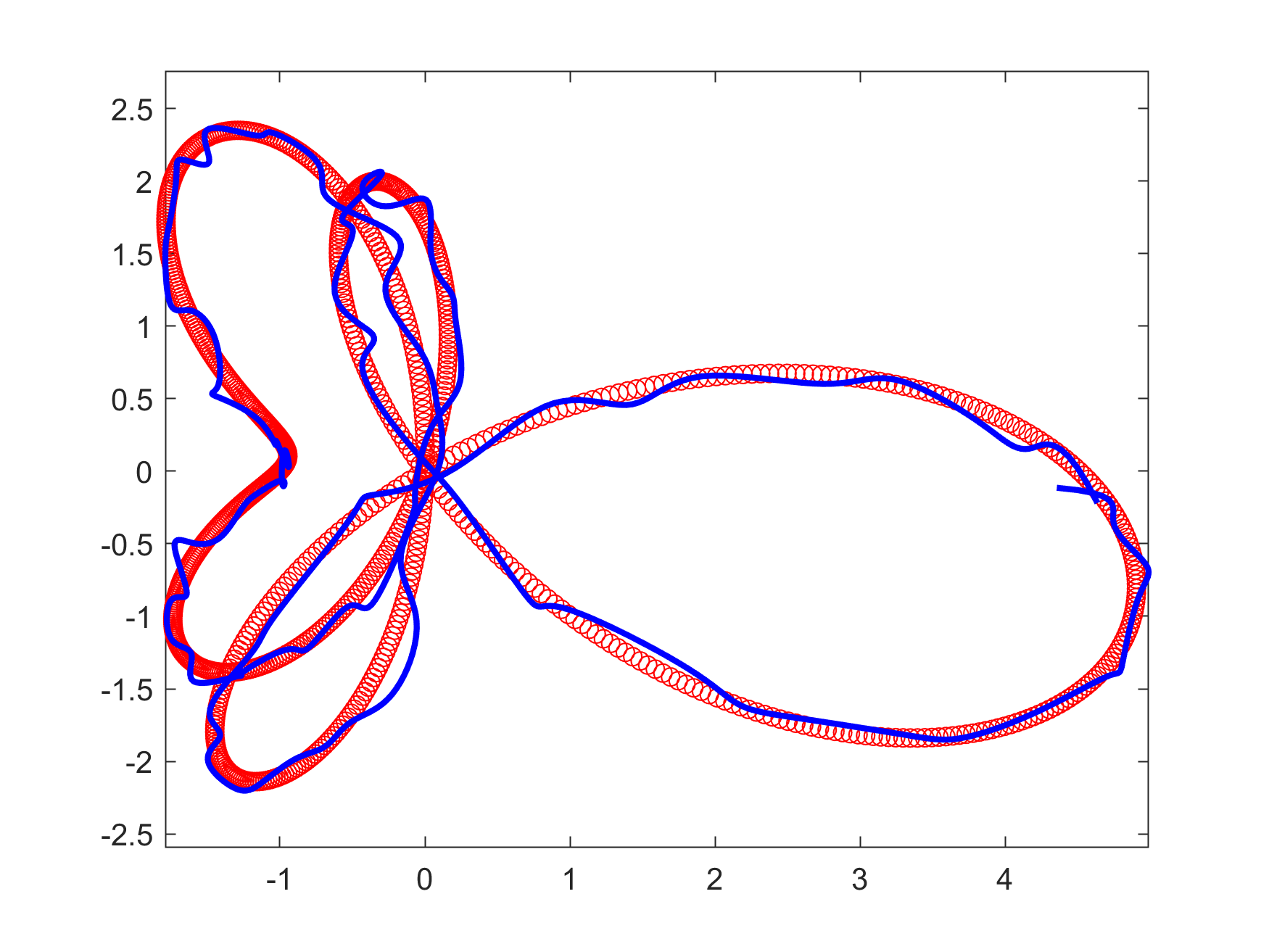}
        \caption{Regularized RPIA Fitted Curve($\lambda=0$)}
    \end{subfigure}
    \hfill
    \begin{subfigure}[t]{0.48\textwidth}
        \centering
        \includegraphics[width=\linewidth]{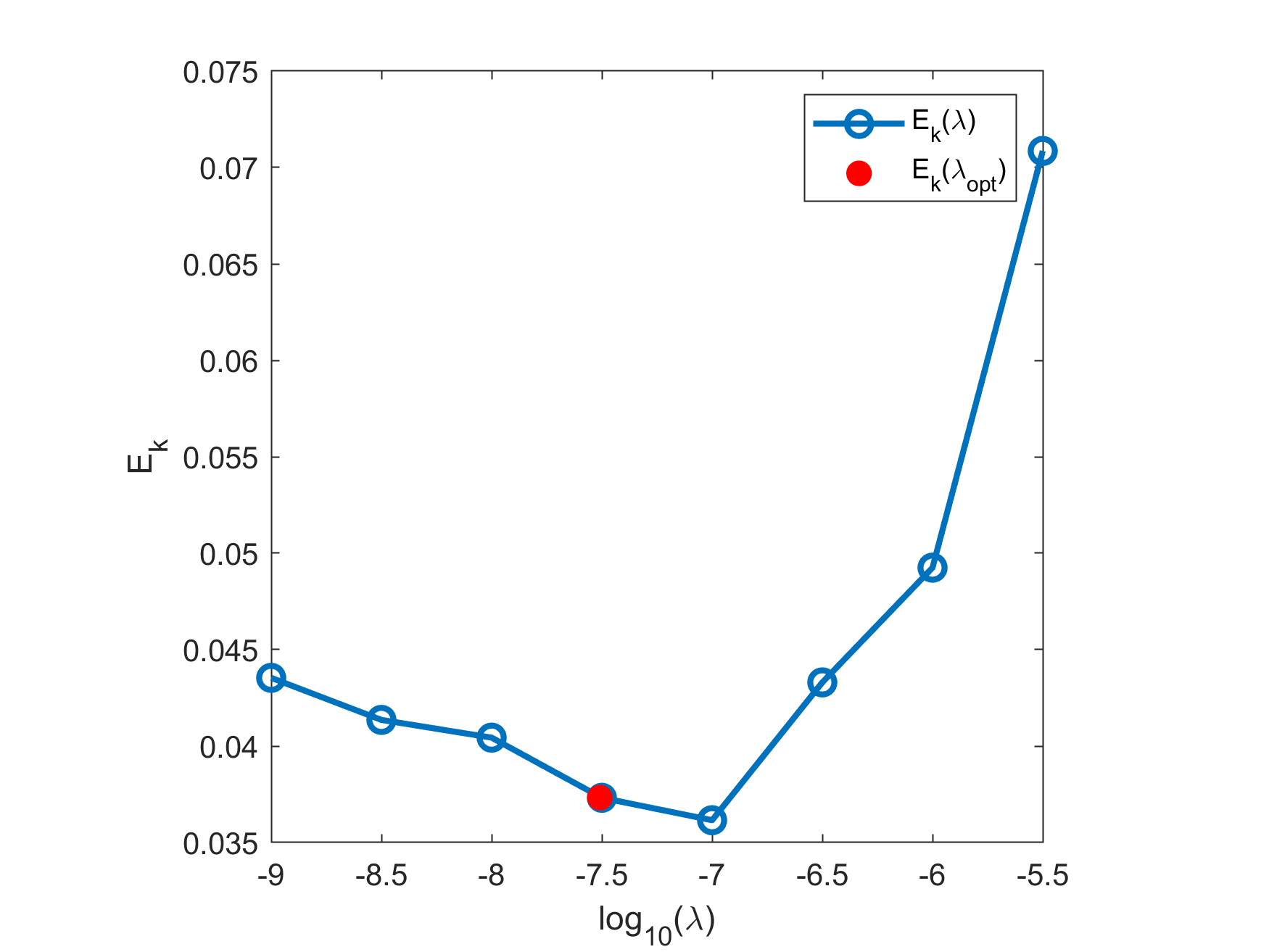}
        \caption{Fitting Error}
    \end{subfigure}    
    \caption{The fitting curves given by Regularized RPIA for Example \ref{example 7.2}}
    \label{Fig:3}
\end{figure}

It can also be seen from Figure \ref{Fig:3} that the regularization method is still better. 
Figure (a) represents the real curve, which perfectly fits the noise-free data points; Figure (b) represents the fitting curve with the optimal parameter ($\lambda$= 3.096e-08) given by our optimal parameter estimation method in Section \ref{section5}, the error $E_k = 0.037331$; Figure (c) represents the curve without regularization ($\lambda=0$), and the error $E_k = 0.044575$; Figure (d) represents the errors as $\lambda$ changes, and the red dot denotes the estimated optimal point, which is very close to the point with the minimum error. 
Figure \ref{Fig:4} shows the decay of the eigenvalue of matrix $Q$ in Section \ref{section5}(Select the first 50 eigenvalues), with a decay rate of 4.1317, i.e. $\alpha=4.1317$ in $\lambda^{1+1/\alpha}=O(\sigma^2n^{-1})\|\Gamma\bar{p}\|_n^{-2}$.
The above results are obtained by running the regularized RPIA 10 times in a row and taking the arithmetic mean of the results.

\subsection{The case of surfaces}\label{section7.2}
Similar to the case of curve fitting, we assign the parameters $\left\{x_h\right\}_{h=0}^m$ and $\left\{y_l\right\}_{l=0}^p$ as $x_0=0, x_m=1$, and
$$
x_h=x_{h-1}+\frac{\sum_{t=0}^p\left\|\boldsymbol{Q}_{h t}-\boldsymbol{Q}_{h-1, t}\right\|}{\sum_{s=1}^m \sum_{t=0}^p\left\|\boldsymbol{Q}_{s t}-\boldsymbol{Q}_{s-1, t}\right\|},
$$
for $h=1,2, \cdots, m$. $y_0=0, y_p=1$, and
$$
y_l=y_{l-1}+\frac{\sum_{s=0}^m\left\|\boldsymbol{Q}_{s l}-\boldsymbol{Q}_{s, l-1}\right\|}{\sum_{h=0}^m \sum_{t=1}^p\left\|\boldsymbol{Q}_{s t}-\boldsymbol{Q}_{s, t-1}\right\|},
$$
for $l=1,2, \cdots, p$, respectively. Define two knot vectors as
$$
\left\{0,0,0,0, \bar{x}_4, \bar{x}_5, \cdots, \bar{x}_{n_1}, 1,1,1,1\right\}
$$
with $\bar{x}_{h+3}=\left(1-\alpha_x\right) x_{i-1}+\alpha_x x_i, i=\left\lfloor h d_x\right\rfloor, \alpha_x=h d_x-i, d_x=(m+1) /\left(n_1-2\right)$ for $h=1,2, \cdots, n_1-3$. And
$$
\left\{0,0,0,0, \bar{y}_4, \bar{y}_5, \cdots, \bar{y}_{n_2}, 1,1,1,1\right\}
$$
with $\bar{y}_{l+3}=\left(1-\alpha_y\right) y_{j-1}+\alpha_y y_j, j=\left\lfloor l d_y\right\rfloor, \alpha_y=l d_y-j, d_y=(p+1) /\left(n_2-2\right)$ for $l=1,2, \cdots, n_2-3$, respectively. The initial control points are selected as $\boldsymbol{P}_{i j}^{(0)}=\boldsymbol{Q}_{\left\lfloor m i / n_1\right\rfloor,\left\lfloor p j / n_2\right\rfloor}$ for $i \in\left[n_1\right]$ and $j \in\left[n_2\right]$. 

Similarly, we give the following settings: the number of data points and control points in the u direction are m=60, $n_1$=20, the number of data points and control points in the v direction are p=60, $n_2$=20, and the sub-block size in the u direction and the sub-block size in the v direction are both set to 5. The algorithm terminates when:
$$
\frac{\left\|A P^{(k+1)} B^T-A P^{(k)} B^T\right\|_F}{\left\|A P^{(k)} B^T\right\|_F}<10^{-8},
$$
or the maximum number of iterations (10000) is reached. Define the noise data ${\boldsymbol{Q^\delta}}$, which is generated from the exact data $\boldsymbol{Q}$ as
$$
{\boldsymbol{Q^\delta}}={\boldsymbol{Q}}+a \frac{\widetilde{\boldsymbol{Q}}}{\|\widetilde{\boldsymbol{Q}}\|_F}~(a=40~or~100),
$$
with the random variable $\widetilde{\boldsymbol{Q}}$ following an i.i.d. standard Gaussian distribution.
$\Gamma_u$ and $\Gamma_v$ are the second-order difference matrices with Dirichlet boundary conditions in the u and v directions, respectively, that is:
\[
\Gamma_u = 
C\begin{pmatrix}
-2 & 1 & 0 & 0 & \cdots & 0 & 0 \\
1 & -2 & 1 & 0 & \cdots & 0 & 0 \\
0 & 1 & -2 & 1 & \cdots & 0 & 0 \\
\vdots & \ddots & \ddots & \ddots & \ddots & \vdots & \vdots \\
0 & 0 & \cdots & 1 & -2 & 1 & 0 \\
0 & 0 & \cdots & 0 & 1 & -2 & 1 \\
0 & 0 & \cdots & 0 & 0 & 1 & -2
\end{pmatrix}_{(n_1+1) \times (n_1+1),}
\]
\[
\Gamma_v = 
C\begin{pmatrix}
-2 & 1 & 0 & 0 & \cdots & 0 & 0 \\
1 & -2 & 1 & 0 & \cdots & 0 & 0 \\
0 & 1 & -2 & 1 & \cdots & 0 & 0 \\
\vdots & \ddots & \ddots & \ddots & \ddots & \vdots & \vdots \\
0 & 0 & \cdots & 1 & -2 & 1 & 0 \\
0 & 0 & \cdots & 0 & 1 & -2 & 1 \\
0 & 0 & \cdots & 0 & 0 & 1 & -2
\end{pmatrix}_{(n_2+1) \times (n_2+1),}
\]
where $C$ is a positive constant. Taking $C=91$ in Example \ref{example 7.3}.

\begin{example}\label{example 7.3}
$(m+1) \times(p+1)$ data points sampled uniformly from a boy surface, whose parametric equation is given by\\
$$
\left\{\begin{array}{l}
\boldsymbol{x}=2 / 3(\cos (t) \cos (2 t)+\sqrt{2} \sin (t) \cos (s)) \cos (t) /(\sqrt{2}-\sin (2 t) \sin (3 s)), \\ 
\boldsymbol{y}=2 / 3(\cos (t) \sin (2 t)-\sqrt{2} \sin (t) \sin (s)) \cos (t) /(\sqrt{2}-\sin (2 t) \sin (3 s)), \\ 
\boldsymbol{z}=\sqrt{2} \cos (t) \cos (t) /(\sqrt{2}-\sin (2 t) \sin (3 s))(-\pi \leq t, s \leq \pi) .\end{array}\right.
$$
\end{example}

\begin{figure}[htbp]
    \centering
    \begin{subfigure}[t]{0.48\textwidth}
        \centering
        \includegraphics[width=\linewidth]{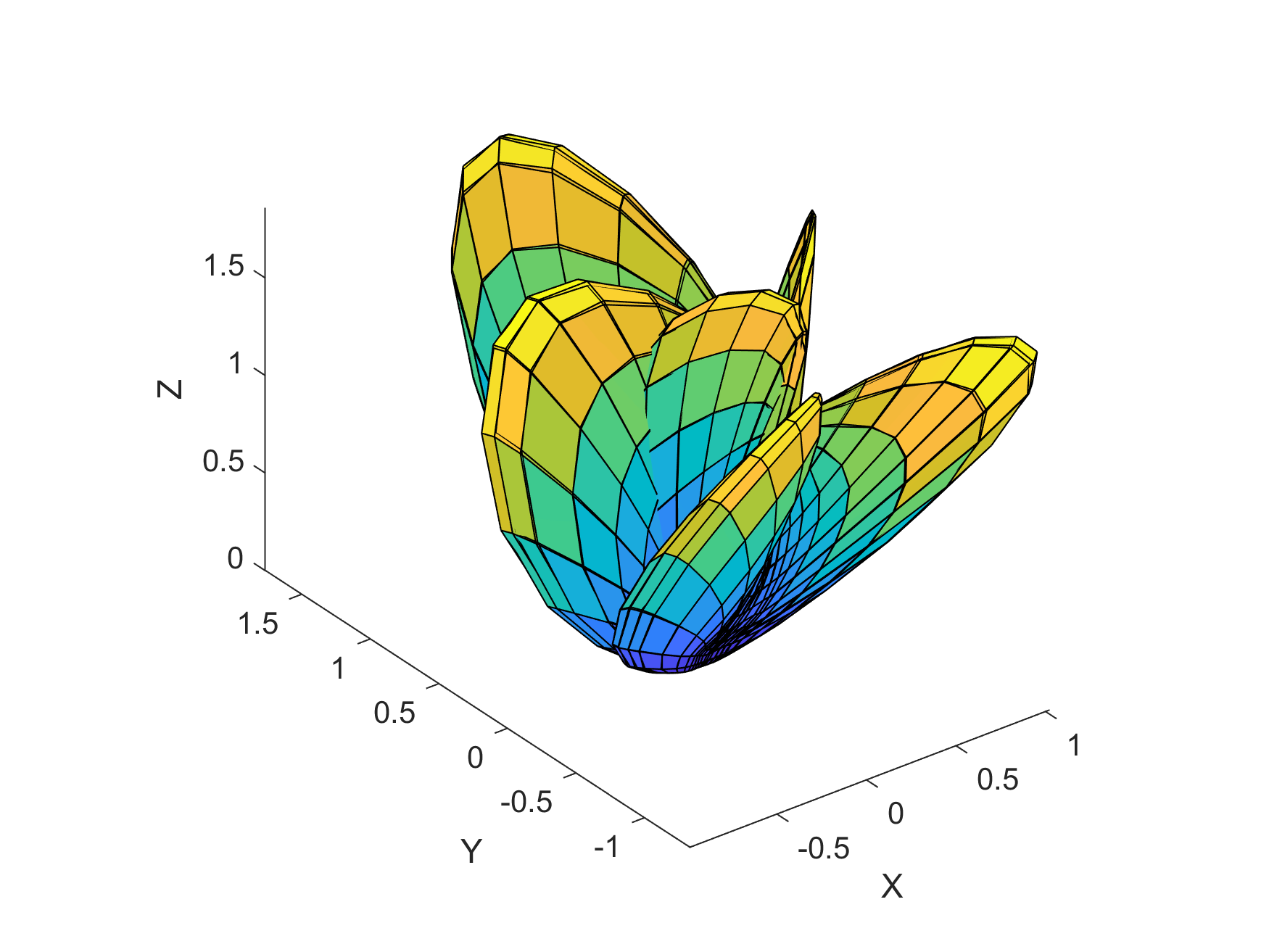}
        \caption{Example \ref{example 7.3} Real Surface}
    \end{subfigure}
    \hfill
    \begin{subfigure}[t]{0.48\textwidth}
        \centering
        \includegraphics[width=\linewidth]{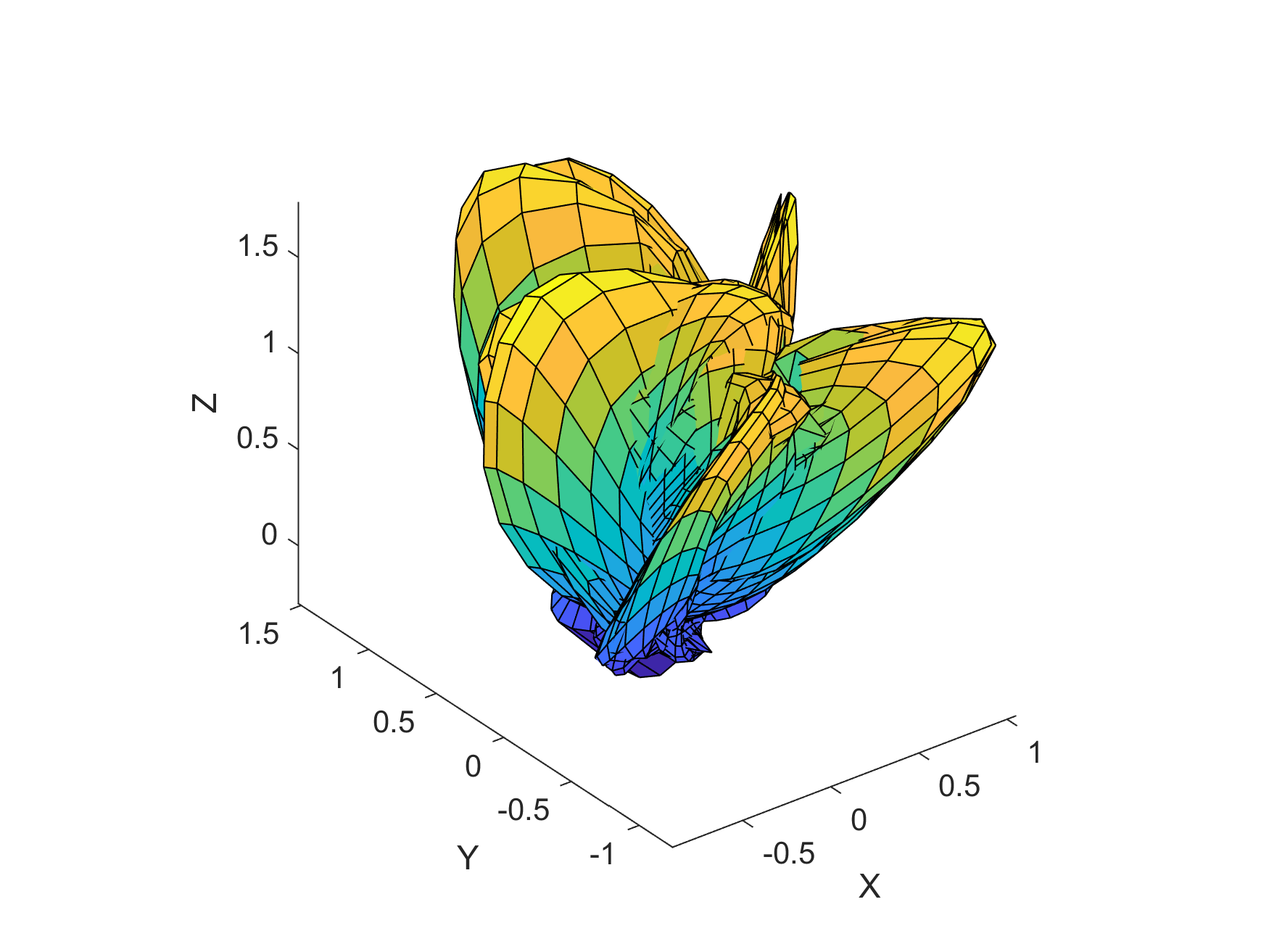}
        \caption{Regularized RPIA Fitted Surface}
    \end{subfigure}
    \vskip\baselineskip 
    \begin{subfigure}[t]{0.48\textwidth}
        \centering
        \includegraphics[width=\linewidth]{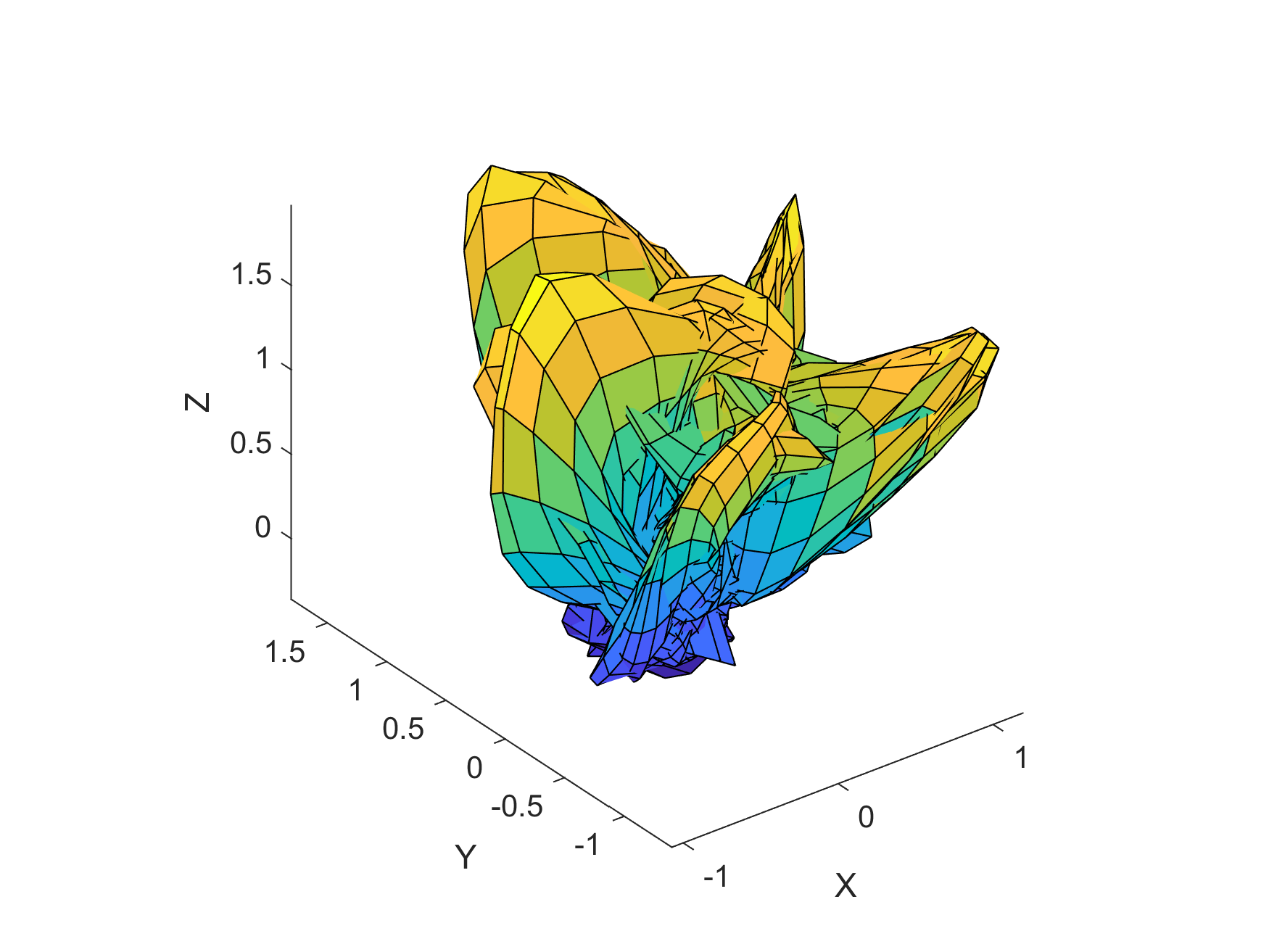}
        \caption{Regularized RPIA Fitted Surface($\lambda=0$)}
    \end{subfigure}
    \hfill
    \begin{subfigure}[t]{0.48\textwidth}
        \centering
        \includegraphics[width=\linewidth]{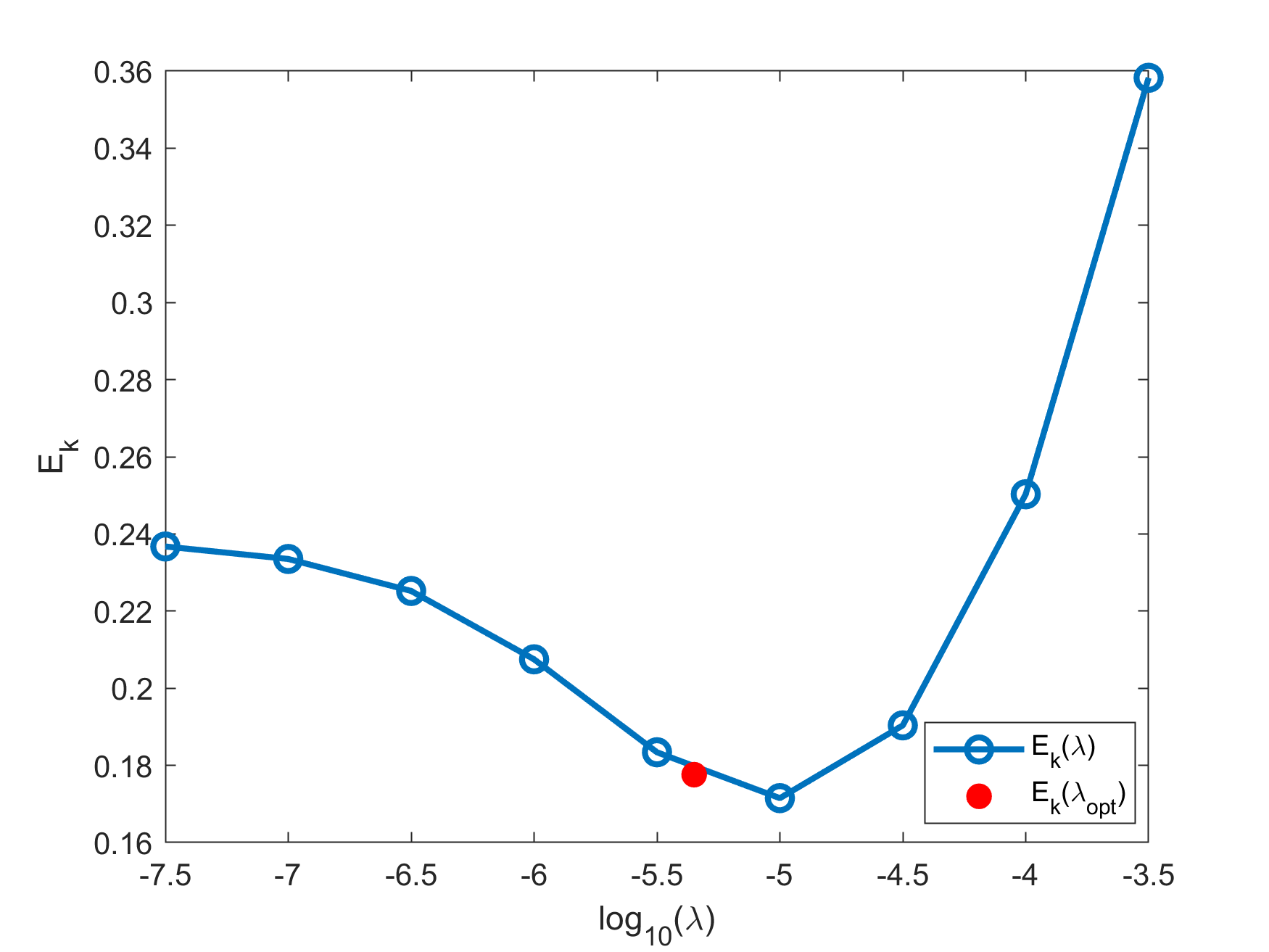}
        \caption{Fitting Error}
    \end{subfigure}
    \caption{The fitting surface given by Regularized RPIA for Example \ref{example 7.3}}
    \label{Fig:5}
\end{figure}

\begin{figure}[htbp]
    \centering
    \begin{subfigure}[t]{0.50\textwidth}
        \centering
        \includegraphics[width=\linewidth]{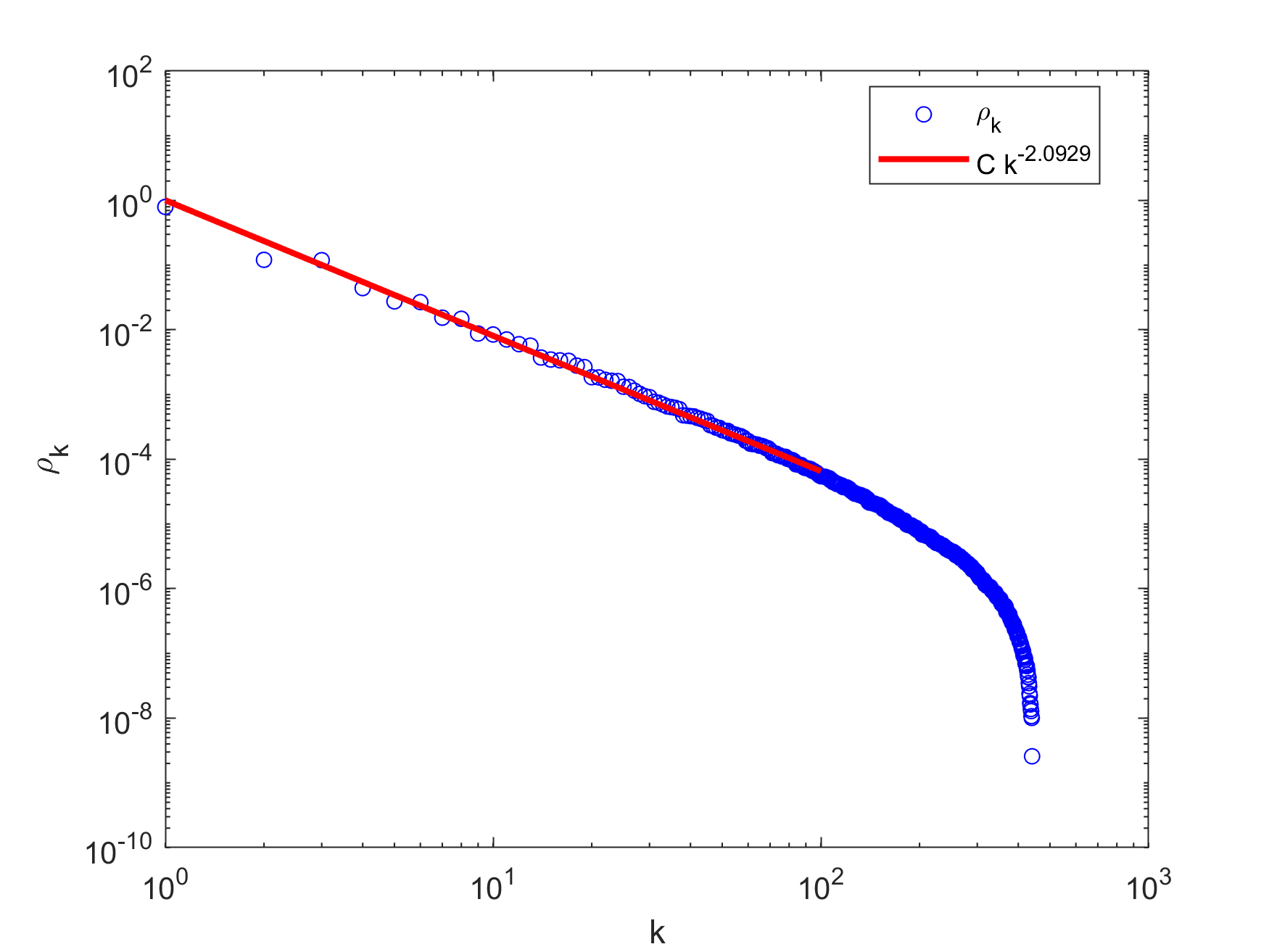}
    \end{subfigure}
    \caption{Spectrum of $Q^TQ$ for Example \ref{example 7.3}}
    \label{Fig:6}
\end{figure}
From the results in Figure \ref{Fig:5}, we can see that the regularization effect is obvious when the noise level  parameter $a$ equals 40. 
Figure (a) represents the real surface, which perfectly fits the noise-free data points; Figure (b) represents the fitting surface with the optimal parameter $\lambda$= 4.480e-06, the error $E_k = 0.177585$; Figure (c) represents the surface without regularization ($\lambda=0$), and the error $E_k = 0.237919$; Figure (d) represents the errors as $\lambda$ changes, and the red dot denotes the estimated optimal point, which is very close to the point with the minimum error. 
This further verifies the correctness of the optimal parameter estimation expression $\lambda^{1+1/\alpha}=O(\sigma^2n^{-1})\|\Gamma\bar{p}\|_n^{-2}$ in Section \ref{section5}. The definition of the fitting error is as follows.
$$
E_k(\lambda)=\frac{\left\|A p^*B^T-A \bar{p}B^T\right\|_F^2}{\|A \bar{p}B^T\|_F^2}.
$$
Figure \ref{Fig:6} shows the decay of the eigenvalue of matrix $Q$ in Section \ref{section5}(Select the first 100 eigenvalues), with a decay rate of 2.0929, i.e. $\alpha=2.0929$ in $\lambda^{1+1/\alpha}=O(\sigma^2n^{-1})\|\Gamma\bar{p}\|_n^{-2}$.

In order to further test the model under high noise level, we set the noise level parameter $a=100$ based on the previous experiment. From Figure \ref{Fig:7}, we can easily see that under high noise level, the fitting surface of Figure (c) without regularization is messy, with an error of 0.579405. But under the optimal regularization parameter $\lambda$ = 1.548e-05, the fitting surface of Figure (b) is significantly improved, with an error of 0.328436. At the same time, it can also be seen from Figure (d) that when the noise level increases, the regularization parameter also increases, which is consistent with the theoretical results. The above results are obtained by running the regularized RPIA 3 times continuously and taking the arithmetic mean of the results.

It is worth noting that in the process of estimating the optimal parameters for the surface fitting problem, we transform the surface optimization problem into the form of the curve mentioned in Section 5. In order to make the right-hand side term $O(\sigma^2n^{-1})\|\Gamma\bar{p}\|_n^{-2}$ of the estimation formula not contain $\lambda$, we discard the high-order term $\lambda^2$ of the surface optimization problem. Numerical experiments also show that $\lambda$ is very small and discarding high-order terms will not affect the numerical results.
\begin{figure}[htbp]
    \centering
    \begin{subfigure}[t]{0.48\textwidth}
        \centering
        \includegraphics[width=\linewidth]{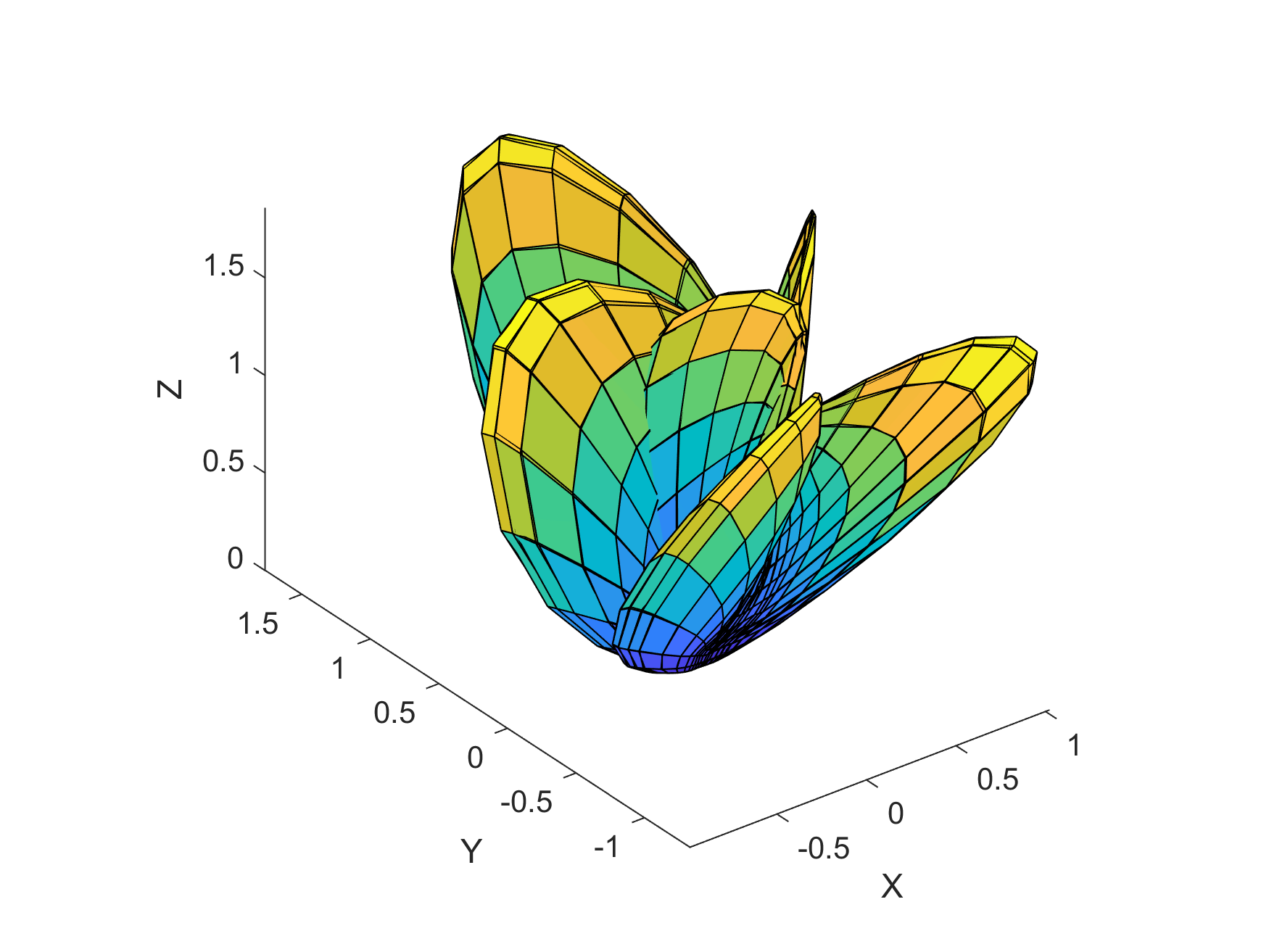}
        \caption{Example \ref{example 7.3} Real Surface}
    \end{subfigure}
    \hfill
    \begin{subfigure}[t]{0.48\textwidth}
        \centering
        \includegraphics[width=\linewidth]{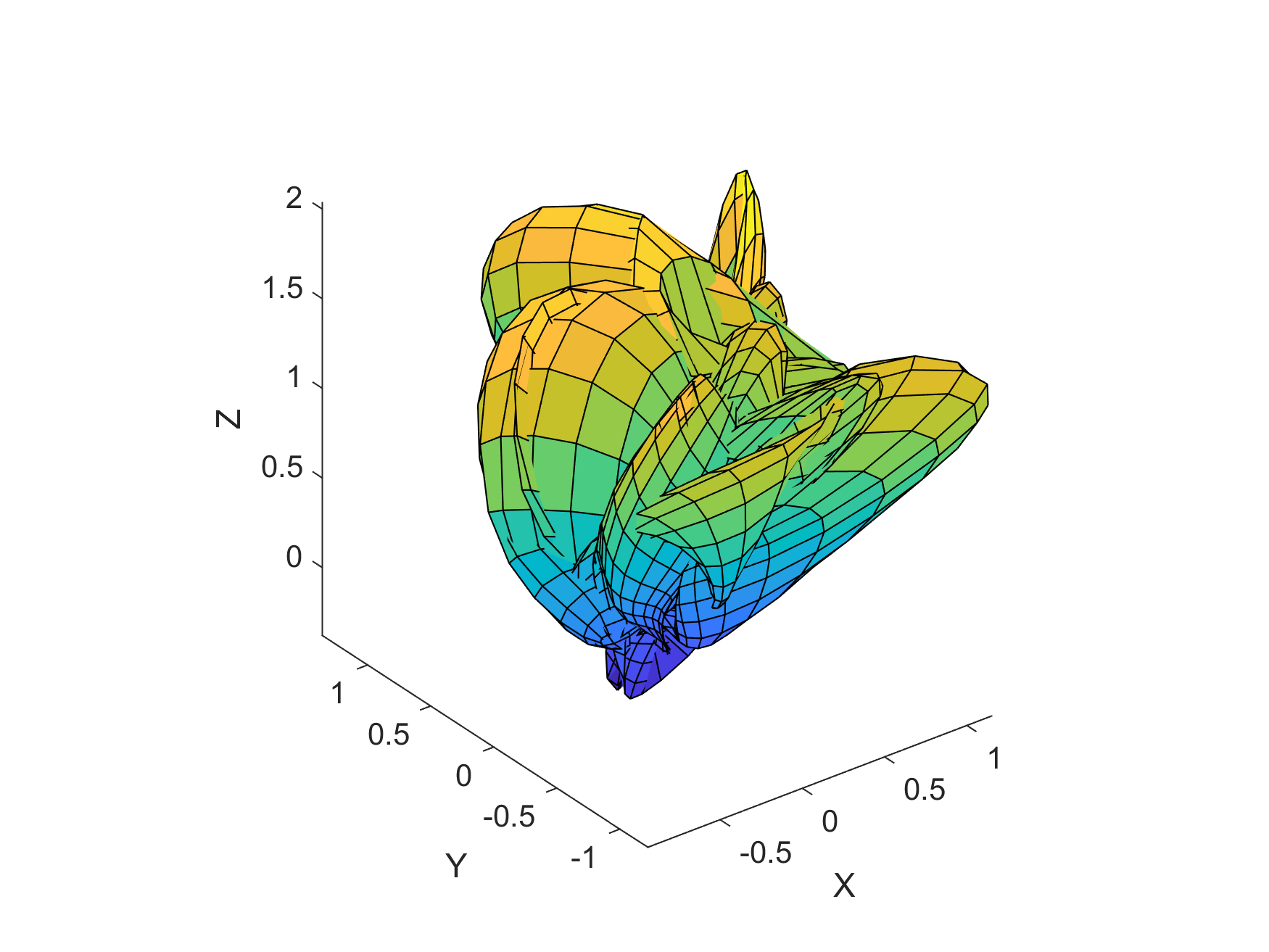}
        \caption{Regularized RPIA Fitted Surface}
    \end{subfigure}
    \vskip\baselineskip 
    \begin{subfigure}[t]{0.48\textwidth}
        \centering
        \includegraphics[width=\linewidth]{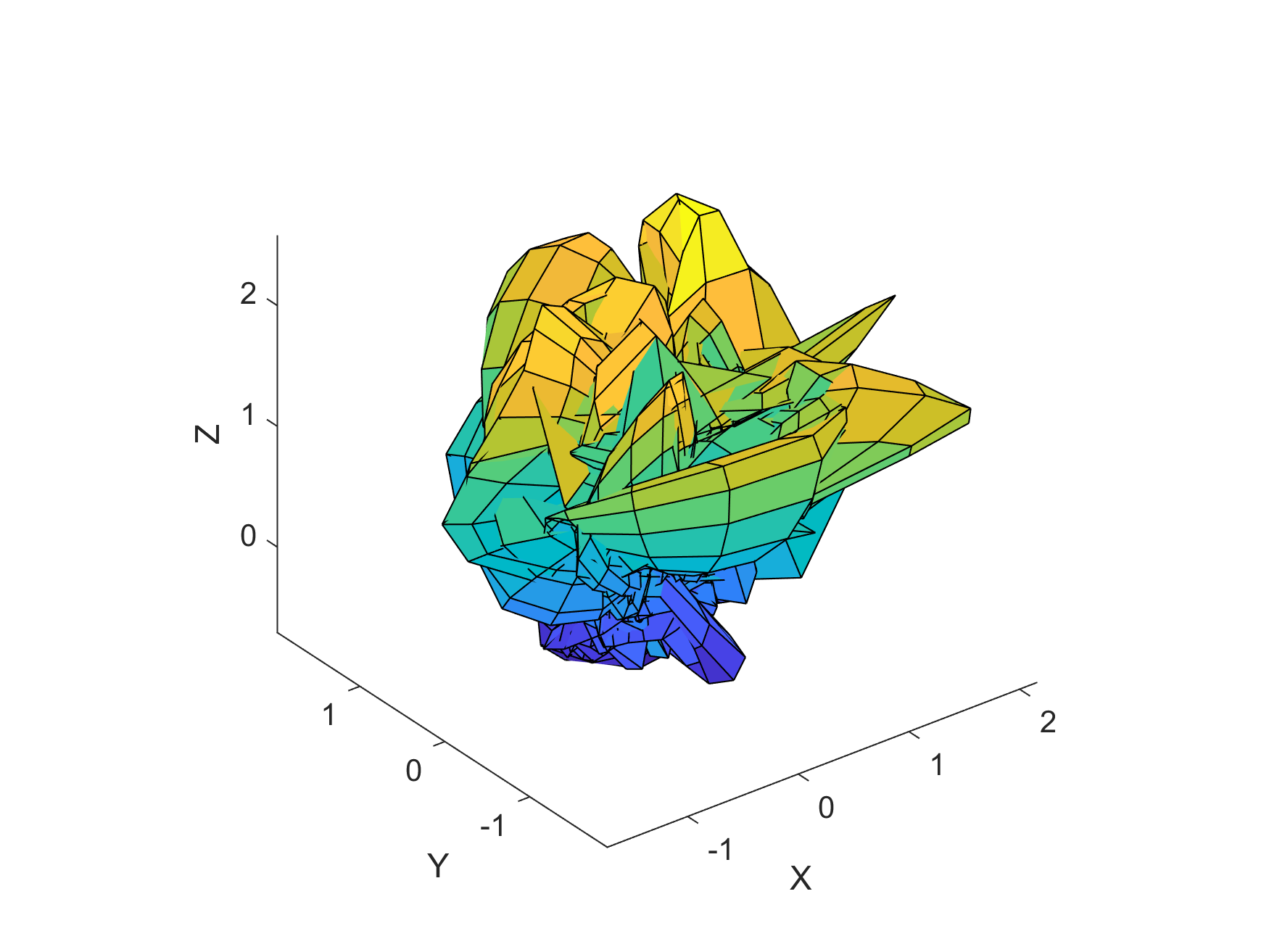}
        \caption{Regularized RPIA Fitted Surface($\lambda=0$)}
    \end{subfigure}
    \hfill
    \begin{subfigure}[t]{0.48\textwidth}
        \centering
        \includegraphics[width=\linewidth]{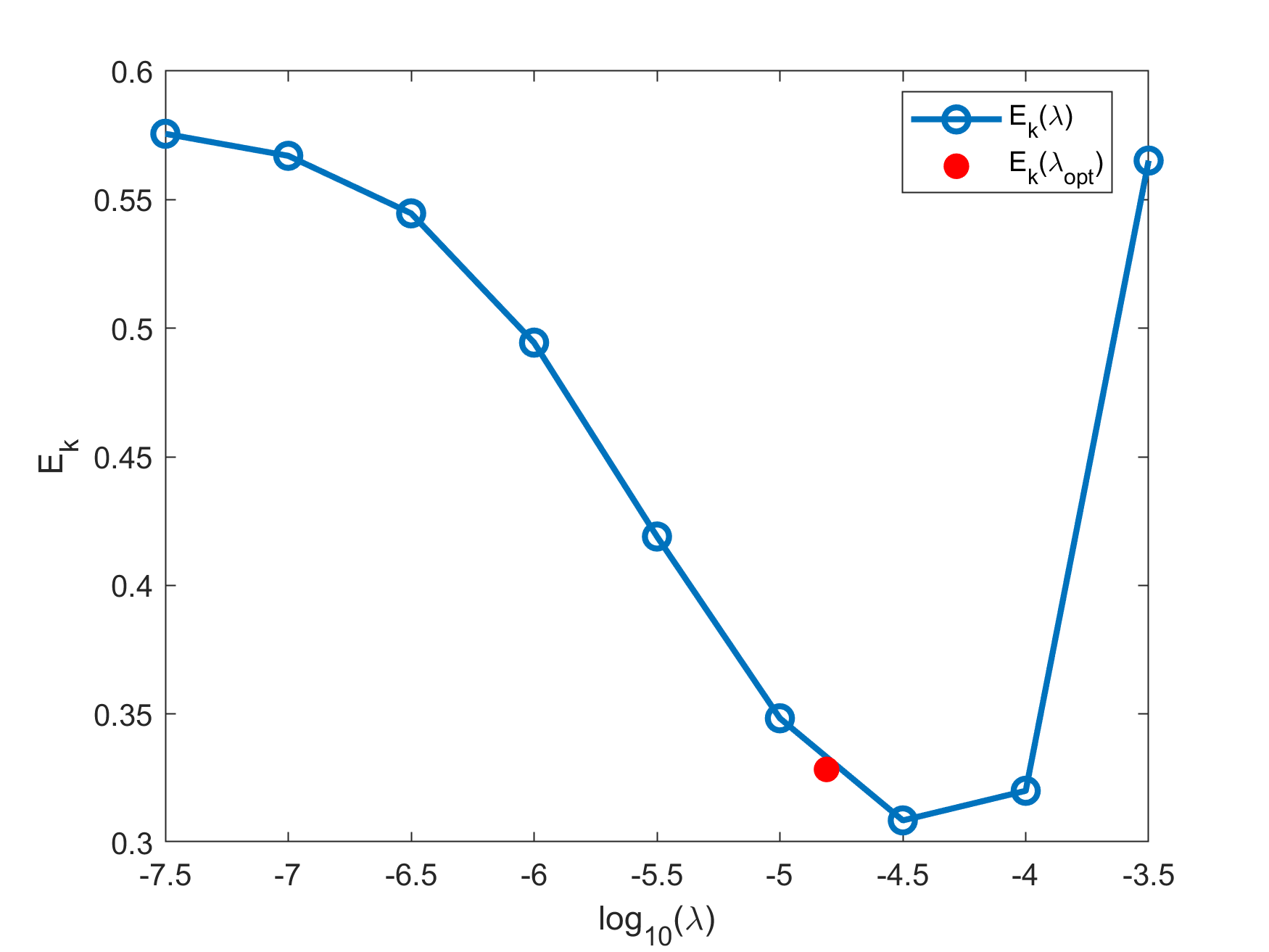}
        \caption{Fitting Error}
    \end{subfigure}
    \caption{The fitting surface given by Regularized RPIA for Example \ref{example 7.3}}
    \label{Fig:7}
\end{figure}

\subsection{Self-consistent iterative algorithms for curves and surfaces}
In order to further solve the problem of unknown information in practical applications, we propose a self-consistent iterative algorithm, the convergence threshold $\varepsilon_\lambda=0.01$, without prior information to find the regularization parameter in Section \ref{section6}. The following are some numerical examples. The settings of all parameters remain the same as Section \ref{section7.1} and Section \ref{section7.2}.

\begin{figure}[htbp]
    \centering
    \begin{subfigure}[b]{0.48\textwidth}
        \centering
        \includegraphics[width=\linewidth]{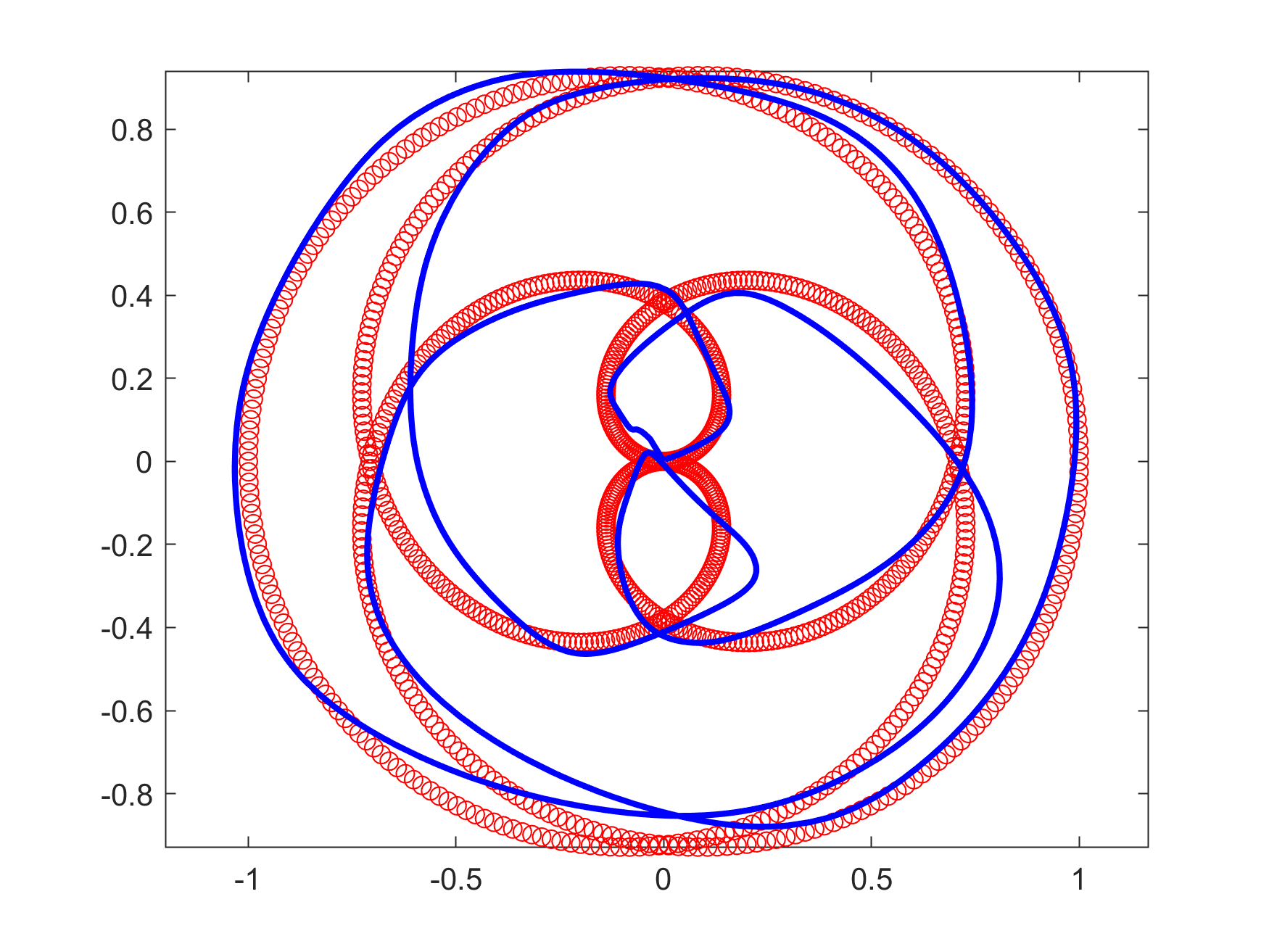}
        \caption{Fitted Curve for Example \ref{example 7.1}}
    \end{subfigure}
    \hfill
    \begin{subfigure}[b]{0.48\textwidth}
        \centering
        \includegraphics[width=\linewidth]{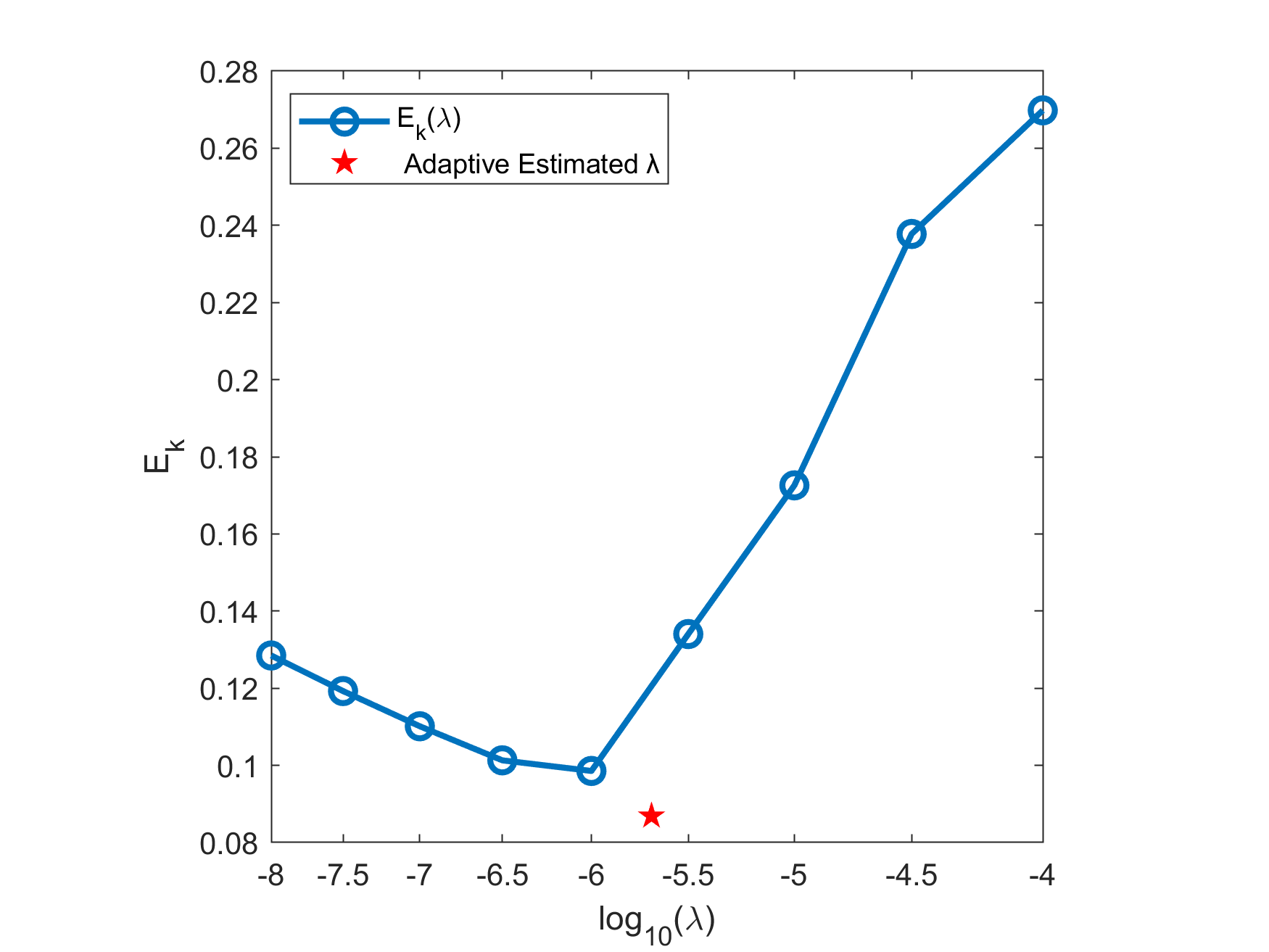}
        \caption{Fitting error}
    \end{subfigure}
    \caption{Adaptive algorithm results for Example \ref{example 7.1}}
    \label{fig: adapt1}
\end{figure}

\begin{figure}[htbp]
    \centering
    \begin{subfigure}[t]{0.50\textwidth}
        \centering
        \includegraphics[width=\linewidth]{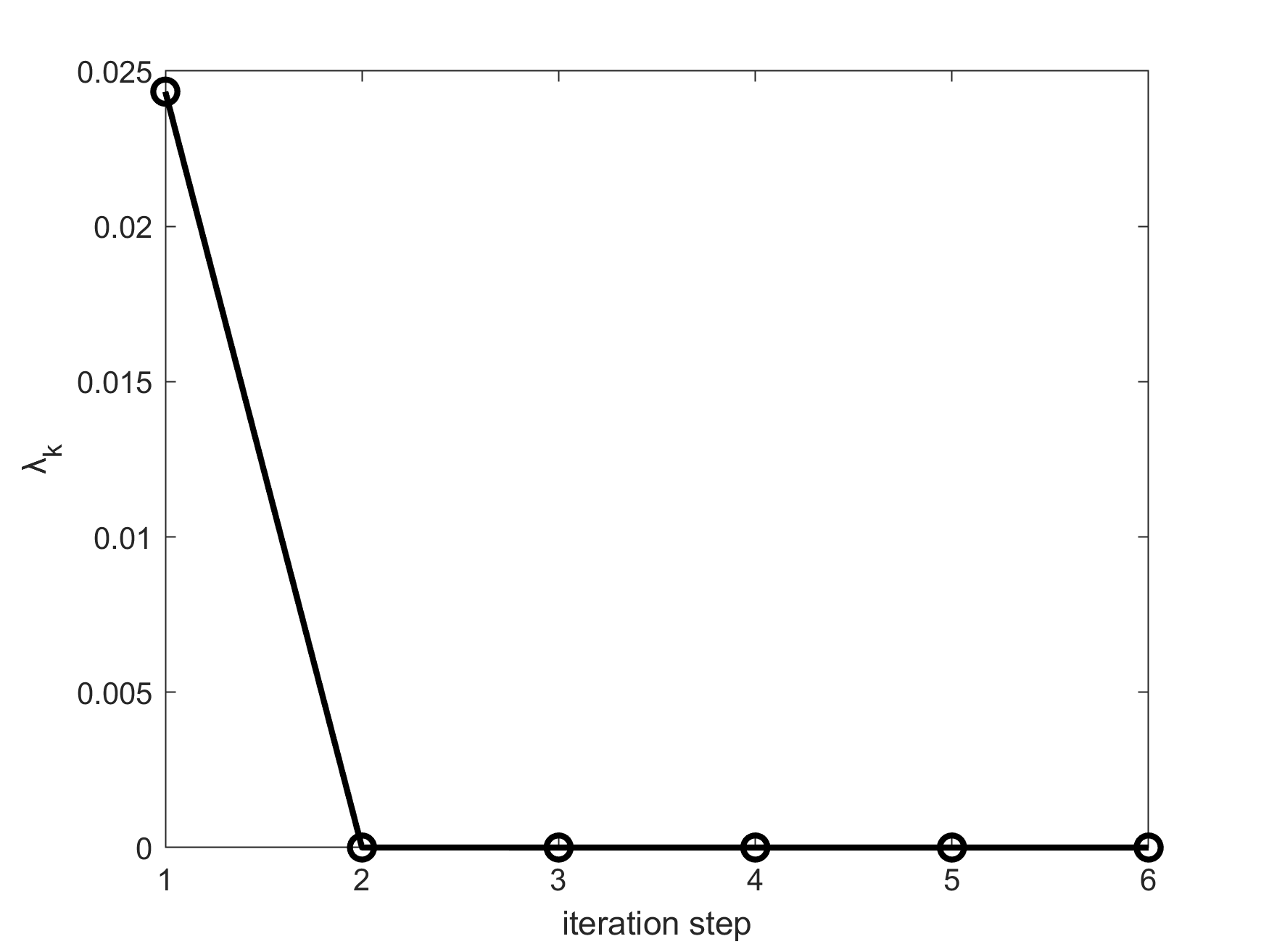}
    \end{subfigure}
    \caption{Iteration Step for Example \ref{example 7.1}}
    \label{Fig:9}
\end{figure}

It can be seen from Figure \ref{fig: adapt1} that the adaptive iterative algorithm works well for Example \ref{example 7.1}. Specifically, the regularization parameter $\lambda$ found by the adaptive algorithm is 2.067e-06, and the corresponding  fitting error is $E_k = 0.086871$(see Figure (a)), which is smaller than the error of the previous non-adaptive algorithm$(E_k = 0.097077)$. More importantly, it only requires a small number of iterations and is closer to the optimal regularization point(see Figure (b)) and Figure \ref{Fig:9}).

\begin{figure}[htbp]
    \centering
    \begin{subfigure}[b]{0.48\textwidth}
        \centering
        \includegraphics[width=\linewidth]{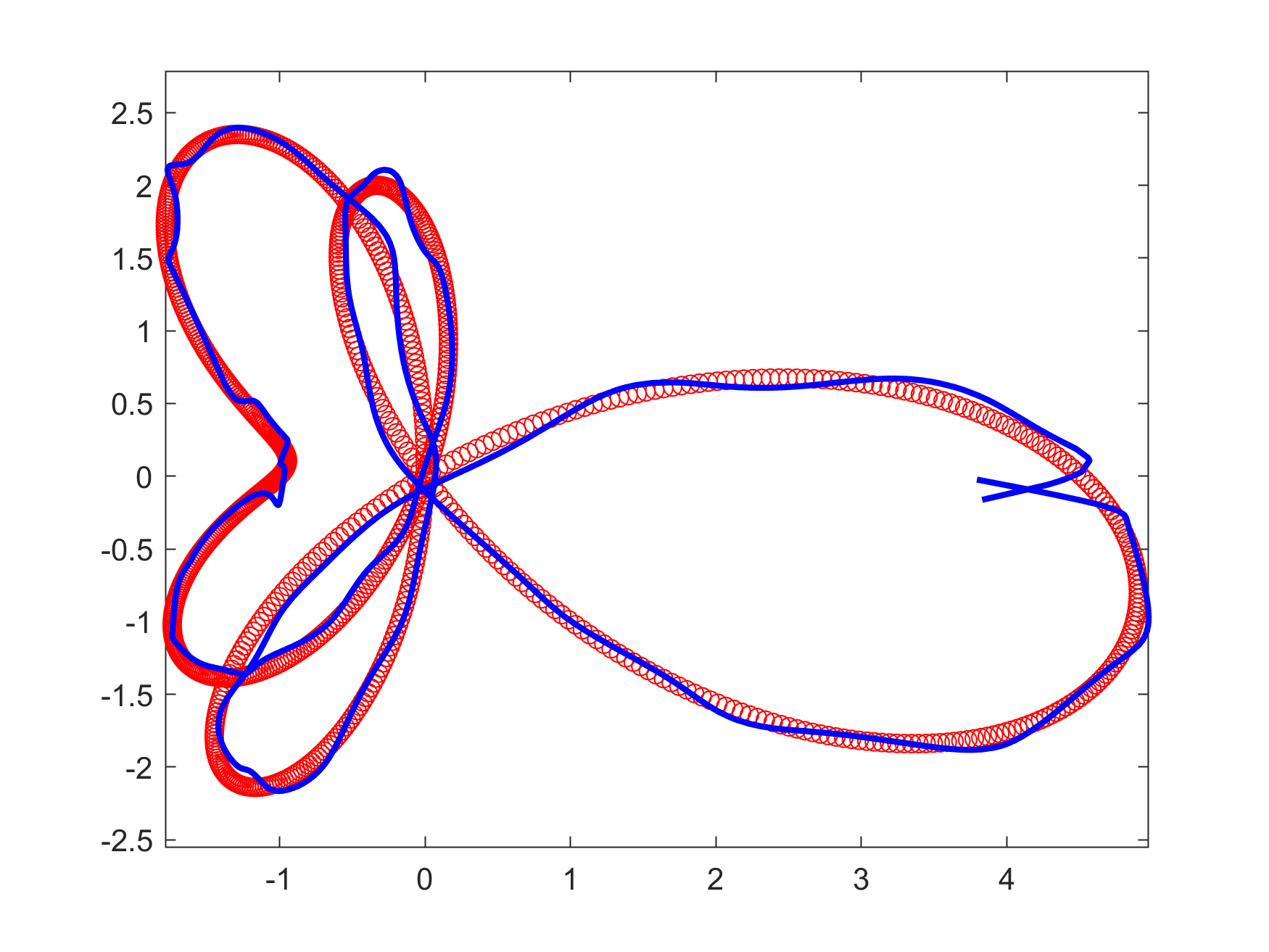}
        \caption{Fitted Curve for Example \ref{example 7.2}}
    \end{subfigure}
    \hfill
    \begin{subfigure}[b]{0.48\textwidth}
        \centering
        \includegraphics[width=\linewidth]{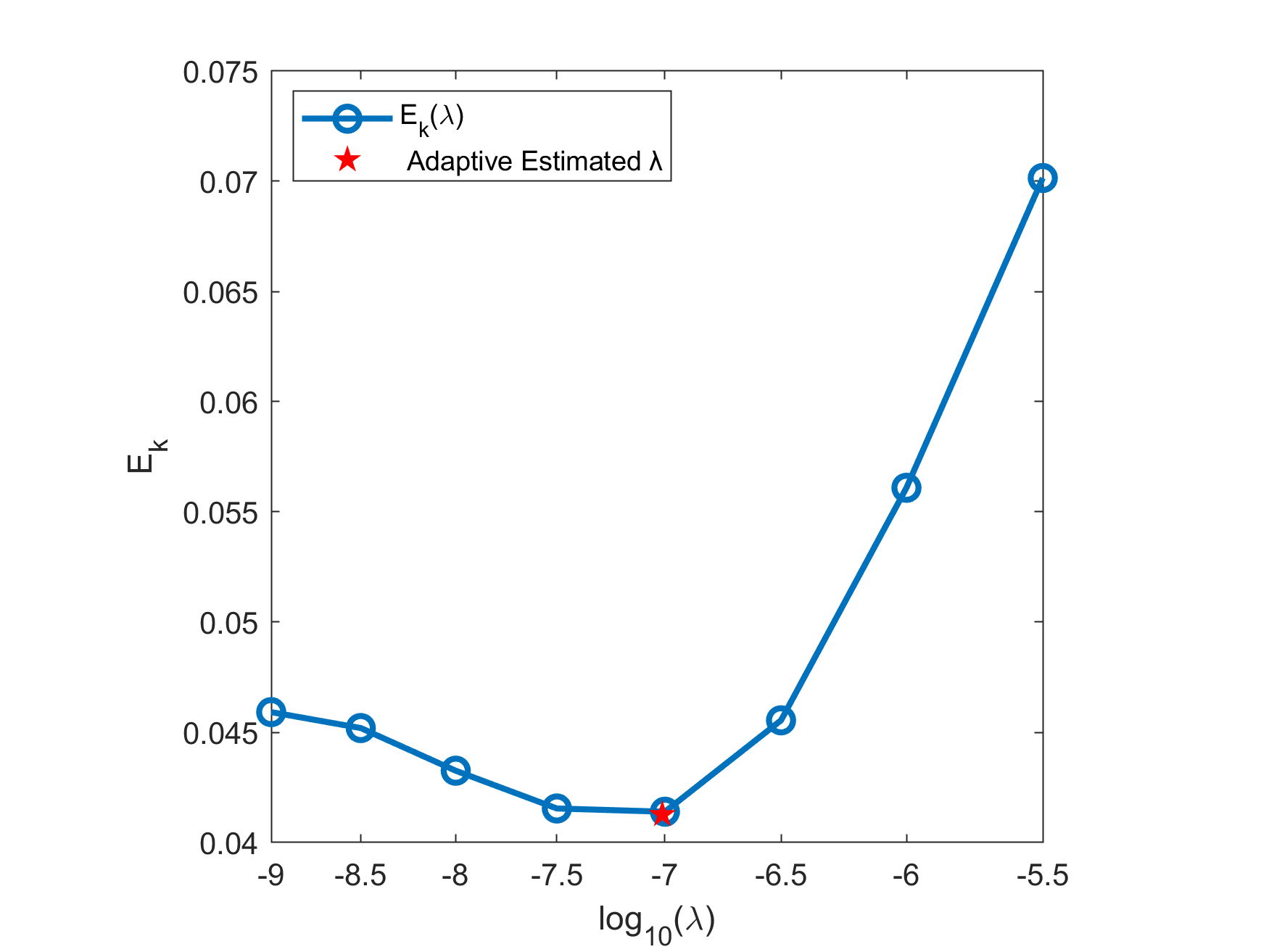}
        \caption{Fitting error}
    \end{subfigure}
    \caption{Adaptive algorithm results for Example \ref{example 7.2}}
    \label{fig:adapt2}
\end{figure}

\begin{figure}[htbp]
    \centering
    \begin{subfigure}[t]{0.50\textwidth}
        \centering
        \includegraphics[width=\linewidth]{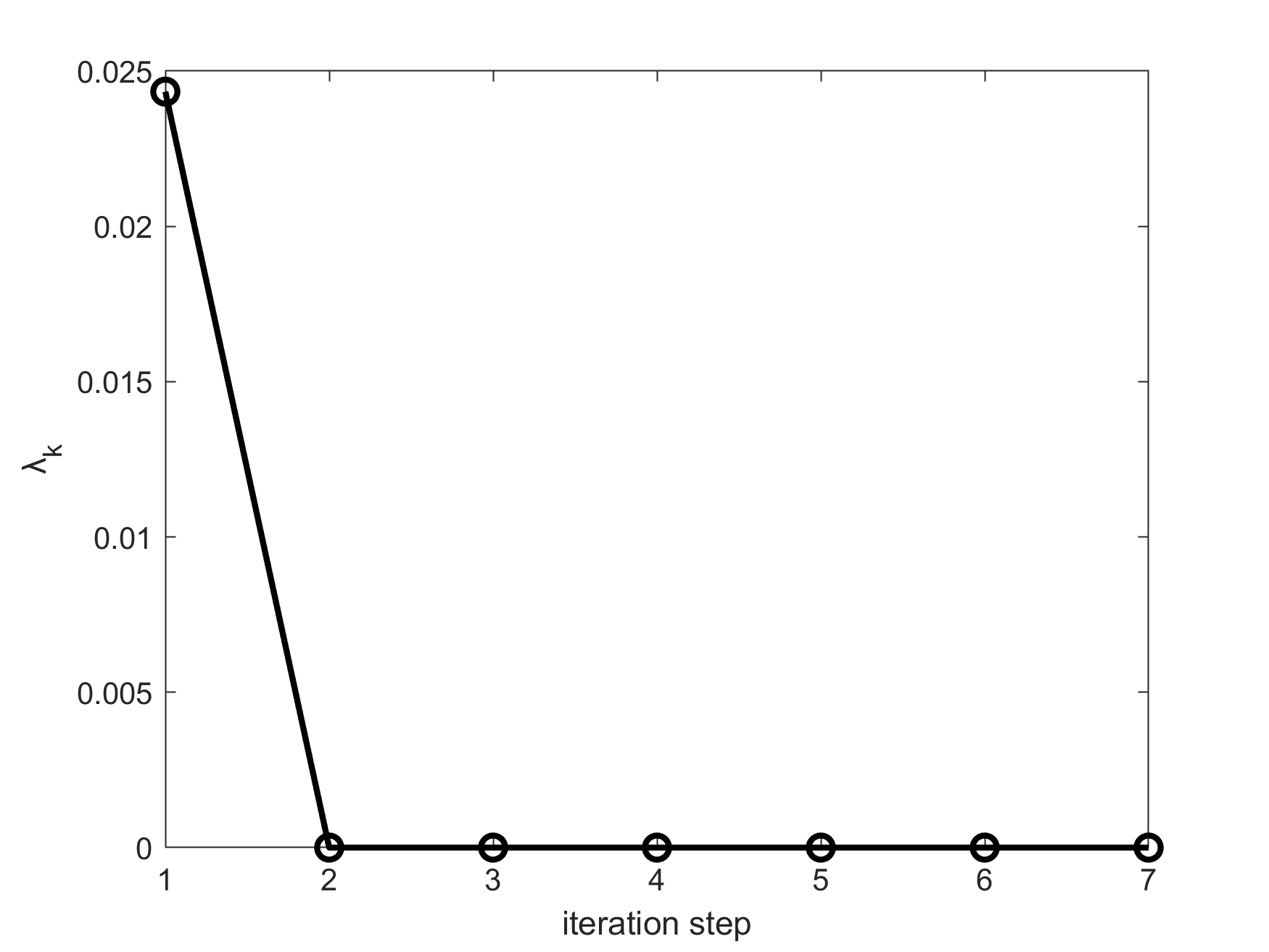}
    \end{subfigure}
    \caption{Iteration Step for Example \ref{example 7.2}}
    \label{Fig:10}
\end{figure}
For example \ref{example 7.2}, the optimal parameter $\lambda$ found by the algorithm 9.738e-08 and the fitting error is $E_k = 0.041325$. It only requires a small number of iterations and is closer to the optimal estimate(see Figure \ref{fig:adapt2} and Figure \ref{Fig:10}).

\begin{figure}[htbp]
    \centering
    \begin{subfigure}[b]{0.48\textwidth}
        \centering
        \includegraphics[width=\linewidth]{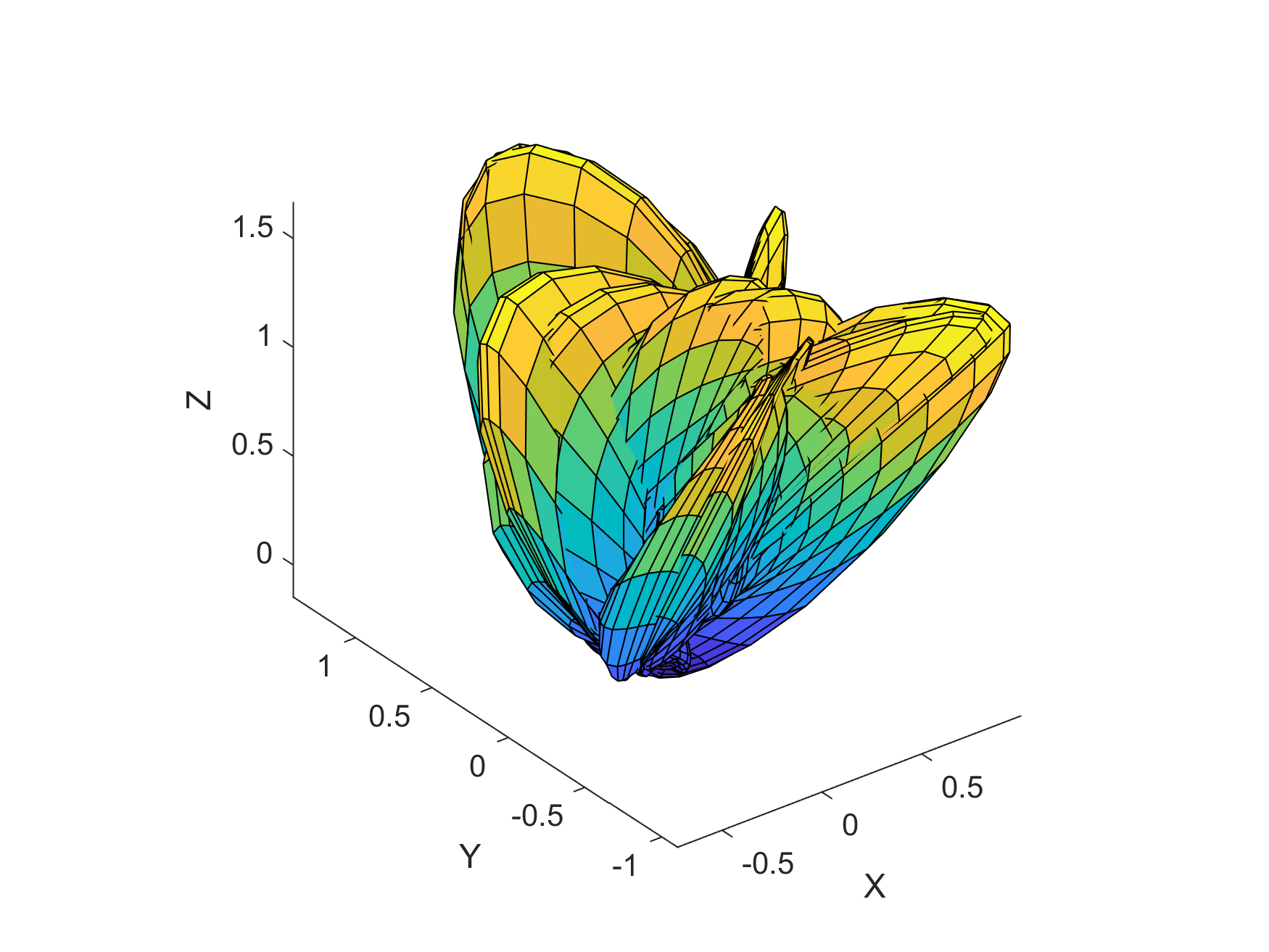}
        \caption{Fitted Surface for Example \ref{example 7.3}}
    \end{subfigure}
    \hfill
    \begin{subfigure}[b]{0.48\textwidth}
        \centering
        \includegraphics[width=\linewidth]{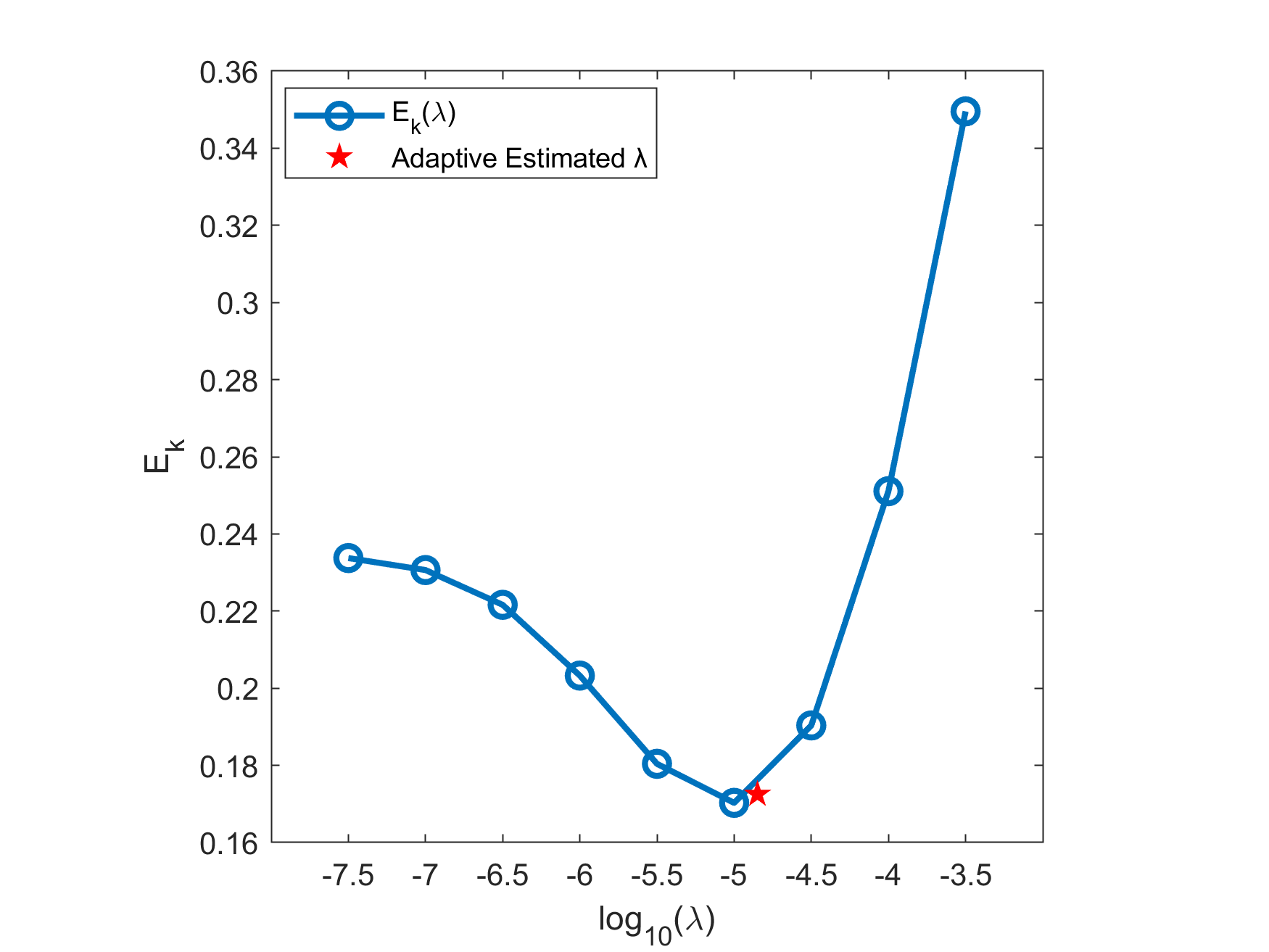}
        \caption{Fitting Error}
    \end{subfigure}
    \caption{Adaptive algorithm results for Example \ref{example 7.3} with noise level parameter $a=40$}
    \label{fig:adapt3}
\end{figure}
\begin{figure}[htbp]
    \centering
    \begin{subfigure}[t]{0.50\textwidth}
        \centering
        \includegraphics[width=\linewidth]{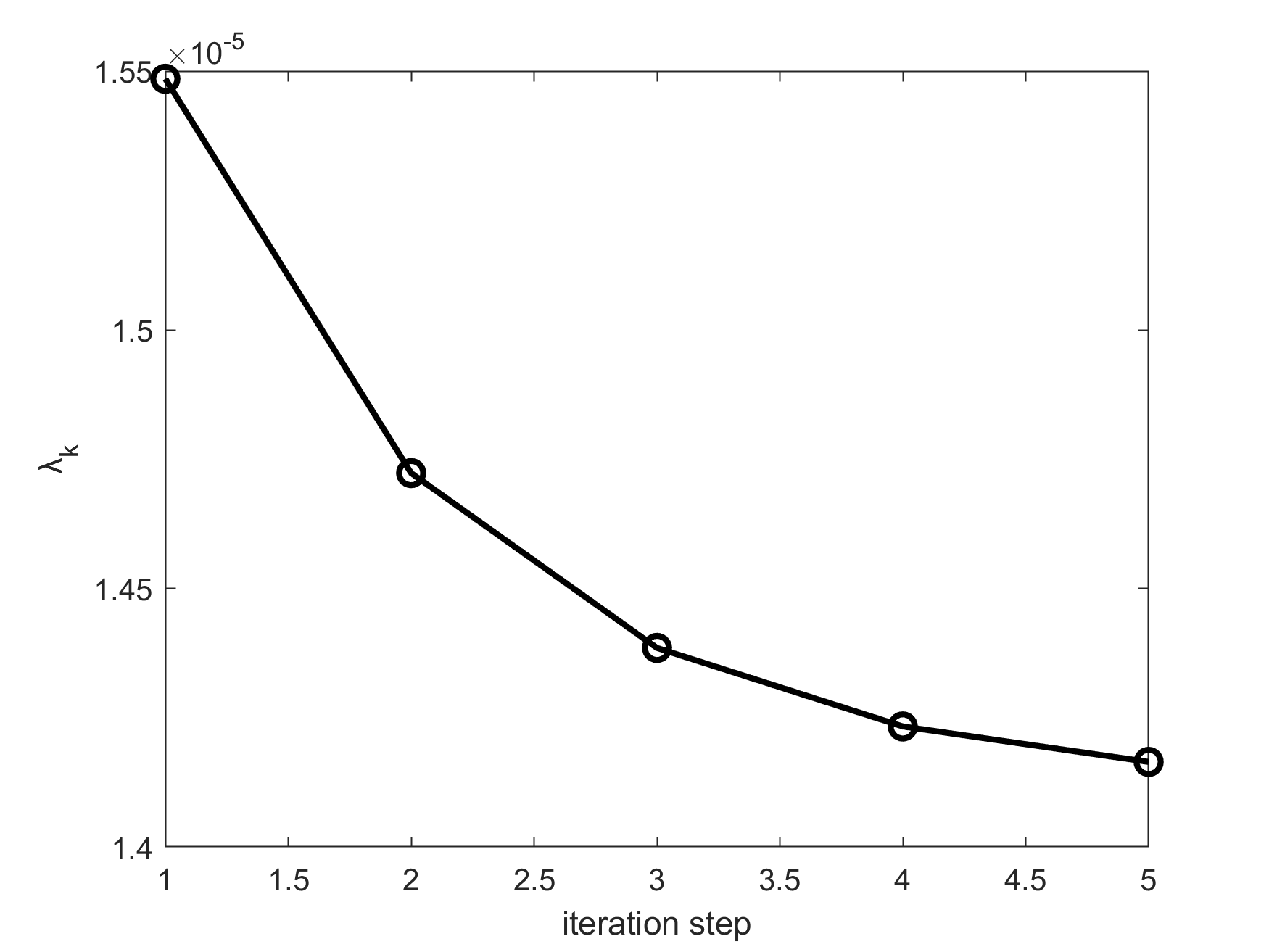}
    \end{subfigure}
    \caption{Iteration Step for Example \ref{example 7.3} with noise level parameter $a= 40$}
    \label{Fig:11}
\end{figure}
\begin{figure}[htbp]
    \centering
    \begin{subfigure}[b]{0.48\textwidth}
        \centering
        \includegraphics[width=\linewidth]{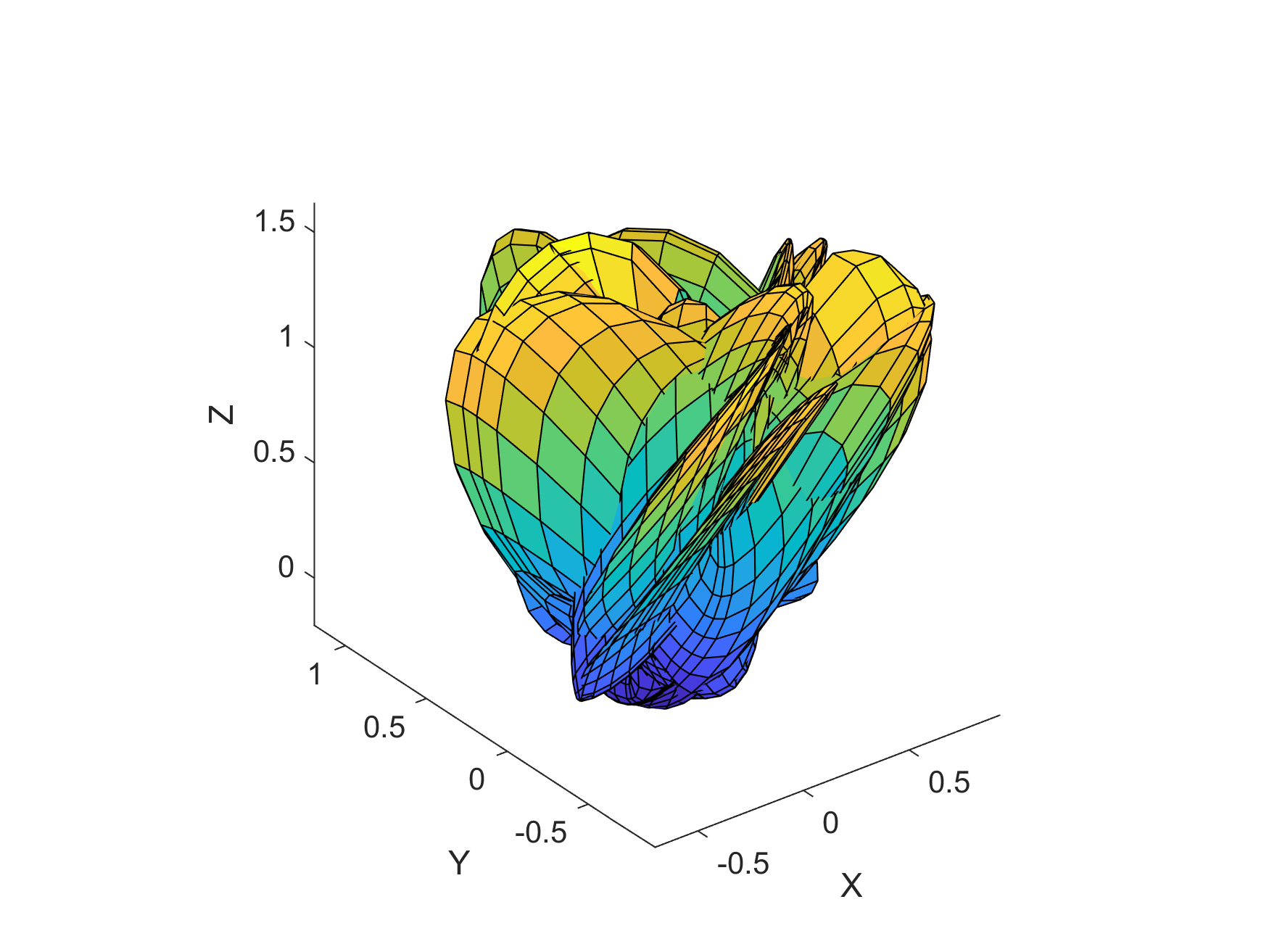}
        \caption{Fitted Surface for Example \ref{example 7.3}}
    \end{subfigure}
    \hfill
    \begin{subfigure}[b]{0.48\textwidth}
        \centering
        \includegraphics[width=\linewidth]{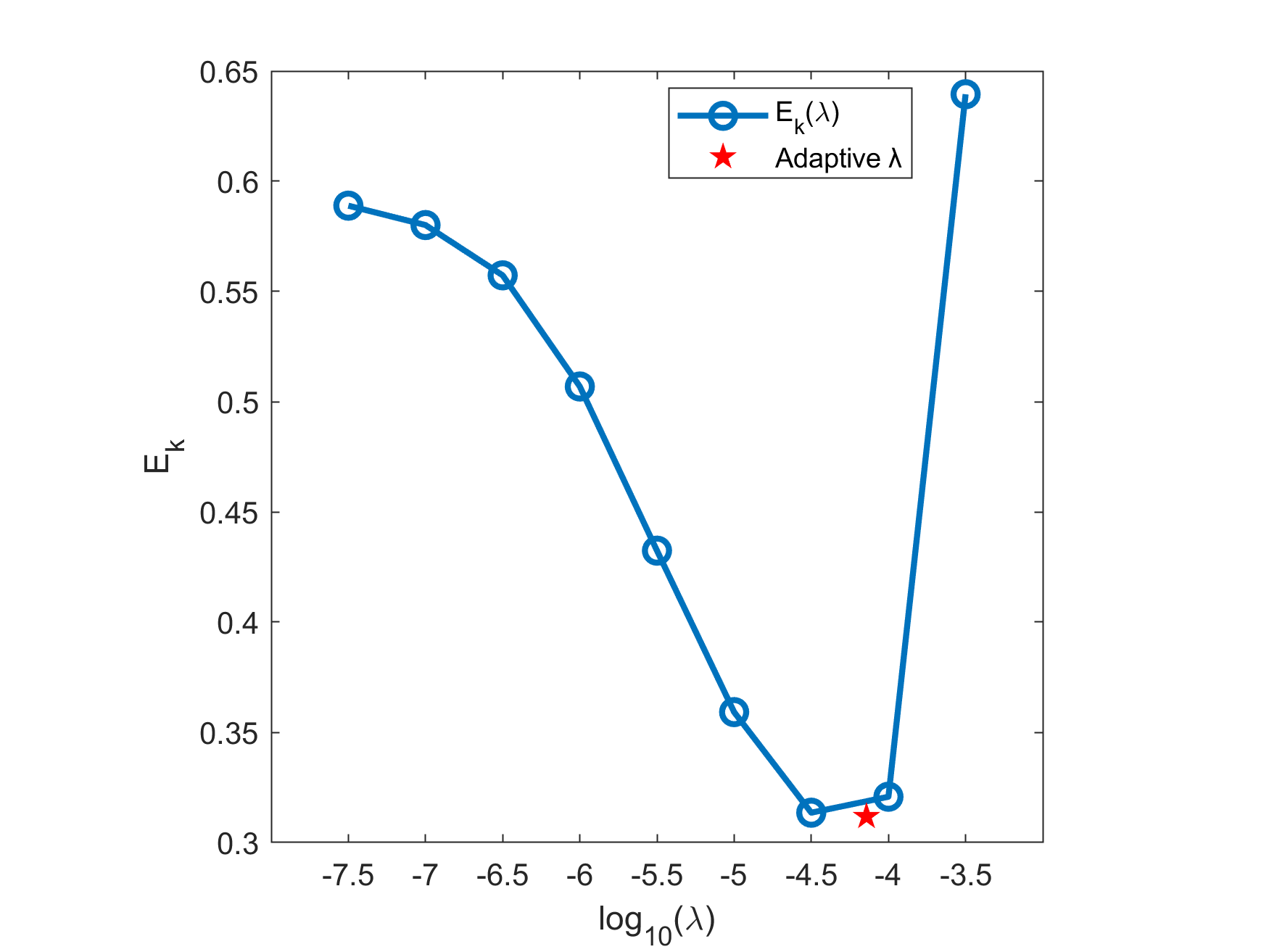}
        \caption{Fitting Error}
    \end{subfigure}
    \caption{Adaptive algorithm results for Example \ref{example 7.3} with noise level parameter $a=100$}
    \label{fig:adapt4}
\end{figure}
\begin{figure}[htbp]
    \centering
    \begin{subfigure}[t]{0.50\textwidth}
        \centering
        \includegraphics[width=\linewidth]{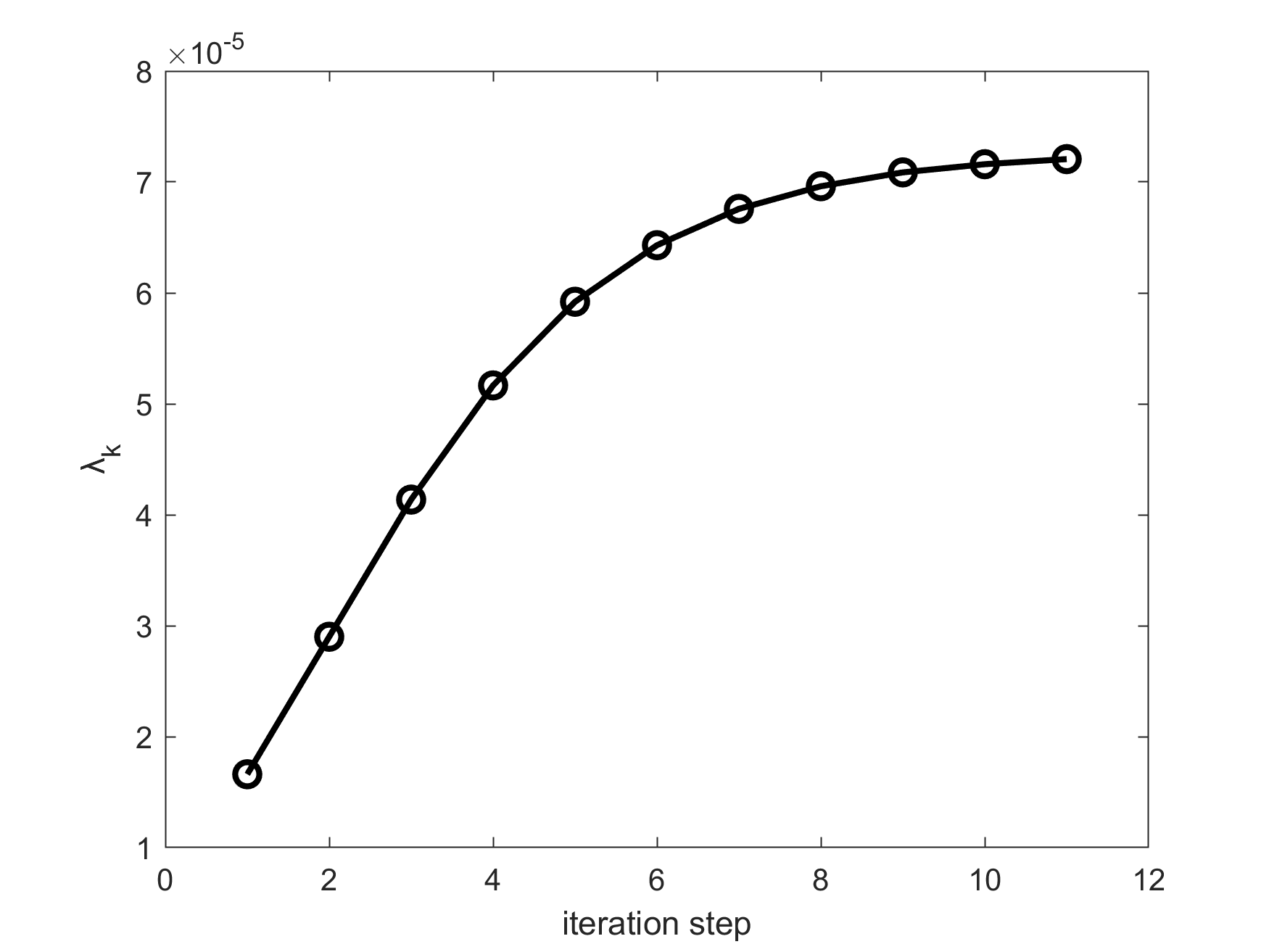}
    \end{subfigure}
    \caption{Iteration Step for Example \ref{example 7.3} with noise level parameter $a=100$}
    \label{Fig:12}
\end{figure}

For Example \ref{example 7.3},  when the noise level parameter $a$ is 40, the regularization parameter $\lambda$ found by the adaptive algorithm is 1.417e-05, and the corresponding fitting error $E_k= 0.172502$. It is better than the result of the non-adaptive algorithm and closer to the optimal regularization point(see Figure \ref{fig:adapt3}). From Figure \ref{Fig:11}, we can know that it also requires only a small number of iterations.
When the noise level parameter $a$ is 100, the regularization parameter $\lambda$ found by the adaptive algorithm is 7.206e-05, and the corresponding fitting error $E_k=0.311808$, which is better than the result of the non-adaptive algorithm. It only requires a small number of iterations and is closer to the optimal regularization point (see Figures \ref{fig:adapt4} and \ref{Fig:12}).

The above curve results are obtained by running the regularized RPIA 10 times in a row and taking the arithmetic mean of the results. The surface results are obtained by running 3 times.

\section{Concluding remarks}
In this paper, we propose a randomized progressive iterative regularization method based on the RPIA method for data fitting problems. By introducing a regularization term, we effectively improve the robustness and stability of the algorithm in the presence of noisy data. Theoretical analysis demonstrates that under appropriate conditions, the proposed method converges in expectation to the least-squares solution. Moreover, we provide a stochastic optimal estimation method for selecting the regularization parameter, which inspire self-consistent iterative algorithms without prior information. The numerical experiments presented in this study validate our theoretical analysis and clearly illustrate the advantages of the proposed algorithm in both the curve and surface fitting scenarios. To the best of our knowledge, convergence analysis of the RPIA method under noisy data conditions has not been addressed in the literature. More importantly, we provide a stochastic optimal estimation method for determining the regularization parameter, which is of significant practical value.

\section*{Declarations}
On behalf of all authors, the corresponding author states that there is no conflict of interest. No datasets were generated or analyzed during the current study.


\begin{thebibliography}{10}

\bibitem{Buccini2021AVN}
Alessandro Buccini and Patricia~Diaz de~Alba.
\newblock A variational non-linear constrained model for the inversion of fdem data.
\newblock {\em Inverse Problems}, 38, 2021.

\bibitem{Calvetti2025DistributedTR}
Daniela Calvetti and Erkki Somersalo.
\newblock Distributed tikhonov regularization for ill-posed inverse problems from a bayesian perspective.
\newblock {\em Computational Optimization and Applications}, 2025.

\bibitem{CARNICER20102010}
J.M. Carnicer, J.~Delgado, and J.M. Peña.
\newblock Richardson method and totally nonnegative linear systems.
\newblock {\em Linear Algebra and its Applications}, 433(11):2010--2017, 2010.

\bibitem{deboor1979agee}
de~Boor~C.
\newblock How does agee’s smoothing method work?
\newblock In {\em Proceedings of the 1979 Army Numerical Analysis and Computers Conference}, ARO Report 79-3, pages 299--302. Army Research Office, 1979.

\bibitem{DENG201432}
Chongyang Deng and Hongwei Lin.
\newblock Progressive and iterative approximation for least squares b-spline curve and surface fitting.
\newblock {\em Computer-Aided Design}, 47:32--44, 2014.

\bibitem{EBRAHIMI20191}
A.~Ebrahimi and G.B. Loghmani.
\newblock A composite iterative procedure with fast convergence rate for the progressive-iteration approximation of curves.
\newblock {\em Journal of Computational and Applied Mathematics}, 359:1--15, 2019.

\bibitem{HUANG2020101931}
Zheng-Da Huang and Hui-Di Wang.
\newblock On a progressive and iterative approximation method with memory for least square fitting.
\newblock {\em Computer Aided Geometric Design}, 82:101931, 2020.

\bibitem{LIN2005575}
Hong-Wei Lin, Hu-Jun Bao, and Guo-Jin Wang.
\newblock Totally positive bases and progressive iteration approximation.
\newblock {\em Computers \& Mathematics with Applications}, 50(3):575--586, 2005.

\bibitem{LIN2010322}
Hongwei Lin.
\newblock Local progressive-iterative approximation format for blending curves and patches.
\newblock {\em Computer Aided Geometric Design}, 27(4):322--339, 2010.

\bibitem{LIN201840}
Hongwei Lin, Takashi Maekawa, and Chongyang Deng.
\newblock Survey on geometric iterative methods and their applications.
\newblock {\em Computer-Aided Design}, 95:40--51, 2018.

\bibitem{LIN2011967}
Hongwei Lin and Zhiyu Zhang.
\newblock An extended iterative format for the progressive-iteration approximation.
\newblock {\em Computers \& Graphics}, 35(5):967--975, 2011.

\bibitem{LIU2020112389}
Chengzhi Liu, Xuli Han, and Juncheng Li.
\newblock Preconditioned progressive iterative approximation for triangular bézier patches and its application.
\newblock {\em Journal of Computational and Applied Mathematics}, 366:112389, 2020.

\bibitem{LU2010129}
Lizheng Lu.
\newblock Weighted progressive iteration approximation and convergence analysis.
\newblock {\em Computer Aided Geometric Design}, 27(2):129--137, 2010.

\bibitem{MONTEGRANARIO2007583}
Hebert Montegranario and Jairo Espinosa.
\newblock A regularization approach for surface reconstruction from point clouds.
\newblock {\em Applied Mathematics and Computation}, 188(1):583--595, 2007.

\bibitem{6035703}
Shoichi Okaniwa, Ahmad Nasri, Hongwei Lin, Abdulwahed Abbas, Yuki Kineri, and Takashi Maekawa.
\newblock Uniform b-spline curve interpolation with prescribed tangent and curvature vectors.
\newblock {\em IEEE Transactions on Visualization and Computer Graphics}, 18(9):1474--1487, 2012.

\bibitem{Qi1975}
D.~Qi, Z.~Tian, Y.~Zhang, and J.~Feng.
\newblock The method of numeric polish in curve fitting.
\newblock {\em ACTA MATHEMATICA SINICA}, 18(3):173--184, 1975.

\bibitem{RIOS2022113921}
Dany Rios and Bert Jüttler.
\newblock Lspia, (stochastic) gradient descent, and parameter correction.
\newblock {\em Journal of Computational and Applied Mathematics}, 406:113921, 2022.

\bibitem{WU2024128669}
Nian-Ci Wu and Cheng-Zhi Liu.
\newblock Randomized progressive iterative approximation for b-spline curve and surface fittings.
\newblock {\em Applied Mathematics and Computation}, 473:128669, 2024.

\bibitem{WU2025102439}
Nian-Ci Wu, Chengzhi Liu, and Juncheng Li.
\newblock On the triple-parameter least squares progressive iterative approximation and its convergence analysis.
\newblock {\em Computer Aided Geometric Design}, 119:102439, 2025.

\bibitem{Yin2020CurveFO}
Fulian Yin, Xiaoli Feng, Fangyuan Ju, and Yanyan Wang.
\newblock Curve fitting of the user barrage emotional change based on the hybrid kernel pso\_lssvm model.
\newblock {\em Proceedings of the 2020 3rd International Conference on Algorithms, Computing and Artificial Intelligence}, 2020.

\bibitem{Zhang2022ImpulseNI}
Benxin Zhang, Guopu Zhu, Zhibin Zhu, Hongli Zhang, Yicong Zhou, and Sam Kwong.
\newblock Impulse noise image restoration using nonconvex variational model and difference of convex functions algorithm.
\newblock {\em IEEE Transactions on Cybernetics}, 54:2257--2270, 2022.

\bibitem{ZHANG2018331}
Li~Zhang, Jieqing Tan, Xianyu Ge, and Guo Zheng.
\newblock Generalized b-splines’ geometric iterative fitting method with mutually different weights.
\newblock {\em Journal of Computational and Applied Mathematics}, 329:331--343, 2018.
\newblock The International Conference on Information and Computational Science, 2--6 August 2016, Dalian, China.

\bibitem{ZHU2024113436}
Zaiping Zhu, Shuangbu Wang, Lihua You, and Jianjun Zhang.
\newblock Parametric surface reconstruction from 3d point data using partial differential equation and bilinearly blended coons patch.
\newblock {\em Journal of Computational Physics}, 519:113436, 2024.

\end{thebibliography}
\end{document}